\documentclass[reqno,10pt]{amsart}
\usepackage{amsmath,mathrsfs}
\usepackage{amssymb, amsmath}
\usepackage{graphicx}
\usepackage{pdfsync}

\def\inte#1{
\displaystyle\mathop{#1\kern0pt}^\circ }

\let\pa=\partial

\let\f=\frac

\def\pa{\partial}

\def\virgp{\raise 2pt\hbox{,}}
\def\cdotpv{\raise 2pt\hbox{;}}

\def\C{\mathop{\mathbb C\kern 0pt}\nolimits}
\def\DD{\mathop{\mathbb D\kern 0pt}\nolimits}
\def\EE{\mathop{{\mathbb E \kern 0pt}}\nolimits}
\def\K{\mathop{\mathbb K\kern 0pt}\nolimits}
\def\N{\mathop{\mathbb N\kern 0pt}\nolimits}
\def\Q{\mathop{\mathbb Q\kern 0pt}\nolimits}
\def\R{\mathop{\mathbb R\kern 0pt}\nolimits}
\def\SS{\mathop{\mathbb S\kern 0pt}\nolimits}
\def\ZZ{\mathop{\mathbb Z\kern 0pt}\nolimits}
\def\TT{\mathop{\mathbb T\kern 0pt}\nolimits}
\def\P{\mathop{\mathbb P\kern 0pt}\nolimits}

\newcommand{\Z}{{\ZZ}}

\def\na{\nabla}

\newcommand{\beq}{\begin{equation}}
\newcommand{\eeq}{\end{equation}}
\newcommand{\ben}{\begin{eqnarray}}
\newcommand{\een}{\end{eqnarray}}
\newcommand{\beno}{\begin{eqnarray*}}
\newcommand{\eeno}{\end{eqnarray*}}

\newtheorem{defi}{Definition}[section]
\newtheorem{thm}{Theorem}[section]
\newtheorem{lem}{Lemma}[section]
\newtheorem{rmk}{Remark}[section]

\newtheorem{prop}{Proposition}[section]
\renewcommand{\theequation}{\thesection.\arabic{equation}}

\begin{document}
\title[Asymptotics of the Linearized Boltzmann operator]
{Asymptotic analysis  of the linearized Boltzmann collision operator  from angular cutoff to non-cutoff
\\
ANALYSE ASYMPTOTIQUE DE L'OP\'ERATEUR DE COLLISION BOLTZMANN LIN\'EARIS\'E, DE "AVEC CUTOFF ANGULAIRE" \`A "SANS CUTOFF ANGULAIRE"
}

\author[L. -B. He and Y. -L. Zhou]{Ling-Bing He and Yu-long Zhou$^{*}$}
\address[L.-B. He]{Department of Mathematical Sciences, Tsinghua University\\
Beijing 100084,  P. R.  China.} \email{hlb@tsinghua.edu.cn}
\address[Y.-L. Zhou]{School of Mathematics, Sun Yat-Sen University, Guangzhou, 510275, P. R.  China.} \email{zhouyulong@mail.sysu.edu.cn}

\thanks{$^{*}$ Corresponding author, email: zhouyulong@mail.sysu.edu.cn}

\begin{abstract} We give quantitative estimates on the asymptotics of the linearized Boltzmann collision operator and its associated equation from angular cutoff to non-cutoff. On one hand, the results disclose the link between the hyperbolic property resulting from the Grad's cutoff assumption and the smoothing property due to the long-range interaction. On the other hand, with the help of the localization techniques in the phase space, we observe some new phenomena in the asymptotic limit process. As a consequence, we give an affirmative answer to the question that there is no jump for the property that the collision operator with cutoff does not have spectrum gap but the operator without cutoff does have for the moderate soft potentials.
\\
\indent Nous donnons des estimations quantitatives sur l'asymptotique de l'op\'erateur de collision de Boltzmann lin\'earis\'e et son \'equation associ\'ee, lorsque l'on passe d'un noyau avec cutoff angulaire à un noyau sans cutoff angulaire. D'une part, les r\'esultats r\'ev\`elent le lien entre la propri\'et\'e hyperbolique r\'esultant de l'hypoth\`ese de cutoff de Grad et la propri\'et\'e de lissage due \`a l'interaction longue distance. D'autre part, \`a l'aide des techniques de localisation dans l'espace des phases, nous observons de nouveaux ph\'enom\`enes dans le processus limite asymptotique. En cons\'equence, nous donnons une r\'eponse \`a la question suivante: Y a t-il un saut pour la propri\'t\'e que l'op\'erateur de collision avec cutoff n'a pas de trou spectral, mais que l'op\'erateur sans cutoff en a un, dans le cas des potentiels mod\'er\'ement doux.
\end{abstract}

\maketitle

\tableofcontents





\renewcommand{\theequation}{\thesection.\arabic{equation}}
\setcounter{equation}{0}

\section{Introduction}

 Let  $\mathcal{L}^\epsilon$ and $\mathcal{L}^0$ be  linearized Boltzmann collision operators with and without angular cutoff respectively.   The   present work aims at   quantitative estiamtes for the asymptotic behavior of  the operator $\mathcal{L}^\epsilon$ and its associated equation from angular cutoff to non-cutoff, which corresponds to the limit that $\epsilon$ goes to zero.  Our main motivation comes from the facts that the following properties of the collision operator are totally changed in the limit process:

\begin{enumerate}
\item[(1)] For fixed $\epsilon>0$,   $\mathcal{L}^\epsilon$ behaves like a damping term for the Boltzmann equation with angular cutoff while $\mathcal{L}^0$ behaves like a fractional Laplace operator   for the equation without cutoff.

\item[(2)] For moderate soft potentials($\gamma\in [-2s,0)$), the operator $\mathcal{L}^\epsilon$ has no spectral gap for  fixed $\epsilon>0$ but the limiting point $\mathcal{L}^0$ of $\{\mathcal{L}^\epsilon\}_{\epsilon>0}$  does have.
\end{enumerate}
Another motivation arises from the approximation problem for the Boltzmann equation. It is of great importance to find out the asymptotic formula to describe the   limit  for the nonlinear equation.

\subsection{Boltzmann operator and its linearized version} We first introduce our basic assumptions and definitions on the Boltzmann collision operator and its associated equation.

\subsubsection{Boltzmann collision operator}
The Boltzmann collision operator $Q$ is a bilinear
operator acting only on the velocity variables $v$, which is defined by,
\beno Q(g,h)(v):=
\int_{\R^3}\int_{\mathbb{S}^{2}}B(v-v_*,\sigma)(g'_*h'-g_*h) dv_* d\sigma.
\eeno Here we use the usual shorthand $h=h(v)$, $g_*=g(v_*)$,
$h'=h(v')$, $g'_*=g(v'_*)$ where $v'$, $v_*'$ are given by
\beno
v'=\frac{v+v_{*}}{2}+\frac{|v-v_{*}|}{2}\sigma, \quad
v'_{*}=\frac{v+v_{*}}{2}-\frac{|v-v_{*}|}{2}\sigma, \quad 
\sigma\in\mathbb{S}^{2}.
\eeno

The nonnegative function $B(v-v_*,\sigma)$  in the collision
operator is called the Boltzmann collision kernel. It is always
  assumed to depend only on $|v-v_{*}|$ and $\frac{v-v_{*}}{|v-v_{*}|} \cdot \sigma
  $. It is convenient to introduce the angle variable $\theta$ through
  $\cos\theta= \frac{v-v_{*}}{|v-v_{*}|}\cdot \sigma$.  Without loss of generality, we may
assume that $B(v-v_{*},\sigma)$ is supported in the set
 $0\leq\theta\leq\frac{\pi}{2}$ , i.e,
$ \cos\theta\ge0
$.

We now state some physically relevant assumptions on the collision kernel. The kernel $B(v-v_*,\sigma)$ satisfies
\begin{itemize}
\item (A-1) The cross-section $B(v-v_{*},\sigma)$ takes a product form of
\beno  B(v-v_{*},\sigma)= |v-v_{*}|^{\gamma}b(\cos\theta),\eeno
 where $-3<\gamma \leq 1$ and $b$ is a nonnegative function satisfying
  \ben  \label{angular-lower-upper-bound}
  K^{-1}\sin^{-2-2s}\frac{\theta}{2} \le
b(\cos\theta)\le K
 \sin^{-2-2s}\frac{\theta}{2},  \quad
\text{ for any }  0 <\theta \leq  \pi/2.
\een
  where $0<s<1, K \geq 1$.
  The parameters $\gamma$ and $s$ verify $\gamma + 2s > -1$.
\end{itemize}

Assumption (A-1) covers inverse power law interactions. For inverse repulsive potentials $r^{-p}, p>1$, one has
$\gamma=\frac{p-5}{p-1}$ and $s=\frac{1}{p-1}$.  Usually, $\gamma>0$, $\gamma=0$, and
$\gamma<0$ are called hard, Maxwellian, and soft
potentials respectively.

The inhomogeneous Boltzmann equation without cutoff reads:
\begin{eqnarray}\label{homb}\left\{ \begin{aligned}
&\partial _t F +  v \cdot \nabla_{x} F=Q(F,F), ~~t > 0, x \in \mathbb{T}^{3}, v \in\R^3 ;\\
&F|_{t=0} = F_{0}.
\end{aligned} \right.
\end{eqnarray}
where $F(t,x,v)\geq 0$ is the density
function of collision particles which move with velocity
$v\in\R^3$ at time $t\geq 0$, position $x \in \mathbb{T}^{3} :=[-\pi,\pi]^{3}$.

We now recall the famous Grad's cutoff assumption. Let us denote by $B^\epsilon(v-v_{*},\sigma)$  the cutoff Boltzmann collision kernel. The kernel $B^\epsilon(v-v_*,\sigma)$ verifies
\begin{itemize}
\item (A-2) The cross-section $B^\epsilon(v-v_{*},\sigma)$ takes a product form of
\ben \label{cutoff-kernel-def}
b^{\epsilon}(\cos\theta) := b(\cos\theta) (1-\phi(\sin\frac{\theta}{2}/\epsilon)), \quad  B^\epsilon(v-v_{*},\sigma)= |v-v_*|^\gamma b^\epsilon(\cos\theta),
\een
 where  $0 <\epsilon \leq \frac{\sqrt{2}}{2}$ and $\phi$ is a smooth function defined in \eqref{function-phi-psi}, which has support in $[0, 4/3]$ and equals to $1$ in $[0, 3/4]$.
\end{itemize}

Note that as $\epsilon \rightarrow 0$, pointwisely $b^{\epsilon} \rightarrow b$. For convenience, let $b^{0}:=b, B^{0}:=B$. Under assumption (A-2), $B^\epsilon(v-v_{*},\sigma)$ is supported in $\theta \gtrsim \epsilon$.
Correspondingly, the Boltzmann collision operator with parameter $\epsilon \geq 0$ and its associated equation are defined by
\beno Q^{\epsilon}(g,h)(v):=
\int_{\R^3}\int_{\mathbb{S}^{2}}B^{\epsilon}(v-v_*,\sigma)(g'_*h'-g_*h) dv_* d\sigma,
\eeno
and
\begin{equation}\label{cutoffboltzmann} \left\{ \begin{aligned}
&\partial _t F +  v \cdot \nabla_{x} F=Q^{\epsilon}(F,F), ~~t > 0, x \in \mathbb{T}^{3}, v \in\R^3 ;\\
&F|_{t=0} = F_{0}.
\end{aligned} \right.
\end{equation}
Note that when $\epsilon=0$, $Q^{0}=Q$ and equation \eqref{cutoffboltzmann} is the same as  \eqref{homb}.

We remark that the solutions  to \eqref{homb} and  \eqref{cutoffboltzmann} have the fundamental physical properties of conserving total mass, momentum and kinetic energy, that is, for all $t\ge0$,
\ben\label{conserveq}  \int_{\mathbb{T}^{3} \times \R^3} [1, v_{1}, v_{2}, v_{3}, |v|^{2}] F(t,x,v)dxdv=\int_{\mathbb{T}^{3} \times \R^3} [1, v_{1}, v_{2}, v_{3}, |v|^{2}]F_{0}(x,v) dxdv.\een
Without loss of generality,  we assume  $F_{0}(x,v)$ has the same mass, momentum and energy as those of the global
Maxwellian $\mu(v) := (2\pi)^{-\frac{3}{2}}e^{-\frac{|v|^{2}}{2}}$.
By \eqref{conserveq}, one has for any $t \geq 0$,
\ben \label{conerve-for-linearization}
\int_{\mathbb{T}^{3} \times \R^3} [1, v_{1}, v_{2}, v_{3}, |v|^{2}](F-\mu)(t,x,v)  dxdv=0.\een

\subsubsection{Linearized Boltzmann collision operator}
For the cutoff case $\epsilon>0$ or the non-cutoff case $\epsilon=0$, the operators based on $Q^{\epsilon}$ are defined by
\ben\label{DefLep}
\Gamma^{\epsilon}(g,h) := \mu^{-\frac{1}{2}} Q^{\epsilon}(\mu^{\frac{1}{2}}g,\mu^{\frac{1}{2}}h), \quad
\mathcal{L}^{\epsilon}_{1}g := -\Gamma^{\epsilon}(\mu^{\frac{1}{2}},g), \quad
\mathcal{L}^{\epsilon}_{2}g := -\Gamma^{\epsilon}(g, \mu^{\frac{1}{2}}), \quad
\mathcal{L}^{\epsilon}g := \mathcal{L}^{\epsilon}_{1}g + \mathcal{L}^{\epsilon}_{2}g.
\een
We recall that the null space $\mathcal{N}(\mathcal{L}^\epsilon)$ of  $\mathcal{L}^\epsilon$ reads
 \beno \mathcal{N}(\mathcal{L}^\epsilon)=\mathcal{N}:= \mathrm{span}\{\mu^{\frac{1}{2}}, \mu^{\frac{1}{2}}v_1, \mu^{\frac{1}{2}}v_2,\mu^{\frac{1}{2}}v_3, \mu^{\frac{1}{2}}|v|^2 \}. \eeno

If we set  $F=\mu +\mu^{\frac{1}{2}}f$, then
    \eqref{cutoffboltzmann}  and   \eqref{homb} reduce to
\begin{equation}\label{linearizedBE} \left\{ \begin{aligned}
&\partial_{t}f + v\cdot \nabla_{x} f + \mathcal{L}^{\epsilon}f= \Gamma^{\epsilon}(f,f), ~~t > 0, x \in \mathbb{T}^{3}, v \in\R^3 ;\\
&f|_{t=0} = f_{0},
\end{aligned} \right.
\end{equation}
and
\begin{equation}\label{linearizedNBE} \left\{ \begin{aligned}
&\partial_{t}f + v\cdot \nabla_{x} f + \mathcal{L}^{0}f= \Gamma^{0}(f,f), ~~t > 0, x \in \mathbb{T}^{3}, v \in\R^3 ;\\
&f|_{t=0} = f_{0},
\end{aligned} \right.
\end{equation}
where $f_0 = \mu^{-\frac{1}{2}}(F_{0}-\mu)$ verifies (thanks to \eqref{conerve-for-linearization})
\ben\label{Nuspace}  \int_{\mathbb{T}^{3} \times \R^3} [1, v_{1}, v_{2}, v_{3}, |v|^{2}]\mu^{\frac{1}{2}}(v)f_0(x,v) dxdv=0. \een

\subsection{Problems and difficulties}
The main purpose of the paper is to understand what  happens to the collision operator $\mathcal{L}^\epsilon $ and its associate equation \eqref{linearizedBE} in the limit that $\epsilon $ goes to zero. More concretely, we are concerned with the following three problems.

 {\bf Problem 1:}  What is the behavior change of the  operator $  \mathcal{L}^\epsilon  $ in the limit process?

 We recall that $\mathcal{L}^\epsilon$ behaves like a damping term for equation \eqref{linearizedBE} while $\mathcal{L}^0$ behaves like a fractional Laplace operator   for  equation \eqref{linearizedNBE}. The  motivation of {\bf (P-1)} is  to see clearly which kind of link between  these two different properties in the limit process.
 Obviously it is a fundamental problem and full of challenge.

To explain the main difficulty of the problem, we focus on the Maxwellian molecules ($\gamma=0$), which is simpler than the other cases.  Previous works \cite{amuxy4,amuxy5,gs1,gs2,he2}  show that
for $\gamma=0$, there holds
\ben\label{eqiL01}  \langle \mathcal{L}^0f,f\rangle_v+|f|_{L^2}^2\sim |f|^{2}_{L^{2}_{s}}+|f|^{2}_{H^{s}} + |(-\Delta_{\mathbb{S}^{2}})^{\frac{s}{2}}f|^{2}_{L^{2}}. \een
Here $\langle f,g \rangle_v$ denotes the inner product for $v$ variable.
On the right-hand side of the equivalence \eqref{eqiL01},
there are three parts  which correspond to gain of weight $|f|^{2}_{L^{2}_{s}}$, gain of Sobolev regularity $|f|^{2}_{H^{s}} $ and gain of tangential derivative on sphere $|(-\Delta_{\mathbb{S}^{2}})^{\frac{s}{2}}f|^{2}_{L^{2}}$ respectively. Observe that \eqref{eqiL01} can be rewritten by
\ben\label{eqiL02}
  \langle \mathcal{L}^0f,f\rangle_v+|f|_{L^2}^2\sim |W_{s}f|^{2}_{L^{2}}+|W_{s}(D)f|^{2}_{L^2} + |W_{s}((-\Delta_{\mathbb{S}^{2}})^{\frac{1}{2}})f|^{2}_{L^{2}},
\een
where $W_{s}(x) :=(1+|x|^2)^{\f{s}2}$. From now on, we call $W_{s}$ the characteristic function of $\mathcal{L}^0$ since it captures full structure of $\mathcal{L}^0$. The operator $W_{s}((-\Delta_{\mathbb{S}^{2}})^{\frac{1}{2}})$ is defined in \eqref{DeltaWe}.

Considering $ \langle \mathcal{L}^\epsilon f,f\rangle_v\rightarrow  \langle \mathcal{L}^0f,f\rangle_v$ as $\epsilon \rightarrow 0$,  we guess that $\langle \mathcal{L}^\epsilon f, f\rangle_v$ has the same structure as the right-hand side of \eqref{eqiL02}. If so, what is the characteristic function of $\mathcal{L}^\epsilon$ when $\epsilon>0$? To find out a good candidate, we go back to the original proof of the coercivity estimate for the collision operator in \cite{advw}. Following the computation used there, we can derive that
\ben\label{sobolevregu}  - \langle Q^\epsilon(g, f), f\rangle_v+|f|_{L^2}^2\ge C_g|W^\epsilon(D)f|_{L^2}^2, \een
where $W^\epsilon$  is  defined by
\ben\label{charicter function}
W^{\epsilon}(v) = W_{s}\phi(\epsilon v) + \epsilon^{-s}(1-\phi(\epsilon v)).
\een
Here $\phi \in C_{0}^{\infty}(B_{\frac{4}{3}})$ is a smooth compactly-supported  function satisfying \eqref{function-phi-psi}. Note that as $\epsilon \rightarrow 0$, at least pointwisely $W^{\epsilon} \rightarrow W_{s}$. For convenience, let $W^{0}:=W_{s}$.
We conjecture that $W^\epsilon$ is the characteristic function of $\mathcal{L}^\epsilon$ in the following sense:
\ben\label{conject1}\langle \mathcal{L}^\epsilon f, f\rangle_v + |f|^{2}_{L^{2}}  \sim  |W^\epsilon f|^{2}_{L^{2}}+|W^\epsilon(D)f|^{2}_{L^2} + |W^\epsilon((-\Delta_{\mathbb{S}^{2}})^{\frac{1}{2}})f|^{2}_{L^{2}}. \een
The operator $W^{\epsilon}((-\Delta_{\mathbb{S}^{2}})^{\frac{1}{2}})$ is defined in \eqref{DeltaWe}.

Let us give some comments on conjecture \eqref{conject1}. Firstly it is easy to see that  when $\epsilon$ goes to zero, \eqref{conject1} will coincide with \eqref{eqiL02}. This shows that the characteristic function $W^\epsilon$ connects the cutoff case and the non-cutoff case. Secondly on the right-hand of \eqref{conject1},   gain of weight only happens in the region $|v|\lesssim 1/\epsilon$ in the phase space, gain of Sobolev regularity only happens in the region $|\xi|\lesssim 1/\epsilon$ in the frequency space and gain of tangential derivative only happens in the region that the eigenvalue $\lambda$ of the operator $(-\Delta_{\mathbb{S}^{2}})^{\frac{1}{2}}$ verifies $\lambda\lesssim 1/\epsilon$.  These properties
 are consistent with the fact that the operator $\mathcal{L}^\epsilon$ has a hyperbolic structure due to the angular cutoff, that is, $\theta \gtrsim \epsilon$.
Thirdly, because of hyperbolic structure of $\mathcal{L}^\epsilon$, it is unclear how to derive
     $|W^\epsilon((-\Delta_{\mathbb{S}^{2}})^{\frac{1}{2}})f|_{L^2}$ and  $|W^\epsilon f|_{L^{2}}$  in \eqref{conject1} using
the methods in the previous works \cite{amuxy4,amuxy5,gs1,gs2,afl,lmkx,he2,pao}.
Therefore we need some new idea to prove the conjecture.

 {\bf Problem 2:} What is the longtime behavior of  $e^{-\mathcal{L}^\epsilon t}f$ with $f\in \mathcal{N}^{\perp}$ for moderate soft potentials in the limit process that $\epsilon$ goes to zero? Here $e^{-\mathcal{L}^\epsilon t}$ is the semi-group generated by
$\mathcal{L}^\epsilon$.

   As we know, for $\gamma\in [-2s,0)$, the operator $\mathcal{L}^\epsilon$ has no spectral gap for any  fixed $\epsilon>0$ but the limiting point $\mathcal{L}^0$ of $\{\mathcal{L}^\epsilon\}_{\epsilon>0}$  does have. It seems that there is a jump. Instead of investigating the spectrum of the operator which looks extremely difficult, we turn to consider the longtime behavior of   $e^{-\mathcal{L}^\epsilon t}f$ because the spectrum information of an operator has strong connection with the corresponding semi-group.

   Thanks to spectral gap of $\mathcal{L}^0$, it is easy to see that for any $f\in {\mathcal{N}}^{\perp}$,
\beno \|e^{-\mathcal{L}^{0} t}f\|_{L^2}\le e^{-c t}\|f\|_{L^2}. \eeno
  As  for the operator $\mathcal{L}^\epsilon$, by imposing the additional assumption that $f\in L^2_l$, we can derive that  $e^{-\mathcal{L}^{\epsilon}t}f$ will decay to zero with polynomial rate. However we have no idea on the explicit rate of this relaxation for $f\in {\mathcal{N}}^{\perp}$ if we only impose $f\in L^2$.  By approximation argument, we only can prove that
\beno  \lim_{t\rightarrow\infty}\|e^{-\mathcal{L}^{\epsilon}t}f\|_{L^2}=0. \eeno
      Therefore from these two estimates, it is difficult to find out the link between these two different longtime behaviors. We comment that this difficulty matches the facts that $\mathcal{L}^\epsilon$ does not have spectral gap but $\mathcal{L}^0$ does have.

 {\bf Problem 3:}  Which kind of asymptotic formula describes the limit that $\epsilon$ goes to zero for the solutions of the nonlinear equations \eqref{linearizedBE} and \eqref{linearizedNBE}?

Formally when $\epsilon$ goes to zero, the solution $f^\epsilon$ to \eqref{linearizedBE} will converge to the solution $f^{0}$ to \eqref{linearizedNBE}. To answer {\bf (P-3)} is to justify the convergence and find an asymptotic formula to describe the limit.

 To guess the relation between $f^\epsilon$ and $f^{0}$, we first have a look at the stationary case.
By Taylor expansion, we can prove that for any smooth compactly supported functions $f$,
 \ben \label{stationary} |Q^\epsilon(f,f)-Q^{0}(f,f)|\sim O(\epsilon^{2-2s}). \een
Thus it is natural to conjecture
\ben \label{solution-error-order} f^\epsilon-f^{0}=O(\epsilon^{2-2s}).\een
Obviously the main difficulty in establishing \eqref{solution-error-order} lies in bringing the error order \eqref{stationary} from operator level to solution level. To this end,  we need some uniform (with respect to $\epsilon$) estimates of  $\mathcal{L}^\epsilon$ and $\Gamma^{\epsilon}$, and also estimates of the difference $\Gamma^{\epsilon}-\Gamma^{0}$.

\subsection{Notations}
We list function spaces and notations
which will be used throughout the paper.

\subsubsection{Basic notations}
  We denote multi-index by $\alpha =(\alpha^1,\alpha^2,\alpha^3) \in \mathbb{Z}_{+}^{3}$ with
$|\alpha |=\alpha^1+\alpha^2+\alpha^3$. We write $a\lesssim b$ to indicate that  there is a
universal constant $C,$ which may be different on different lines,
such that $a\leq Cb$.  We use the notation $a\sim b$ whenever $a\lesssim b$ and $b\lesssim
a$.  The notation $[a]$ denotes the maximum integer which does not exceed $a$.  The bracket $\langle \cdot\rangle$ is defined by $\langle v\rangle :=(1+|v|^2)^{\frac{1}{2}}$. Then the weight function $W_l$ is defined by $W_l(v):= \langle v\rangle^l $.  We denote $C(\lambda_1,\lambda_2,\cdots, \lambda_n)$ or $C_{\lambda_1,\lambda_2,\cdots, \lambda_n}$ by a constant depending on   parameters $\lambda_1,\lambda_2,\cdots, \lambda_n$. The notations  $\langle f,g\rangle_v:= \int_{\R^3}f(v)g(v)dv$ and $(f,g):= \int_{\mathbb{T}^3 \times \R^3} f(x,v)g(x,v)dxdv$ are used to denote the inner products for $v$ variable and for $x,v$ variables respectively.
As usual, $\mathrm{1}_A$ is the characteristic function of a set $A$. If $A,B$ are two operators, then their commutator $[A,B]:= AB-BA$. Recall $|f|_{L \mathrm{log}L}:=\int_{\R^3} |f(v)|\log(1+|f(v)|) d v$.

\subsubsection{Function spaces} Several spaces are introduced as follows:

(1). For   real number $n, l $, we define the weighted Sobolev space on $\R^3$
\begin{equation*}
H^{n}_l :=\bigg\{f(v)\big| |f|^2_{H^n_l}:=\int_{\R^3} |(W_{n}(D) W_{l} f)(v)|^2 dv
 <\infty\bigg\}.
\end{equation*} For any symbol $a: \R^3 \rightarrow \R$, recall that
 $a(D)$ is the pseudo-differential operator defined by
\beno  \big(a(D)f\big)(v):=\f1{(2\pi)^3}\int_{\R^3}\int_{\R^3} e^{i(v-y)\cdot\xi}a(\xi)f(y)dyd\xi.\eeno

(2). We  introduce the $L^p_l$ space as
\beno
L^p_l:= \bigg\{f(v)\big| |f|_{L^p_l}:=\bigg(\int_{\R^3} |f(v)|^p\langle v
\rangle^{l p}dv\bigg)^{\f1{p}}<\infty \bigg\}.
\eeno

(3). For $m\in\N$, we denote the Sobolev space on $\mathbb{T}^{3}$ by
\begin{equation*} H^{m}_{x} :=\bigg\{f(x)\big| |f|^{2}_{H^{m}_{x}}:= \sum_{|\alpha | \leq m}
\int_{\mathbb{T}^3}|\partial^{\alpha}_{x} f|^{2}dx <\infty\bigg\}.
\end{equation*}

(4). For a function $f(x,v)$, we define the following weighted Sobolev  spaces
with weight on velocity variable $v$. For $m,n \in \N, l \in \R$, the weighted (in $v$) Sobolev space on $\mathbb{T}^{3}\times\R^{3}$ is defined by
\beno H^{m}_xH^{n}_{l} := \bigg\{f(x,v) \big| \|f\|^{2}_{ H^{m}_xH^{n}_{l}} :=
\sum_{|\alpha| \leq m, |\beta| \leq n}  \int_{\mathbb{T}^3}| \partial^{\alpha}_x\pa^{\beta}_v f|^{2}_{L^{2}_l}dx < \infty \bigg\}.\eeno
 For simplicity,   we write $\|f\|_{H^{m}_{x}L^{2}_{l}} := \|f\|_{ H^{m}_xH^{0}_{l} }$ if $n=0$ and   $\|f\|_{L^{2}_{l}} := \|f\|_{ H^{0}_xH^{0}_{l} }$ if $m=n=0$.  We can define the homogeneous space $\dot{H}^{m}_x\dot{H}^{n}_{l}$ if we replace by $|\alpha| \leq m, |\beta| \leq n$ by $|\alpha| = m, |\beta| = n$. Similarly we can introduce the partial homogeneous space  $\dot{H}^{m}_xH^{n}_{l}$ and $H^{m}_x\dot{H}^{n}_{l}$.

\subsubsection{Dyadic decompositions} We now recall dyadic
decomposition. Let $B_{\frac{4}{3}} := \{v\in\R^{3}: |v| \leq \frac{4}{3}\}$ and $C := \{v\in\R^{3}: \frac{3}{4} \leq |v| \leq \frac{8}{3}\}$.  Then one may introduce two
radial functions $\phi \in C_{0}^{\infty}(B_{\frac{4}{3}})$ and $\psi \in C_{0}^{\infty}(C)$ which satisfy
\ben \label{function-phi-psi} 0\leq \phi, \psi \leq 1, \text{ and } \phi(v) + \sum_{j \geq 0} \psi(2^{-j}v) =1, \text{ for all } v \in \R^{3}. \een
Since $\phi$ is a radical function, we can interchangeably use $\phi(v)$ and $\phi(|v|)$.
Now define $\varphi_{-1}(v) :=  \phi(v)$ and $\varphi_{j}(v) :=  \psi(2^{-j}v)$ for any $v \in \R^{3}$ and $j \geq 0$.
Let $(\mathcal{P}_j f)(v):= \varphi_{j}(v)f(v)$, then one has the following dyadic decomposition
\beno  f =\sum_{j=-1}^\infty \mathcal{P}_j f = \sum_{j=-1}^{\infty} \varphi_{j}f, \eeno
for any function $f$ defined on $\R^{3}$.
We will  use the notations \ben \label{defphilh} f_\phi:= \phi(\epsilon D) f,\quad f^\phi:=(1-\phi(\epsilon D))f,\quad f^l=\phi(\epsilon \cdot) f,\quad f^h=(1-\phi(\epsilon\cdot))f.\een

\subsubsection{Micro-Macro decomposition} Recalling $\mathcal{N}= \mathrm{span}\{\mu^{\frac{1}{2}}, \mu^{\frac{1}{2}}v_1, \mu^{\frac{1}{2}}v_2,\mu^{\frac{1}{2}}v_3, \mu^{\frac{1}{2}}|v|^2 \}$, we introduce the projection operator $\mathbb{P}$ on $\mathcal{N}$ as follows:
\ben\label{DefProj} \mathbb{P}f=(a+b\cdot v+c|v|^2)\mu^{\frac{1}{2}}, \een
where for $1\le i\le 3$, \beno
 a=\int_{\R^3} (\frac{5}{2}-\frac{|v|^{2}}{2})\mu^{\frac{1}{2}}fdv, \quad b_i=\int_{\R^3} v_i\mu^{\frac{1}{2}}fdv, \quad c=\int_{\R^3} (\frac{|v|^2}{6}-\frac{1}{2})\mu^{\frac{1}{2}}fdv.
\eeno

\subsubsection{Anisotropic function spaces}  Let $Y_l^m$ with $-l\le m\le l$ be real spherical harmonics verifying that
$ (-\triangle_{\mathbb{S}^2})Y_l^m=l(l+1)Y_l^m. $
Then the operator $W^{\epsilon}((-\Delta_{\mathbb{S}^{2}})^{\frac{1}{2}})$ for $\epsilon \geq 0$ is defined by: if $v = r \sigma$, then
\ben\label{DeltaWe} (W^{\epsilon}((-\Delta_{\mathbb{S}^{2}})^{\frac{1}{2}})f)(v) := \sum_{l=0}^\infty\sum_{m=-l}^{l} W^\epsilon((l(l+1))^{\frac{1}{2}}) Y^{m}_{l}(\sigma)f^{m}_{l}(r),
 \een
where
$ f^{m}_{l}(r) = \int_{\mathbb{S}^{2}} Y^{m}_{l}(\sigma) f(r \sigma) d\sigma$. We recall that when $\epsilon>0$,  $W^\epsilon$ is defined in \eqref{charicter function}. When $\epsilon=0$,  $W^{0}(v) = \langle v \rangle^{s}$.

Now we  introduce several anisotropic function spaces induced by $\mathcal{L}^\epsilon$.

(1). {\it The space $L^2_{\epsilon,l}$ with $l\in\R$.} For functions on $\R^3$,
the space $L^2_{\epsilon,l}$ is defined by
\ben \label{new-norm}
L^2_{\epsilon,l}:=\bigg\{f(v)\big||f|^{2}_{\epsilon,l}:= |W^{\epsilon}((-\Delta_{\mathbb{S}^{2}})^{\frac{1}{2}})W_{l}f|^{2}_{L^{2}} + |W^{\epsilon}(D)W_{l}f|^{2}_{L^{2}} + |W^{\epsilon}W_{l}f|^{2}_{L^{2}}<\infty\bigg\}.
\een

 (2). {\it The  space  $H^{m}_xH^{n}_{\epsilon,l}$ with $m,n\in \N, l\in \R .$} For functions on $\mathbb{T}^{3}\times\R^{3}$, the space $H^{m}_xH^{n}_{\epsilon,l}$  is defined  by
\beno  H^{m}_xH^{n}_{\epsilon,l}:=\bigg\{f(x,v)\big|
\|f\|^{2}_{H^{m}_xH^{n}_{\epsilon,l}} := \sum_{|\alpha| \leq m, |\beta| \leq n}\int_{\mathbb{T}^{3}} |\pa^\alpha_x\pa^\beta_vf(x,\cdot)|^{2}_{\epsilon,l} dx<\infty\bigg\}.
\eeno
 For simplicity,   we set $\|f\|_{H_x^{m}L^2_{\epsilon,l}}:= \|f\|_{H^{m}_xH^{0}_{\epsilon,l}}$ if $n=0$ and $\|f\|_{L^2_{\epsilon,l}}:= \|f\|_{H^{0}_xH^{0}_{\epsilon,l}}$ if $m=n=0$. Similarly we can introduce the   spaces $\dot{H}^{m}_x\dot{H}^{n}_{\epsilon,l}$  $\dot{H}^{m}_xH^{n}_{\epsilon,l}$ and $H^{m}_x\dot{H}^{n}_{\epsilon,l}$.

(3). {\it Functionals related to $\mathcal{L}^\epsilon$.} We introduce
\ben \label{defintion-R-g}
\mathcal{R}_g^{\epsilon,\gamma}(f) := \int_{\mathbb{S}^2 \times \mathbb{R}^6} b^{\epsilon}(\cos\theta)| v-v_{*} |^{\gamma} g_{*} (f^{\prime}-f)^{2} d\sigma dv dv_{*},
\\ \label{defintion-R-star-g}
\mathcal{R}^{\epsilon,\gamma}_{*,g}(f) := \int_{\mathbb{S}^2 \times \mathbb{R}^6} b^{\epsilon}(\cos\theta)\langle v-v_{*} \rangle^{\gamma} g_{*}(f^{\prime}-f)^{2}d\sigma dv dv_{*},
\\ \label{defintion-M}
\mathcal{M}^{\epsilon,\gamma}(f):=\int_{\mathbb{S}^2 \times \mathbb{R}^6} b^{\epsilon}(\cos\theta)| v-v_{*} |^{\gamma}  f_{*}^{2} ((\mu^{\frac{1}{2}})^{\prime}-\mu^{\frac{1}{2}})^{2} d\sigma dv dv_{*}.
\een
As we will show in Section 2, \beno \langle \mathcal{L}^{\epsilon} f,f \rangle_v+|f|^{2}_{L^2_{\gamma/2}} \gtrsim \mathcal{R}_\mu^{\epsilon,\gamma}(f)+\mathcal{M}^{\epsilon,\gamma}(f).\eeno
The quantities $\mathcal{R}_g^{\epsilon,\gamma}(f) $ and $\mathcal{M}^{\epsilon,\gamma}(f)$ correspond to gain of regularity and gain of weight respectively. In contrast to $\mathcal{R}^{\epsilon,\gamma}_{g}(f)$, when $\gamma<0$, $\mathcal{R}^{\epsilon,\gamma}_{*,g}(f)$ contains no singularity in the relative velocity $v-v_*$ near origin.

\subsection{Main results}
We are in a position to state our main results. The first one is a uniform coercivity estimate of $\mathcal{L}^{\epsilon}$, which fully solves {\bf (P-1)}.
\begin{thm}\label{main1}
There exists a constant $\epsilon_0>0$ such that for $0 \leq \epsilon \le \epsilon_0$ and  any suitable function $f$,
\begin{eqnarray}
\langle \mathcal{L}^{\epsilon}f, f\rangle_v + |f|^{2}_{L^{2}_{\gamma/2}} \sim |f|^{2}_{\epsilon,\gamma/2}. \label{uniforml2}
\end{eqnarray}
Here the norm $|\cdot|_{\epsilon,\gamma/2}$ is defined in \eqref{new-norm}.
\end{thm}
Some remarks are in order:
\begin{rmk} Through the characteristic function $W^\epsilon$,  the coercivity estimate \eqref{uniforml2} discloses the link
between the hyperbolic structure due to the cutoff assumption ($\epsilon>0$) and the smoothing property due to the long-range interaction ($\epsilon=0$).
\end{rmk}

\begin{rmk}  Recall $f^l $ and $f^h$ in \eqref{defphilh}, then
\beno
|W^\epsilon f|^2_{L^2_{\gamma/2}}\sim | f^l|^2_{L^2_{\gamma/2+s}}+\epsilon^{-2s}|f^h|^2_{L^2_{\gamma/2}}. \eeno
Let us focus on moderate soft potentials, that is, $\gamma\in [-2s,0)$. Obviously
in the region $|v|\lesssim 1/\epsilon$ of the phase space, the operator $\mathcal{L}^\epsilon$ produces some weight since $\gamma+2s\geq0$. While in the region $|v|\gtrsim 1/\epsilon$, the operator $\mathcal{L}^\epsilon$ loses some weight since $\gamma<0$. This observation is consistent with the fact that $\mathcal{L}^\epsilon$ has no spectral gap for any fixed $\epsilon>0$ but $\mathcal{L}^0$ does have.
\end{rmk}

  Our second result is on the diversity of the longtime behavior of  $e^{-\mathcal{L}^\epsilon t}f_{0}$ with $f_{0} \in \mathcal{N}^\perp$ for moderately soft potentials, which solves {\bf (P-2)}.
\begin{thm}\label{main2} Suppose $0<\epsilon\le\epsilon_0, -2s \leq \gamma <0$
and $f_0\in   \mathcal{N}^{\perp}$. There is a universal constant $c>0$ such that
\ben\label{semigroupLe1}  |e^{-\mathcal{L}^\epsilon t}f_0|_{L^2}^2
\lesssim e^{-ct}|f_0^l|_{L^2}^2+|f_0^h|_{L^2}^2+\epsilon^{2s}|f_0|_{L^2}^2, \een
where $f^l $ and $f^h$ are defined in \eqref{defphilh}. Furthermore,
\begin{enumerate}
\item Let $p>0$ and $ -\gamma p/2\ge2$.
 Suppose $f_0\in L^2_{-\gamma p/2}$. Set $\mathcal{C}_I:=|e^{-\mathcal{L}^\epsilon t_1}f_0|_{L^2}^{-\f2p}|f_0|_{L^2_{-\gamma p/2}}^{\f2p}\epsilon^{2s}$ for some $t_1\ge0$. For any $t\ge 0$,
\begin{enumerate} \item
if $\mathcal{C}_I\ll 1$, then there exists a universal constant  $c_{1}>0$ and a critical time $t_*=O(p(-\ln \mathcal{C}_I ))$ such that
\ben\label{semigroupLe2} |e^{-\mathcal{L}^\epsilon (t+t_1)}f_0|_{L^2}^2
\lesssim |e^{-\mathcal{L}^\epsilon t_1}f_0|_{L^2}^2\big(e^{-c_{1}t}\mathrm{1}_{t\le t_* }+C(p)\mathcal{C}_I^p (1+t)^{-p}\mathrm{1}_{t\ge t_*}\big);\een
\item
 if $\mathcal{C}_I\sim 1$, then
\ben\label{semigroupLe3} |e^{-\mathcal{L}^\epsilon (t+t_1)}f_0|_{L^2}^2
\lesssim  C(p)|f_0|_{L^2_{-\gamma p/2}}^2\epsilon^{2sp}(1+t)^{-p}.\een
\end{enumerate}
\item Let $f_0$ additionally verify that $|f_0|^2_{L^2}=1$ and $|\mathcal{P}_{j}f_0|_{L^2}^2= 1-\eta$ with $\eta$ sufficiently small and $2^{j\gamma}\ll  \epsilon^{2s}$(which implies that $1/\epsilon\ll 2^j$). Then for $t\in [0, C^{-1}\eta2^{-j\gamma}\epsilon^{2s}]$ where $C$ is a universal constant,
\ben\label{semigroupLe4} |e^{-\mathcal{L}^\epsilon t}f_0|_{L^2}^2\ge  |\mathcal{P}_je^{-\mathcal{L}^\epsilon t}f_0|_{L^2}^2\ge 1-2\eta-C\epsilon^{2s}. \een
\end{enumerate} As a consequence, for any fixed sufficiently small $\epsilon>0$, the estimate $\lim\limits_{t\rightarrow\infty}|e^{-\mathcal{L}^\epsilon t}f_0|_{L^2}=0$ is sharp.

\end{thm}

Some remarks are in order:
\begin{rmk} We have three comments on estimate \eqref{semigroupLe1}. Firstly it shows that the longtime behavior of $e^{-\mathcal{L}^\epsilon t}f_0$ depends heavily on energy distribution of  $f_0$. Secondly the estimate is sharp for general data $f_0\in  \mathcal{N}^{\perp}$ thanks to the estimates \eqref{semigroupLe2}  and \eqref{semigroupLe4}, which deal with the case that the energy of $f_0$ is concentrated in   the ball $B_{1/\epsilon}$ and  the case that the energy of $f_0$ is concentrated far away from  the ball $B_{1/\epsilon}$.
Thirdly, by passing the limit $\epsilon\rightarrow0$, we recover from \eqref{semigroupLe1} that   for all $t\ge0$,
\ben\label{semigroupL0} |e^{-\mathcal{L}^0 t}f_0|_{L^2}^2
\lesssim e^{-ct}|f_0|_{L^2}^2.\een This demonstrates that there is no jump for the facts that the operator $\mathcal{L}^\epsilon$ has no spectral gap for  fixed $\epsilon>0$ but $\mathcal{L}^0$  does have.
 \end{rmk}

\begin{rmk} Estimates \eqref{semigroupLe2} and \eqref{semigroupL0}  show that up to a  critical time $t_*=O(|\ln \epsilon|)$, in terms of decay pattern,
there is no difference between $e^{-\mathcal{L}^\epsilon t}f_0$ and  $e^{-\mathcal{L}^0 t}f_0$. The  difference appears only after the critical time $t_*$. In fact, after $t_*$  the hyperbolic structure will take over the behavior of the semi-group $ e^{-\mathcal{L}^\epsilon t}$, which corresponds to the polynomial decay in \eqref{semigroupLe2}. To our best knowledge, this phenomenon is observed for the first time.   \end{rmk}

\begin{rmk} We have two remarks on \eqref{semigroupLe4}. Firstly the total energy of $f_0$ can be almost conserved in $e^{-\mathcal{L}^{\epsilon}t}f_0$ in any given time interval if the datum $f_{0}$ is suitably chosen.  Such kind of datum prevents the formation of spectral gap for  $\mathcal{L}^\epsilon$  no matter how small $\epsilon>0$ is. Secondly we want to show there are extensive data $f_0$ verifying all the assumptions in $(ii)$.  Taking a arbitrary function $f \in L^2$ with $|f|_{L^2}=1$ and the support of $f$ belonging to the ring $\{v\in \R^3| \frac{4}{3} \times 2^j\le |v|\le \frac{3}{2} \times 2^j\}$. Let   $f_0=f-\mathbb{P}f$. Then $f_0$ verifies $f_0\in \mathcal{N}^\perp$,
$|\mathcal{P}_j f_0|_{L^2} \ge 1-O(e^{-\f18 \times 2^{2j}})$. Then $f_0 / |f_{0}|_{L^2}$ fulfills all the assumptions.
\end{rmk}

\begin{rmk} Sharpness of the estimate, $\lim\limits_{t\rightarrow\infty}|e^{-\mathcal{L}^\epsilon t}f_0|_{L^2}=0$, directly follows   \eqref{semigroupLe2}-\eqref{semigroupLe4}. Indeed, the estimate can be derived thanks to \eqref{semigroupLe2} and \eqref{semigroupLe3}.  On the other hand, due to \eqref{semigroupLe4}, it is impossible to get an explicit and uniform decay rate for the above relaxation. These two facts reveal the diversity of the longtime behavior of   $e^{-\mathcal{L}^\epsilon t}f_0$. Our results are comparable to the results for the homogeneous Boltzmann equation with moderately soft potentials.  As is shown in \cite{cclu}, the rate of convergence to equilibrium can be very slow if we only assume that a solution conserves mass, momentum and energy.
\end{rmk}

\begin{rmk} Let us comment on the connection between the constant $c_{1}$ in  \eqref{semigroupLe2} and  the spectral gap $\lambda$ of the operator $\mathcal{L}^0$. Obviously $c_{1} \le \lambda$. An interesting problem is to  see the dependence of $\lambda$ on $c_{1}$. By Lemma \ref{estimate-operator-difference}, there holds $|(\mathcal{L}^\epsilon-\mathcal{L}^0)f|_{L^2}\lesssim \epsilon^{2-2s}|f|_{H^2_{\gamma+2}}$. Therefore if $f_0$ is smooth, then \eqref{semigroupLe2} can be improved to
\beno |e^{-\mathcal{L}^\epsilon t}f_0|_{L^2}^2
\lesssim |f_0|_{L^2}^2e^{-\lambda t}+\epsilon^{2-2s} |f_0|_{H^2_{\gamma+2}}^2.\eeno
\end{rmk}

Our third result is on global well-posedness, propagation of regularity and global dynamics of equation \eqref{linearizedBE}. Based on propagation of regularity, we derive an asymptotic formula for solutions to \eqref{linearizedBE} and \eqref{linearizedNBE}, which  solves  {\bf (P-3)}.

Let us introduce some weight functions.  For $J,N\in\N$ with $N \geq  1, 0 \leq J \leq N$, we  introduce a sequence of weight functions $\{W_{m,j}\}_{0 \leq m+j \le N-1}\cup \{W_{N-j,j}\}_{0\le j\le J}$ with $W_{m,j}=W_{l_{m,j}}, l_{m,j} \in \mathbb{R}$ verifying
\ben \label{AsuWf} l_{N-J,J}\ge 2, l_{m,j} \ge  l_{m-1,j+1} - \gamma,  l_{0,m-1} \ge l_{m,0}.\een
We remark that $l_{m,j}$ is the weight order for $x$ derivative of order $m$ and $v$ derivative of order $j$. Note that \eqref{AsuWf} means the weight increase by $-\gamma (\gamma<0)$ if the $v$ derivative order decrease by 1 and the $x$ derivative order increase by 1. Note that this type of weight sequence is designed to control the term $v\cdot \nabla_{x} f$ and is used in \cite{Guo1} in cutoff case.

Let $\pa^\alpha_\beta := \pa^\alpha_{x} \pa^\beta_{v}$. For $0 \leq k\leq N-1$  and $0 \leq m+j \leq N-1$ or $m+j=N, 0 \leq j \leq J$, we define energy and dissipation functional
\ben \label{pure-order-m-j}
\dot{\mathcal{E}}^{m,j}(f):=\sum_{|\alpha|= m,|\beta|=j}\|W_{m,j}\pa^\alpha_\beta f\|_{L^2}^2,  \quad \dot{\mathcal{D}}^{m,j}(f):=\sum_{|\alpha|= m,|\beta|=j}\|W_{m,j}\pa^\alpha_\beta f\|_{L^2_{\epsilon,\gamma/2}}^2,
\\ \label{total-order-k}
\dot{\mathcal{E}}^{k}(f):=\sum_{j=0}^k\dot{\mathcal{E}}^{k-j,j}(f), \quad \dot{\mathcal{D}}^{k}(f):=\sum_{j=0}^k\dot{\mathcal{D}}^{k-j,j}(f),
\\ \label{full-order-k}
\mathcal{E}^{N,J}(f):=\sum_{k=0}^{N-1}\dot{\mathcal{E}}^{k}(f)+\sum_{j=0}^J\dot{\mathcal{E}}^{N-j,j}(f), \quad
  \mathcal{D}^{N,J}(f):=\sum_{k=0}^{N-1}\dot{\mathcal{D}}^{k}(f)+\sum_{j=0}^J\dot{\mathcal{D}}^{N-j,j}(f).
 \een
Here $\dot{\mathcal{E}}^{m,j}$ contains $x$ derivative of order $m$ and $v$ derivative of order $j$. $\dot{\mathcal{E}}^{k}$ contains all mixed $x$ and $v$ derivative of total order $k$, that is, $|\alpha|+|\beta|=k$. $\mathcal{E}^{N,J}$ contains derivatives $\pa^\alpha_\beta$ either
$|\alpha|+|\beta| \leq N-1$ or $|\alpha|+|\beta| = N, |\beta| \leq J$. $\dot{\mathcal{D}}^{m,j}, \dot{\mathcal{D}}^{k}$ and $\mathcal{D}^{N,J}$ are the corresponding dissipation. If $J=N$, we simplify the notation $\mathcal{E}^{N}(f):=\mathcal{E}^{N,N}(f), \mathcal{D}^{N}(f):=\mathcal{D}^{N,N}(f)$.
The energy functional $\mathcal{E}^{N,J}$ is introduced to prove the propagation of full regularity of the solution.

We are ready to present our last main result.
\begin{thm}\label{main3}
Let $0 \leq \epsilon\le\epsilon_0$, $\gamma\in(-3/2,0)\cap[-2s,0)$. There is a constant $\delta_0>0$ independent of $\epsilon$ such that the following statements are valid. Let  $f_0$ verify \eqref{Nuspace} and $\|f_{0}\|_{H^{2}_{x}L^{2}} \leq \delta_{0}$.
\begin{enumerate}
\item{\bf (Global well-posedness)} The Cauchy problem
 \eqref{linearizedBE} (interpreted as problem \eqref{linearizedNBE} if $\epsilon=0$) admits a unique and global solution $f^\epsilon$ verifying
 $\sup_{t\ge0}\|f^\epsilon(t)\|_{H^{2}_{x}L^{2}}\lesssim \|f_{0}\|_{H^{2}_{x}L^{2}}$.

\begin{enumerate}
\item [(i)] If additionally $f_0\in H^N_xL^2_l$ with $N,l\ge2$, then
\ben \label{propagation-h-n-l-2-l} \sup_{t\ge0}\|f^\epsilon(t)\|_{H^N_xL^2_l}^{2} + \int_0^\infty \|f^\epsilon(\tau)\|_{H^N_xL^2_{\epsilon, l+\gamma/2}}^{2} d\tau
\leq C(\|f_0\|_{H^N_xL^2_l}^{2}).\een
\item [(ii)] If  additionally
 $ \mathcal{E
 }^{N,J}(f_0)<\infty$ with $N \geq  2, 0 \leq J \leq N$, then
 \ben \label{propagation-h-n-h-m-l}
 \sup_{t\ge0}\mathcal{E
 }^{N,J}(f^\epsilon(t))+\int_0^\infty \mathcal{D}^{N,J}(f^\epsilon(\tau))d\tau   \leq C(\mathcal{E
 }^{N,J}(f_0)). \een
\end{enumerate}
Here $C(\cdot)$ is a continuous increasing function verifying $C(0)=0$.
\item{\bf (Global dynamics)} There are two results.
\begin{enumerate}
\item[(i)] If $f_0\in H_x^2L^2_{-p\gamma/2}$ with $-p\gamma/2\ge2$. Set $\mathcal{C}_I:= \|f_0\|_{H^2_xL^2}^{-\f2p}
|f_0|_{H_x^2 L^2_{-\gamma p/2}}^{\f2p}\epsilon^{2s}$, then for $t\ge 0$,
\begin{enumerate} \item[(a)]
if $\mathcal{C}_I\ll 1$, then there exists a universal constant  $c>0$ and a critical time $t_*=O(p(-\ln \mathcal{C}_I ))$ such that
\ben\label{decay-uniform-formula1} \|f^\epsilon(t)\|^2_{H^2_xL^2}\lesssim \|f_0\|_{H^2_xL^2}^2\big(e^{-ct}\mathrm{1}_{t\le t_* }+C(p)\mathcal{C}_I^p(1+t)^{-p}\mathrm{1}_{t\ge t_*}\big);\een
\item[(b)]
 if $\mathcal{C}_I\sim 1$, then
\ben\label{decay-uniform-formula2} \|f^\epsilon(t)\|^2_{H^2_xL^2}
\lesssim  C(p)\|f_0\|_{H^2_x L^2_{-\gamma p/2}}^2\epsilon^{2sp}(1+t)^{-p}.\een
\end{enumerate}
\item[(ii)] Let $0<\eta<1$. If $2^{j\gamma}\ll \epsilon^{2s}$, then for $t\in [0,  C^{-1} \eta 2^{-j\gamma}\epsilon^{2s}]$ where $C$ is a universal constant, we have
\ben\label{localizedenergy} \|\mathcal{P}_jf^\epsilon(t)\|^2_{L^2}\ge \|\mathcal{P}_jf_0\|^2_{L^2}-\eta \delta_{0}-C\epsilon^{2s}\delta_{0}. \een
\end{enumerate}

\item{\bf (Global asymptotic formula for the limit process)}
If $ \mathcal{E}^{N+2,2}(f_0)<\infty$ with $N \ge 2$, then
\ben \label{error-function-uniform-estimate} \sup_{t \geq 0} \|f^{\epsilon}(t)-f^{0}(t)\|^{2}_{H^{N}_{x}L^{2}} \leq C(\mathcal{E}^{N+2,2}(f_0)) \epsilon^{4-4s}. \een
\end{enumerate}
\end{thm}
Some comments are in order:
\begin{rmk} Theorem \ref{main3} gives global well-posedness of \eqref{linearizedBE} for all $0 \leq \epsilon \leq \epsilon_{0}$. The non-cutoff case ($\epsilon=0$) is established in \cite{amuxy4} and \cite{gs1}, and the cutoff case (somehow can be considered as $\epsilon=\epsilon_{0}$) is finished in \cite{Guo1}. We get global well-posedness of \eqref{linearizedBE} uniformly in the whole range $0 \leq \epsilon \leq \epsilon_{0}$.
\end{rmk}

\begin{rmk} Estimates \eqref{decay-uniform-formula1}-\eqref{localizedenergy} show that the diversity of semi-group $e^{-\mathcal{L}^{\epsilon}t}$ in Theorem \ref{main2}
can also be observed  in the non-linear level. In other words, even in the perturbation framework,  the original solution $F$ of the Boltzmann equation converges to the equilibrium without any explicit rate. That is, we can only derive
$$\lim_{t\rightarrow\infty}\|\mu^{-\frac{1}{2}}(F(t)-\mu)\|_{L^2}=0.$$
We have two comments on this phenomenon. Firstly, if we go back to original equation, by energy-entropy method introduced in \cite{He-Jiang}, it holds that
\beno \lim_{t\rightarrow\infty}\|F(t)-\mu\|_{L^2}=O(t^{-\infty}). \eeno  Secondly, it is very interesting to ask what is the impact of such convergence on the hydrodynamic limit for the soft potentials.
\end{rmk}

\begin{rmk}  To our best knowledge, the results in Theorem \ref{main3} are new for the moderate soft potentials. To keep the paper in a reasonable size, we refrain from generalizing the results to the other potentials, which can be done by noticing that all the estimates involving $\mathcal{L}^\epsilon$ and $\Gamma^\epsilon$ in this article are valid for $\gamma>-3$. Using the estimates in this article, the very soft potentials $-3<\gamma < -2s$ are considered in \cite{He-Yao-Zhou}.
\end{rmk}

\subsection{Ideas and novelties} Let us illustrate  the ideas and novelties of the proof to our main theorems.

\subsubsection{Proof of Theorem \ref{main1}} We illustrate our strategy under the Maxwellian molecules case $\gamma=0$. It is not difficult to see (in the proof of Theorem \ref{equivalencenorm}) that the coercivity estimate of $\langle \mathcal{L}^{\epsilon}f, f\rangle_v$ can be reduced to the control of quantities $\mathcal{M}^{\epsilon,0}(f)$ and $\mathcal{R}_\mu^{\epsilon,0}(f)$ which correspond to gain of weight and  gain of regularity respectively.

\noindent $\bullet$ Instead of using Carleman representation of the collision operator, in Lemma \ref{lowerboundpart1}
we introduce a new coordinate system which enables us to make full use of the cancellation and the law of sines to estimate $\mathcal{M}^{\epsilon,0}(f)$. The method is elementary but effective to catch the hyperbolic structure of $\mathcal{L}^\epsilon$ uniformly in $\epsilon$.

\noindent$\bullet$ To give a precise description of $\mathcal{R}_\mu^{\epsilon,0}(f)$, we develop some new techniques. The first new idea is to apply the geometric decomposition to $\mathcal{R}_\mu^{\epsilon,0}(f)$ in the frequency space rather than phase space. More precisely,
by Bobylev's equality, we have
\beno
\mathcal{R}_\mu^{\epsilon,0}(f) &=& \frac{1}{(2\pi)^{3}} \int_{\mathbb{R}^3 \times  \mathbb{S}^2} b^{\epsilon}(\frac{\xi}{|\xi|} \cdot \sigma)(\hat{\mu}(0)|\hat{f}(\xi) - \hat{f}(\xi^{+})|^{2} + 2\Re((\hat{\mu}(0) - \hat{\mu}(\xi^{-}))\hat{f}(\xi^{+})\bar{\hat{f}}(\xi)) d\xi  d\sigma
\\ &:=& \frac{\hat{\mu}(0)}{(2\pi)^{3}}\mathcal{I}_{1} + \frac{2}{(2\pi)^{3}}\mathcal{I}_{2},
\eeno
where $\xi^{+} = \frac{\xi+|\xi|\sigma}{2}$ and $\xi^{-} = \frac{\xi-|\xi|\sigma}{2}$. It is not difficult to prove that
\beno
|\mathcal{I}_{2}| \lesssim |W^{\epsilon}(D)f|^{2}_{L^{2}}\lesssim \langle\mathcal{L}^{\epsilon} f,f \rangle_v+|f|_{L^2}^2,
\eeno
where the latter $\lesssim$ is given by \eqref{sobolevregu}. Therefore
we only need to consider the estimate of $\mathcal{I}_{1}$. By the geometric decomposition introduced in \cite{he2},
\beno
\hat{f}(\xi) - \hat{f}(\xi^{+}) = \hat{f}(\xi) - \hat{f}(|\xi|\frac{\xi^{+}}{|\xi^{+}|})+ \hat{f}(|\xi|\frac{\xi^{+}}{|\xi^{+}|}) - \hat{f}(\xi^{+}),
\eeno
we have
\beno
\mathcal{I}_{1} &=& \int_{\mathbb{R}^3 \times  \mathbb{S}^2} b^{\epsilon}(\frac{\xi}{|\xi|} \cdot \sigma)|\hat{f}(\xi) - \hat{f}(\xi^{+})|^{2} d\xi d\sigma
\\&\geq& \frac{1}{2} \int_{\mathbb{R}^3 \times  \mathbb{S}^2} b^{\epsilon}(\frac{\xi}{|\xi|} \cdot \sigma)|\hat{f}(\xi) - \hat{f}(|\xi|\frac{\xi^{+}}{|\xi^{+}|})|^{2} d\xi d\sigma
- \int_{\mathbb{R}^3 \times  \mathbb{S}^2} b^{\epsilon}(\frac{\xi}{|\xi|} \cdot \sigma)|\hat{f}(|\xi|\frac{\xi^{+}}{|\xi^{+}|}) - \hat{f}(\xi^{+})|^{2} d\xi d\sigma
\\&:=& \frac{1}{2}\mathcal{I}_{1,1} - \mathcal{I}_{1,2}.
\eeno
Thanks to the fact that Fourier transform is commutative with $W^\epsilon((-\triangle_{\mathbb{S}^2})^{\f12})$, we obtain the anisotropic regularity from  $\mathcal{I}_{1,1}$ (see Proposition \ref{key-anisotropic-result} and Lemma \ref{lowerboundpart2} for details). Now we only need to give the upper the  bound for  $\mathcal{I}_{1,2}$. Our key observation lies in that  $\hat{f}(\frac{\xi^{+}}{|\xi^{+}|})$ and $ \hat{f}(\xi^{+})$ can be localized in the same region both in frequency and phase space, which enables us to derive that  $\mathcal{I}_{1,2}$ can be bounded by $|W^{\epsilon}(D)f|^{2}_{L^{2}} + |W^{\epsilon}f|^{2}_{L^{2}}.$

To get the $\lesssim$ side of Theorem \ref{main1}, we have to give upper bound for $\langle Q^\epsilon (g,h),f\rangle_v$ and $\langle \Gamma^\epsilon (g,h),f\rangle_v$. To this end,  our new idea is to separate the integration domain into two regions, $|v-v_*|\le 1$ and $|v-v_*|\ge1$, to manifest the hyperbolic structure and the smoothing property of the operator.

(i). In the region $|v-v_*|\le 1$, the hyperbolic structure  prevails over
the anisotropic structure, which can be checked from the proof of the sharp bounds for the operator in weighted Sobolev spaces(see \cite{he2} for details). It suggests that we can use Sobolev regularity to give the upper bounds for the operator. See Proposition \ref{ubqepsilonsingular} for more details.

(ii). In the region  $|v-v_*|\ge 1$, the operator is dominated by the anisotropic structure. We resort to the geometric decomposition in the phase space to give the corresponding upper bounds. In particular, we make full use of the symmetric property of the structure inside the operator and also the dissipation $\mathcal{R}^{\epsilon,\gamma}_{*,g}(f)$ obtained from the lower bound of the operator. See Proposition \ref{ubqepsilonnonsingular} for more details.

\subsubsection{The proof of Theorem \ref{main2}} We have two novelties in the proof.

\noindent$\bullet$  The first one lies in the localization techniques in the phase space which are totally new and important considering that the Boltzmann equation is a non-local equation. It shows    that the linear or even non-linear Boltzmann equations can be almost localized  thanks to the commutator estimates (in Lemma \ref{CommSemi}) between $\mathcal{L}^\epsilon$ and the localization function.  This fact enables us to consider the evolution of the local energy which is the key to prove diversity of
  longtime behavior of  $e^{-\mathcal{L}^{\epsilon}t}f$.

\noindent$\bullet$ We reduce longtime behavior of  $e^{-\mathcal{L}^{\epsilon}t}f$ to some special ODE system. Based on a technical argument, we obtain a sharp estimate (in Proposition \ref{propODE}) for the ODE system which in turn gives the precise behavior of the semi-group.  The result shows that there exists a critical time $t_*$ such that the decay rate is totally different  before and after $t_*$ which matches the complex property of $\mathcal{L}^\epsilon$.

\subsubsection{Proof of Theorem \ref{main3}} The proof has some new features.

\noindent$\bullet$ Since we only impose the smallness assumption on $\|f\|_{H^2_xL^2}$, we have to find a new way to
  prove propagation of full regularity(\eqref{propagation-h-n-l-2-l} and \eqref{propagation-h-n-h-m-l}). To this end, we first close energy estimates for pure spatial regularity. Then the desired result is reduced to prove that if we have the control of $\dot{\mathcal{E}}^{N-j,j}(f)$ with $j\le N$, then by the equation we can get  the control of $\dot{\mathcal{E}}^{N-j-1,j+1}(f)$ thanks to the well-designed weight functions in \eqref{AsuWf}.

\noindent$\bullet$  To prove the global error estimate \eqref{error-function-uniform-estimate}, the key idea is to regard the error equation as a linear equation since we already have high order energy estimates \eqref{propagation-h-n-h-m-l} of the solutions to \eqref{linearizedBE} and \eqref{linearizedNBE}.

\subsection{Plan of the article} In Section 2, we endeavor to prove Theorem \ref{main1} and some upper bounds for the nonlinear term $\Gamma^\epsilon$.  Theorem \ref{main2} and Theorem \ref{main3} are proved in  Section 3 and Section 4 respectively. In the appendix, we give some necessary results for the sake of completeness.

\section{Bounds of the linearzied Boltzmann operator and the nonlinear term}
In this section, we will prove Theorem \ref{main1}. To this end, we separate the proof into two parts: lower bound of $\langle \mathcal{L}^\epsilon f, f\rangle_v$ and upper bound of $\langle \Gamma^\epsilon (g,h), f\rangle_v$. Moreover,  we give an estimate of the commutator between the collision operator $\Gamma^\epsilon (g,\cdot)$ and the weight function $W_{l}$ which are crucial to Theorem \ref{main2} and Theorem \ref{main3}.

Throughout this section, we assume that $0 \leq \epsilon \leq \epsilon_0$ with $\epsilon_0>0$ sufficiently small. Since many variables are
frequently used, we will sometimes omit their range in integrals.
Usually, $\sigma, \tau, \varsigma \in \mathbb{S}^{2}, v, v_{*}, u, \xi \in \mathbb{R}^{3}, \kappa \in [0,1], r \in \mathbb{R}_{+}$.  For instance, $\int (\cdots) d\sigma := \int_{\mathbb{S}^{2}} (\cdots) d\sigma,  \int (\cdots) d\sigma dv dv_{*} := \int_{\mathbb{S}^{2} \times \mathbb{R}^{3} \times \mathbb{R}^{3}} (\cdots) d\sigma dv dv_{*}$.

\subsection{Lower bound of the linearzied operator} Our strategy of the proof can be summarized as follows. We first give the  estimates of $\mathcal{R}_g^{\epsilon,\gamma}(f) $  and $ \mathcal{M}^{\epsilon,\gamma}(f)$.
 Then the lower bound of $\mathcal{L}^\epsilon$ is obtained by proving $\langle \mathcal{L}^{\epsilon}_{1}f, f\rangle_v + |f|^{2}_{L^{2}_{\gamma/2}} \gtrsim  \mathcal{R}_{\mu}^{\epsilon,\gamma}(f) + \mathcal{M}^{\epsilon,\gamma}(f)$ and $\langle \mathcal{L}^{\epsilon}_{2}f, f\rangle_v$ is a lower order term.

\subsubsection{Estimate of $\mathcal{M}^{\epsilon,\gamma}(f) $} Now we state a lemma on the functional  $\mathcal{M}^{\epsilon,\gamma}(f) $.
\begin{prop}\label{lowerboundpart1} There exists $\epsilon_{0} >0$ such that for  $0 \leq \epsilon \leq \epsilon_{0}$,
\beno
\mathcal{M}^{\epsilon,\gamma}(f) + |f|^{2}_{L^{2}_{\gamma/2}} \sim |W^{\epsilon}f|^{2}_{L^{2}_{\gamma/2}}.
\eeno
\end{prop}
\begin{proof} We only consider the case $\epsilon>0$. Note that with slight modification, our method also works for the case $\epsilon=0$. We divide the proof into two steps.

{\it Step 1: The lower bound of $\mathcal{M}^{\epsilon,\gamma}(f)$.} Note that $\nabla \mu^{\frac{1}{2}} = -\frac{\mu^{\frac{1}{2}}}{2} v$ and $\nabla^{2} \mu^{\frac{1}{2}} = \frac{\mu^{\frac{1}{2}}}{4} (-2I_{3}+v \otimes v)$ where $I_{3}$ is the $3 \times 3$ identity matrix. By Taylor expansion, we have
\beno
\mu^{\frac{1}{2}}(v^{\prime}) - \mu^{\frac{1}{2}}(v) = -\frac{\mu^{\frac{1}{2}}(v)}{2} v \cdot (v^{\prime}-v) + \int_{0}^{1} (1-\kappa) (\nabla^{2} \mu^{\frac{1}{2}}) (v(\kappa)):(v^{\prime}-v)\otimes(v^{\prime}-v) d \kappa,
\eeno
where $v(\kappa) = v+\kappa (v^{\prime}-v)$.
Using the inequality $(a-b)^{2} \geq \frac{a^{2}}{2} - b^{2}$, we have
\beno
(\mu^{\frac{1}{2}}(v^{\prime}) - \mu^{\frac{1}{2}}(v))^{2} \geq \frac{\mu(v)}{8} |v \cdot (v^{\prime}-v)|^{2} - \int_{0}^{1}  |(\nabla^{2} \mu^{\frac{1}{2}}) (v(\kappa))|^{2}|v^{\prime}-v|^{4} d \kappa.
\eeno

{\it Step 1.1: $|v_{*}| \leq \eta / \epsilon$.} Here $0 <\eta <1$ is a constant to be determined later.
Set $r = 4 \sqrt{2}$ and $A(\epsilon, \eta, r) = \{(v_{*},v,\sigma): 2r \leq |v_{*}| \leq \eta/\epsilon, |v| \leq r, 2\eta|v-v_{*}|^{-1} \leq \sin(\theta/2) \leq 4\eta|v-v_{*}|^{-1}\}$. Recall $\mathcal{M}^{\epsilon,\gamma}(f)$ defined in \eqref{defintion-M}. For simplicity, let $B^{\epsilon,\gamma}:= b^{\epsilon}(\cos\theta)|v-v_{*}|^{\gamma}$, then
\begin{eqnarray}\label{vsmallvstarsmall}
\mathcal{M}^{\epsilon,\gamma}(f) &\geq& \int B^{\epsilon,\gamma} \mathrm{1}_{A(\epsilon, \eta, r)} f_{*}^{2} ((\mu^{\frac{1}{2}})^{\prime}-\mu^{\frac{1}{2}})^{2} d\sigma dv dv_{*}
 \nonumber \\&\geq& \frac{1}{8}\int B^{\epsilon,\gamma} \mathrm{1}_{A(\epsilon, \eta, r)} \mu(v)|v \cdot (v^{\prime}-v)|^{2} f_{*}^{2}  d\sigma dv dv_{*} \nonumber
\\&& -  \int B^{\epsilon,\gamma} \mathrm{1}_{A(\epsilon, \eta, r)} |(\nabla^{2} \mu^{\frac{1}{2}}) (v(\kappa))|^{2}|v^{\prime}-v|^{4} f_{*}^{2}  d\sigma dv dv_{*} d\kappa \nonumber
\\&:=& \frac{1}{8}\mathcal{M}_{1}^{\epsilon,\gamma} (\eta) - \mathcal{M}_{2}^{\epsilon,\gamma} (\eta).
\end{eqnarray}

{\underline{The estimate of $\mathcal{M}_{1}^{\epsilon,\gamma} (\eta)$.}} For fixed $v, v_*$, we introduce an orthonormal basis $(h^{1}_{v,v_{*}},h^{2}_{v,v_{*}}, \frac{v-v_{*}}{|v-v_{*}|})$ such that $d\sigma= \sin\theta d\theta d\varphi$. Then one has
\beno
\frac{v^{\prime}-v}{|v^{\prime}-v|} = \cos\frac{\theta}{2}\cos\varphi h^{1}_{v,v_{*}} + \cos\frac{\theta}{2}\sin\varphi h^{2}_{v,v_{*}} -\sin\frac{\theta}{2} \frac{v-v_{*}}{|v-v_{*}|},
\\
\frac{v}{|v|} = c_{1} h^{1}_{v,v_{*}} + c_{2} h^{2}_{v,v_{*}} + c_{3} \frac{v-v_{*}}{|v-v_{*}|},
\eeno
where $c_{3}=\frac{v}{|v|}\cdot \frac{v-v_{*}}{|v-v_{*}|}$ and $c_{1}, c_{2}$ are  constants independent of $\theta$ and $\varphi$. Then we have
\beno
|\frac{v}{|v|} \cdot \frac{v^{\prime}-v}{|v^{\prime}-v|}|^{2}  &=& |c_{1}\cos\frac{\theta}{2}\cos\varphi + c_{2}\cos\frac{\theta}{2}\sin\varphi - c_{3}\sin\frac{\theta}{2}|^{2},
\\&=& c^{2}_{1}\cos^{2}\frac{\theta}{2}\cos^{2}\varphi + c^{2}_{2}\cos^{2}\frac{\theta}{2}\sin^{2}\varphi + c^{2}_{3}\sin^{2}\frac{\theta}{2}
\\ && + 2c_{1}c_{2}\cos^{2}\frac{\theta}{2}\cos\varphi\sin\varphi - 2c_{3}\cos\frac{\theta}{2}\sin\frac{\theta}{2}(c_{1}\cos\varphi + c_{2}\sin\varphi).
\eeno
Integrating with respect to $\sigma$, we have
\ben \nonumber
&& \int b^{\epsilon}(\cos\theta)\mathrm{1}_{A(\epsilon, \eta, r)}|v \cdot (v^{\prime}-v)|^{2}d\sigma
\\ \nonumber &=& \int_{0}^{\pi/2}\int_{0}^{2\pi}b^{\epsilon}(\cos\theta)\sin\theta \mathrm{1}_{A(\epsilon, \eta, r)}|v \cdot (v^{\prime}-v)|^{2} d\theta d\varphi
\\ &\geq& \pi(c^{2}_{1}+c^{2}_{2})|v|^{2}|v-v_{*}|^{2}
\int_{0}^{\pi/2} b^{\epsilon}(\cos\theta)\sin\theta \cos^{2}\frac{\theta}{2} \sin^{2}\frac{\theta}{2}\mathrm{1}_{A(\epsilon, \eta, r)} d\theta, \label{sin-square}
\een
If $(v_{*},v,\sigma) \in A(\epsilon, \eta, r)$, then $|v-v_{*}| \geq |v_{*}|-|v| \geq r$ and thus $4\eta|v-v_{*}|^{-1} \leq 4 r^{-1} \leq \sqrt{2}/{2}$. Suppose $\epsilon \leq \eta/2r$, then
$|v-v_{*}| \leq |v|+|v_{*}| \leq r + \eta/\epsilon \leq 3\eta / 2\epsilon$ and thus $2\eta|v-v_{*}|^{-1} \geq 4\epsilon/3$. Recall $\phi$ in \eqref{function-phi-psi} and $b^{\epsilon}$ in \eqref{cutoff-kernel-def} to see
\ben \label{b-epsilon-is-b-when-theta-large}
b^{\epsilon}(\cos\theta) \mathrm{1}_{4\epsilon/3 \leq \sin(\theta/2) \leq \sqrt{2}/{2}} = b(\cos\theta) \mathrm{1}_{4\epsilon/3 \leq \sin(\theta/2) \leq \sqrt{2}/{2}},
\een
which gives
\ben \nonumber
&&\int_{0}^{\pi/2} b^{\epsilon}(\cos\theta)\sin\theta \cos^{2}\frac{\theta}{2} \sin^{2}\frac{\theta}{2}\mathrm{1}_{A(\epsilon, \eta, r)}d\theta
\\&=& \int_{0}^{\pi/2} b(\cos\theta)\sin\theta \cos^{2}\frac{\theta}{2} \sin^{2}\frac{\theta}{2}\mathrm{1}_{A(\epsilon, \eta, r)}d\theta
\nonumber \\&\gtrsim& \mathrm{1}_{B(\epsilon, \eta, r)} \int_{2\eta|v-v_{*}|^{-1}}^{4\eta|v-v_{*}|^{-1}} t^{1-2s}dt
 \gtrsim \eta^{2-2s}|v-v_{*}|^{2s-2}\mathrm{1}_{B(\epsilon, \eta, r)}, \label{integrate-the-angular-function}
\een
where we use \eqref{angular-lower-upper-bound} and the change of variable $t = \sin\frac{\theta}{2}$. Here $B(\epsilon, \eta, r) = \{(v_{*},v): 2r \leq |v_{*}| \leq \eta/\epsilon, |v| \leq r\}$.
Plugging \eqref{sin-square} and \eqref{integrate-the-angular-function} into the definition of $\mathcal{M}_{1}^{\epsilon,\gamma} (\eta)$, we have
\beno
\mathcal{M}_{1}^{\epsilon,\gamma} (\eta) &\gtrsim& \eta^{2-2s} \int (c^{2}_{1}+c^{2}_{2})|v-v_{*}|^{\gamma+2s}|v|^{2}\mathrm{1}_{B(\epsilon, \eta, r)} \mu(v) f_{*}^{2}dv dv_{*}
\\&=& \eta^{2-2s} \int (1-(\frac{v}{|v|}\cdot\frac{v_{*}}{|v_{*}|})^{2})|v_{*}|^{2}|v-v_{*}|^{\gamma+2s-2}|v|^{2}\mathrm{1}_{B(\epsilon, \eta, r)} \mu(v) f_{*}^{2}dv dv_{*},
\eeno
where in the last line we use the fact $c_1^2+c_2^2+c_3^2=1$ and the law of sines $$(1-(\frac{v}{|v|}\cdot\frac{v_{*}}{|v_{*}|})^{2})^{-1}|v-v_{*}|^{2}= (1-c_3^2)^{-1}|v_{*}|^{2}.$$
Note that in the region $B(\epsilon, \eta, r)$, one has $|v-v_{*}| \sim \frac{3}{2}|v_{*}|$ and thus $|v-v_{*}|^{\gamma+2s-2} \sim |v_{*}|^{\gamma+2s-2}$. Recall $r=4 \sqrt{2}$, denote $c_{1} = \int (1-(\frac{v}{|v|}\cdot\frac{v_{*}}{|v_{*}|})^{2}) |v|^{2} \mu(v) \mathrm{1}_{|v|\leq 4 \sqrt{2}}dv$ and observe that the value of the integral is independent of $v_{*}$. Therefore we have
\ben
\mathcal{M}_{1}^{\epsilon,\gamma} (\eta, r) &\gtrsim&  \eta^{2-2s} c_{1} \int |v_{*}|^{\gamma+2s} \mathrm{1}_{2r \leq |v_{*}|\leq \eta/\epsilon}  f_{*}^{2} dv_{*}
\nonumber \\ \label{M-1-lower-bound}
&\geq& \eta^{2-2s} c_{1} (\int \langle v_{*} \rangle^{\gamma+2s}\mathrm{1}_{|v_{*}|\leq \eta/\epsilon}  f_{*}^{2} dv_{*}- (8\sqrt{2})^{2s}|f|^{2}_{L^{2}_{\gamma/2}}) .
\een

{\underline{The estimate of $\mathcal{M}_{2}^{\epsilon,\gamma} (\eta)$.}} By the change of variable $v \rightarrow v(\kappa)$,
let $\cos \theta(\kappa) =\f{v(\kappa)-v_*}{|v(\kappa)-v_*|}\cdot \sigma$, then $\theta/2 \leq \theta(\kappa) \leq \theta$ and
it is not difficult to check
\ben
\mathcal{M}_{2}^{\epsilon,\gamma} (\eta)
&\lesssim&  \int (\int_{\eta|v-v_{*}|^{-1}}^{8\eta|v-v_{*}|^{-1}} t^{3-2s} dt) |v-v_{*}|^{\gamma+4} \mathrm{1}_{|v_{*}|\leq \eta/\epsilon} f_{*}^{2} \mu^{\frac{1}{2}} dv  dv_{*}
\nonumber \\ \label{M-2-upper-bound}
&\lesssim& \eta^{4-2s} \int \langle v_{*} \rangle^{\gamma+2s} \mathrm{1}_{|v_{*}|\leq \eta/\epsilon}  f_{*}^{2} dv_{*}.
\een

Plugging \eqref{M-1-lower-bound} and \eqref{M-2-upper-bound} into \eqref{vsmallvstarsmall}, for some universal constant $C\geq 1$, we have
\beno
\mathcal{M}^{\epsilon,\gamma}(f) \gtrsim \eta^{2-2s} (1-C\eta^{2}) \int \langle v_{*} \rangle^{\gamma+2s}\mathrm{1}_{|v_{*}|\leq \eta/\epsilon}  f_{*}^{2} dv_{*} - C\eta^{2-2s}|f|^{2}_{L^{2}_{\gamma/2}}.
\eeno
Choosing $\eta$ such that $C\eta^{2} = 1/2$, we have
\begin{eqnarray}\label{lowerboundvstarsmall}
\mathcal{M}^{\epsilon,\gamma}(f)+ |f|^{2}_{L^{2}_{\gamma/2}} \gtrsim  \int \langle v_{*} \rangle^{\gamma+2s}\mathrm{1}_{|v_{*}|\leq \eta/\epsilon}  f_{*}^{2} dv_{*}.
\end{eqnarray}

{\it Step 1.2: $|v_{*}| \geq R / \epsilon$.} Here $R \geq 1$.
By direct computation, we have
\beno
\mathcal{M}^{\epsilon,\gamma}(f) &=& \int B^{\epsilon,\gamma} f_{*}^{2} ((\mu^{\frac{1}{2}})^{\prime}-\mu^{\frac{1}{2}})^{2} d\sigma dv dv_{*}
\\&\geq&  \int b^{\epsilon}|v-v_{*}|^{\gamma}\mathrm{1}_{|v_{*}|\geq R/\epsilon} f_{*}^{2} \mu d\sigma dv dv_{*}
- 2\int b^{\epsilon}|v-v_{*}|^{\gamma}\mathrm{1}_{|v_{*}|\geq R/\epsilon} f_{*}^{2} (\mu^{\frac{1}{2}})^{\prime}\mu^{\frac{1}{2}} d\sigma dv dv_{*}
\\&:=&\mathcal{M}_{1}^{\epsilon,\gamma,R}-\mathcal{M}_{2}^{\epsilon,\gamma,R}.
\eeno
Recalling \eqref{b-epsilon-is-b-when-theta-large} and \eqref{angular-lower-upper-bound}, using the change of variable $t = \sin\frac{\theta}{2}$, we have
\beno
\mathcal{M}_{1}^{\epsilon,\gamma,R} \gtrsim \int_{4\epsilon/3}^{\sqrt{2}/2} t^{-1-2s} dt \int |v-v_{*}|^{\gamma}\mathrm{1}_{|v_{*}| \geq R/\epsilon}  f_{*}^{2} \mu dv  dv_{*}
\gtrsim \epsilon^{-2s}\int \langle v_{*} \rangle^{\gamma}\mathrm{1}_{|v_{*}| \geq R/\epsilon} f_{*}^{2}dv_{*},
\eeno
where we use and $\int |v-v_{*}|^{\gamma} \mu dv \gtrsim \langle v_{*} \rangle^{\gamma}$ and
$\int_{4\epsilon/3}^{\sqrt{2}/2} t^{-1-2s} dt \gtrsim \epsilon^{-2s}$ when $\epsilon \leq \frac{1}{10}$.
Recalling the support of $b^{\epsilon}$ belongs to $\sin \frac{\theta}{2} \geq \frac{3}{4}\epsilon$,
there holds $|v^{\prime}|+|v| \geq |v^{\prime}-v| = \sin\frac{\theta}{2}|v-v_{*}|\geq \frac{3}{4}\epsilon|v-v_{*}|\geq \frac{3}{4}\epsilon(|v_{*}|-|v|)$ and thus $|v^{\prime}|+(1+\frac{3}{4}\epsilon)|v| \geq \frac{3}{4}\epsilon|v_{*}| \geq \frac{3}{4} R$. Then
$R^{2}/2 \leq  4(|v^{\prime}|+|v|)^{2} \leq 8 (|v^{\prime}|^{2}+|v|^{2})$ and so
\beno
(\mu^{\frac{1}{2}})^{\prime}\mu^{\frac{1}{2}} = (2 \pi)^{-3/2} e^{-\frac{|v^{\prime}|^{2}+|v|^{2}}{4}} \lesssim e^{-\frac{|v|^{2}}{8}}e^{-\frac{R^{2}}{2^{7}}}.
\eeno
Thus we have
\beno
\mathcal{M}_{2}^{\epsilon,\gamma,R} \lesssim e^{-\frac{R^{2}}{2^{7}}} \epsilon^{-2s}\int \langle v_{*} \rangle^{\gamma}\mathrm{1}_{|v_{*}| \geq R/\epsilon} f_{*}^{2}dv_{*}.
\eeno
Patching together the above estimates of $\mathcal{M}_{1}^{\epsilon,\gamma,R}$ and $\mathcal{M}_{2}^{\epsilon,\gamma,R}$, we arrive at
\ben \label{v-is-large}
\mathcal{M}^{\epsilon,\gamma}(f) \geq  (C_{1} - C_{2}e^{-\frac{R^{2}}{2^{7}}})\epsilon^{-2s}\int \langle v_{*} \rangle^{\gamma}\mathrm{1}_{|v_{*}|\geq R/\epsilon}  f_{*}^{2} dv_{*},
\een
for some universal constants $C_{1}$ and $C_{2}$.

{\it Step 1.3: $|v_{*}| \geq \eta / \epsilon$.} Here $\eta$ is the fixed constant in {\it Step 1.1}.
Note that the estimate \eqref{v-is-large} is valid for any $R \geq 1$ and $\epsilon \leq \frac{1}{10}$.
We choose $R=N\eta$, where $N \geq 1$ is large enough such that $C_{1}- C_{2}e^{-\frac{(N\eta)^{2}}{2^{7}}} \geq \frac{C_{1}}{2}$. Then   by \eqref{v-is-large}, when $N\epsilon \leq \frac{1}{10}$, we have
\beno
\mathcal{M}^{\epsilon,\gamma}(f) \geq  \mathcal{M}^{N\epsilon,\gamma}(f) &\geq&  (C_{1}- C_{2}e^{-\frac{(N\eta)^{2}}{2^{7}}})(N\epsilon)^{-2s}\int \langle v_{*} \rangle^{\gamma}\mathrm{1}_{|v_{*}|\geq \eta/\epsilon}  f_{*}^{2} dv_{*}
\\&\geq& \frac{C_{1}}{2} N^{-2s}\epsilon^{-2s}\int \langle v_{*} \rangle^{\gamma}\mathrm{1}_{|v_{*}|\geq \eta/\epsilon}  f_{*}^{2} dv_{*}.
\eeno
From which together with \eqref{lowerboundvstarsmall}, taking $\epsilon_{0}:=\min\{ \frac{\eta}{8\sqrt{2}}, \frac{1}{10N}\}$, when $\epsilon \leq \epsilon_{0}$, we arrive at
\beno
\mathcal{M}^{\epsilon,\gamma}(f) + |f|^{2}_{L^{2}_{\gamma/2}} \gtrsim  \int \langle v_{*} \rangle^{\gamma+2s}\mathrm{1}_{|v_{*}|\leq \eta/\epsilon}  f_{*}^{2} dv_{*}+\epsilon^{-2s}\int \langle v_{*} \rangle^{\gamma}\mathrm{1}_{|v_{*}|\geq \eta/\epsilon}  f_{*}^{2} dv_{*}
\gtrsim |W^{\epsilon}f|^{2}_{L^{2}_{\gamma/2}}.
\eeno

{\it Step 2: The upper bound of $\mathcal{M}^{\epsilon,\gamma}(f)$.} Since $((\mu^{\frac{1}{2}})^{\prime}-\mu^{\frac{1}{2}})^{2} \leq 2 ((\mu^{\frac{1}{4}})^{\prime}-\mu^{\frac{1}{4}})^{2}((\mu^{\frac{1}{2}})^{\prime}+\mu^{\frac{1}{2}})$,
we have
\beno
\mathcal{M}^{\epsilon,\gamma}(f) &\lesssim& \int  B^{\epsilon,\gamma} f_{*}^{2} ((\mu^{\frac{1}{4}})^{\prime}-\mu^{\frac{1}{4}})^{2}(\mu^{\frac{1}{2}})^{\prime} d\sigma dv dv_{*} + \int  B^{\epsilon,\gamma} f_{*}^{2} ((\mu^{\frac{1}{4}})^{\prime}-\mu^{\frac{1}{4}})^{2}\mu^{\frac{1}{2}} d\sigma dv dv_{*}
\\&:=& \mathcal{M}^{\epsilon,\gamma}_{1}(f) + \mathcal{M}^{\epsilon,\gamma}_{2}(f).
\eeno
By Taylor expansion, one has
$((\mu^{\frac{1}{4}})^{\prime} - \mu^{\frac{1}{4}})^{2} \lesssim \min\{1,|v-v_{*}|^{2}\sin^{2}\frac{\theta}{2}\} \sim \min\{1,|v^{\prime}-v_{*}|^{2}\sin^{2}\frac{\theta}{2}\}.$
By Proposition \ref{symbol}, we have
$
\int b^{\epsilon}(\cos\theta) \min\{1,|v-v_{*}|^{2}\sin^{2}\frac{\theta}{2}\} d\sigma \lesssim (W^{\epsilon})^2(v-v_*).
$
After checking
\begin{eqnarray}\label{combinewithgamma}
 (W^{\epsilon})^2(v-v_*)   \lesssim (W^{\epsilon})^{2}(v)(W^{\epsilon})^{2}(v_{*}),
\end{eqnarray}
we have
$
\int b^{\epsilon}(\cos\theta) \min\{1,|v-v_{*}|^{2}\sin^{2}\frac{\theta}{2}\} d\sigma \lesssim  (W^{\epsilon})^{2}(v)(W^{\epsilon})^{2}(v_{*}).
$
Thus we have
\beno
\mathcal{M}^{\epsilon,\gamma}_{2}(f) \lesssim \int f_{*}^{2} |v-v_{*}|^{\gamma}(W^{\epsilon})^{2}(v)(W^{\epsilon})^{2}(v_{*})  \mu^{\frac{1}{2}} dv dv_{*} \lesssim  |W^\epsilon f|_{L^2_{\gamma/2}}^2,
\eeno
where we use the following estimate(see \cite{amuxy4}). For $\gamma>-3, a >0$, there holds
\ben \label{mu-cancel-sigularity}
\int |v-v_{*}|^{\gamma}  \mu^{a}(v) dv \leq C_{\gamma,a} \langle v_{*} \rangle^{\gamma}.
\een

The term $\mathcal{M}^{\epsilon,\gamma}_{1}(f)$ can be similarly estimated by the change of variable $v \rightarrow v^{\prime}$. Indeed, one has $\mathcal{M}^{\epsilon,\gamma}_{1}(f) \lesssim \int  b^{\epsilon}(\cos(2\theta^{\prime})) |v^{\prime}-v_{*}|^{\gamma} f_{*}^{2} ((\mu^{\frac{1}{4}})^{\prime}-\mu^{\frac{1}{4}})^{2}(\mu^{\frac{1}{2}})^{\prime} d\sigma dv^{\prime} dv_{*},$
where $\theta^{\prime}$ is the angle between $v^{\prime}-v_{*}$ and $\sigma$. With the fact $\theta^{\prime} = \theta/2$, we also have
\beno
\int b^{\epsilon}(\cos(2\theta^{\prime})) \min\{1,|v^{\prime}-v_{*}|^{2}\sin^{2}\frac{\theta}{2}\} d\sigma \lesssim  (W^{\epsilon})^{2}(v^{\prime})(W^{\epsilon})^{2}(v_{*}).
\eeno
Thus by exactly the same argument as that for $\mathcal{M}^{\epsilon,\gamma}_{2}(f)$, we have $\mathcal{M}^{\epsilon,\gamma}_{1}(f) \lesssim |W^\epsilon f|_{L^2_{\gamma/2}}^2$.
The proof is complete.
\end{proof}

\subsubsection{Estimate of $\mathcal{R}_\mu^{\epsilon,\gamma}(f)$ } We recall from  \cite{advw}
that  for $g\ge0$ with $|g|_{L^1}\ge\delta>0$ and $|g|_{L^1_1 \cap L\log L}\le \lambda$,
\beno \int b(\cos\theta)g_*(f'-f)^2d\sigma dv_*dv+|f|_{L^2}^2\ge C(\delta, \lambda)|a(D)f|_{L^2}^2, \eeno
where $a(\xi):= \int b(\f{\xi}{|\xi|}\cdot \sigma)\min\{ |\xi|^2\sin^2(\theta/2),1\} d\sigma + 1$.
Recalling \eqref{defintion-R-g} and by Proposition \ref{symbol}, we get
\begin{prop}\label{lowerboundpart1-general-g}
There holds
\beno
\mathcal{R}_\mu^{\epsilon,0}(f)+ |f|^{2}_{L^{2}} \gtrsim |W^{\epsilon}(D)f|^{2}_{L^{2}}.
\eeno
\end{prop}
Thus it remains to derive the anisotropic regularity from the lower bound of  $\mathcal{R}_\mu^{\epsilon,\gamma}(f)$.  To this end, we derive three technical lemmas.
 \begin{lem}\label{a-technical-lemma}
There holds
\beno \mathcal{A} := \int_{\R^{3}}\int_{\epsilon}^{\pi/4} \theta^{-1-2s}|f(v) - f(v/\cos\theta)|^{2} dv d\theta \lesssim |W^{\epsilon}(D)f|^{2}_{L^{2}} + |W^{\epsilon}f|^{2}_{L^{2}}.\eeno
\end{lem}
\begin{proof}
Applying dyadic decomposition in the phase space, since $\frac{\sqrt{2}}{2} \leq \cos\theta \leq 1$ for $\theta \in [0, \pi/4]$,
we have
\beno
\mathcal{A} &=& \int_{\R^{3}}\int_{\epsilon}^{\pi/4} \theta^{-1-2s}| \sum_{k=-1}^{\infty}(\varphi_{k}f)(v)- \sum_{k=-1}^{\infty}(\varphi_{k}f)(v/\cos\theta)|^{2} dv d\theta
\\&\lesssim& \sum_{k=-1}^{\infty}\int_{\R^{3}}\int_{\epsilon}^{\pi/4} \theta^{-1-2s}| (\varphi_{k}f)(v)- (\varphi_{k}f)(v/\cos\theta)|^{2} dv d\theta :=\sum_{k=-1}^{\infty}\mathcal{A}_{k}.
\eeno
It is easy to check $\sum_{2^{k} \geq 1/\epsilon} \mathcal{A}_{k} \lesssim |W^{\epsilon}f|^{2}_{L^{2}}$ since $\int_{\epsilon}^{\pi/4} \theta^{-1-2s} d \theta \lesssim \epsilon^{-2s}$.
For the case $2^{k} \leq 1/\epsilon$, by Plancherel's theorem and dyadic decomposition in the frequency space, we have
\beno
\mathcal{A}_{k}&=&
\int_{\R^{3}}\int_{\epsilon}^{\pi/4} \theta^{-1-2s}|\widehat{\varphi_{k}f}(\xi)- \cos^{3}\theta\widehat{\varphi_{k}f}(\xi\cos\theta)|^{2} d\xi d\theta
\\&\lesssim& \int_{\R^{3}}\int_{\epsilon}^{\pi/4} \theta^{-1-2s}|\widehat{\varphi_{k}f}(\xi)- \widehat{\varphi_{k}f}(\xi\cos\theta)|^{2} d\xi d\theta + |\varphi_{k}f|^{2}_{L^{2}}
\\&=& \int_{\R^{3}}\int_{\epsilon}^{\pi/4} \theta^{-1-2s}|\sum_{l=-1}^{\infty}(\varphi_{l}\widehat{\varphi_{k}f})(\xi)- \sum_{l=-1}^{\infty}(\varphi_{l}\widehat{\varphi_{k}f})(\xi\cos\theta)|^{2} d\xi d\theta + |\varphi_{k}f|^{2}_{L^{2}}
\\&\lesssim& \sum_{l=-1}^{\infty} \int_{\R^{3}}\int_{\epsilon}^{\pi/4} \theta^{-1-2s}|(\varphi_{l}\widehat{\varphi_{k}f})(\xi)- (\varphi_{l}\widehat{\varphi_{k}f})(\xi\cos\theta)|^{2} d\xi d\theta + |\varphi_{k}f|^{2}_{L^{2}}
\\&:=& \sum_{l=-1}^{\infty}\mathcal{A}_{k,l}  + |\varphi_{k}f|^{2}_{L^{2}}.
\eeno
Note that $\sum_{2^{l} \geq 1/\epsilon} \mathcal{A}_{k,l} \lesssim |W^{\epsilon}(D)\varphi_{k}f|^{2}_{L^{2}}$, thus $\mathcal{A}_{k} \lesssim \sum_{2^{l} \leq 1/\epsilon} \mathcal{A}_{k,l} + |W^{\epsilon}(D)\varphi_{k}f|^{2}_{L^{2}}+|\varphi_{k}f|^{2}_{L^{2}}$.
Using $\sum_{k \geq -1}^{\infty}|\varphi_{k}f|^{2}_{L^{2}} \lesssim |f|^{2}_{L^{2}}$ and \eqref{decompostionpacth}, we have
\beno
\mathcal{A}  \lesssim \sum_{2^{k} \leq 1/\epsilon,2^{l}\leq 1/\epsilon}\mathcal{A}_{k,l}+ |W^{\epsilon}(D)f|^{2}_{L^{2}} + |W^{\epsilon}f|^{2}_{L^{2}}.
\eeno
For each $k$ and $l$ such that $2^{k} \leq 1/\epsilon,2^{l}\leq 1/\epsilon$, we have
\ben \nonumber
\mathcal{A}_{k,l}  &=&  \int_{\R^{3}}\int_{\epsilon}^{2^{-k/2-l/2}} \theta^{-1-2s}|(\varphi_{l}\widehat{\varphi_{k}f})(\xi)- (\varphi_{l}\widehat{\varphi_{k}f})(\xi\cos\theta)|^{2} d\xi d\theta
\\ \nonumber &&+ \int_{\R^{3}}\int_{2^{-k/2-l/2}}^{\pi/4} \theta^{-1-2s}|(\varphi_{l}\widehat{\varphi_{k}f})(\xi)- (\varphi_{l}\widehat{\varphi_{k}f})(\xi\cos\theta)|^{2} d\xi d\theta
\\ \nonumber &\lesssim& \int_{\R^{3}}\int_{\epsilon}^{2^{-k/2-l/2}} \theta^{-1-2s}|(\varphi_{l}\widehat{\varphi_{k}f})(\xi)- (\varphi_{l}\widehat{\varphi_{k}f})(\xi\cos\theta)|^{2} d\xi d\theta
+ 2^{s(l+k)}|\varphi_{l}\widehat{\varphi_{k}f}|^{2}_{L^{2}}
\\ \label{A-k-l-to-Bkl} &:=&\mathcal{B}_{k,l}+ 2^{s(l+k)}|\varphi_{l}\widehat{\varphi_{k}f}|^{2}_{L^{2}}.
\een
By Taylor expansion,
$
(\varphi_{l}\widehat{\varphi_{k}f})(\xi)- (\varphi_{l}\widehat{\varphi_{k}f})(\xi\cos\theta) = (1-\cos\theta)\int_{0}^{1}
(\nabla \varphi_{l}\widehat{\varphi_{k}f})(\xi(\kappa))\cdot \xi d\kappa,
$
where $\xi(\kappa) = (1-\kappa)\xi\cos\theta + \kappa \xi$. Thus  we obtain
\beno
\mathcal{B}_{k,l}
\lesssim \int_{0}^{1}\int_{\R^{3}}\int_{\epsilon}^{2^{-k/2-l/2}} \theta^{3-2s}|\xi|^{2}
|(\nabla \varphi_{l}\widehat{\varphi_{k}f})(\xi(\kappa)) |^{2} d\kappa d\xi d\theta .
\eeno
By the change of variable $\xi \rightarrow \eta = \xi(\kappa)$, we have
\beno
\mathcal{B}_{k,l}
&=& \int_{0}^{1}\int_{\R^{3}}\int_{\epsilon}^{2^{-k/2-l/2}} \theta^{3-2s}\frac{|\eta|^{2}}{((1-\kappa)\cos\theta + \kappa)^{5}}
|(\nabla \varphi_{l}\widehat{\varphi_{k}f})(\eta)|^{2} d\kappa d\eta d\theta
\\&\lesssim&\int_{\R^{3}}\int_{\epsilon}^{2^{-k/2-l/2}} \theta^{3-2s}|\eta|^{2}|(\nabla \varphi_{l}\widehat{\varphi_{k}f})(\eta)|^{2} d\eta d\theta
\\&\lesssim& 2^{-(2-s)(l+k)} \int_{\R^{3}} |\eta|^{2}|(\nabla \varphi_{l}\widehat{\varphi_{k}f})(\eta)|^{2} d\eta
\lesssim 2^{s(l+k)}2^{-2k} \int_{\R^{3}} |(\nabla \varphi_{l}\widehat{\varphi_{k}f})(\eta)|^{2} d\eta.
\eeno
Note that
$
(\nabla \varphi_{l}\widehat{\varphi_{k}f})(\eta) = (\nabla \varphi_{l})(\eta) \widehat{\varphi_{k}f}(\eta) + \varphi_{l}(\eta) (\nabla \widehat{\varphi_{k}f})(\eta)
= 2^{-l}(\nabla \varphi) (\frac{\eta}{2^{l}}) \widehat{\varphi_{k}f}(\eta) - \mathrm{i} (\varphi_{l} \widehat{v\varphi_{k}f}) (\eta),
$
which gives
$
|(\nabla \varphi_{l}\widehat{\varphi_{k}f}) (\eta)|^{2} \lesssim 2^{-2l}|\nabla \varphi|^{2}_{L^{\infty}} |\widehat{\varphi_{k}f} (\eta)|^{2}
+  |(\varphi_{l}\widehat{v \varphi_{k}f}) (\eta)|^{2},
$
and
\beno
\mathcal{B}_{k,l}
\lesssim 2^{-(2-s)(l+k)} |\varphi_{k}f|^{2}_{L^{2}} + 2^{s(l+k)-2k}|\varphi_{l} \widehat{v\varphi_{k}f}|^{2}_{L^{2}}.
\eeno
Recalling \eqref{A-k-l-to-Bkl}, we arrive at
\beno
\mathcal{A}_{k,l} \lesssim  2^{-(2-s)(l+k)} |\varphi_{k}f|^{2}_{L^{2}} + 2^{s(l+k)-2k}|\varphi_{l} \widehat{v\varphi_{k}f}|^{2}_{L^{2}} + 2^{s(l+k)}|\varphi_{l}\widehat{\varphi_{k}f}|^{2}_{L^{2}} :=\mathcal{A}_{k,l,1}+\mathcal{A}_{k,l,2}+\mathcal{A}_{k,l,3}.
\eeno
The first term is estimated by $\sum_{2^{k} \leq 1/\epsilon,2^{l}\leq 1/\epsilon} \mathcal{A}_{k,l,1} \lesssim |f|^{2}_{L^{2}}.$
For the second term $\mathcal{A}_{k,l,2}$, we have
\beno
\sum_{2^{k} \leq 1/\epsilon,2^{l}\leq 1/\epsilon} \mathcal{A}_{k,l,2}
&\lesssim& \sum_{j=1}^{3}\sum_{2^{k} \leq 1/\epsilon,2^{l}\leq 1/\epsilon} 2^{2sl}2^{-2k}|\varphi_{l} \widehat{v_{j}\varphi_{k}f}|^{2}_{L^{2}} + \sum_{j=1}^{3}\sum_{2^{k} \leq 1/\epsilon,2^{l}\leq 1/\epsilon} 2^{2sk}2^{-2k}|\varphi_{l} \widehat{v_{j}\varphi_{k}f}|^{2}_{L^{2}}
\\&\lesssim& \sum_{j=1}^{3}\sum_{2^{k} \leq 1/\epsilon} 2^{-2k}|W^{\epsilon}\widehat{v_{j}\varphi_{k}f}|^{2}_{L^{2}} + \sum_{j=1}^{3}\sum_{2^{k} \leq 1/\epsilon} 2^{2sk}2^{-2k}|v_{j} \varphi_{k}f|^{2}_{L^{2}}
\\&\lesssim& \sum_{j=1}^{3}\sum_{2^{k} \leq 1/\epsilon} 2^{-2k}|W^{\epsilon}(D) v_{j} \varphi_{k}f|^{2}_{L^{2}} + |W^{\epsilon}f|^{2}_{L^{2}}\lesssim |W^{\epsilon}(D)f|^{2}_{L^{2}} + |W^{\epsilon}f|^{2}_{L^{2}}.
\eeno
Here in the last inequality, we apply Lemma \ref{operatorcommutator1}  to get $|W^{\epsilon}(D) v_{j} \varphi_{k}f|^{2}_{L^{2}} \lesssim |v_{j} \varphi_{k}W^{\epsilon}(D)f|^{2}_{L^{2}} +  |f|^{2}_{H^{s-1}}$
thanks to $W^{\epsilon}\in S^{s}_{1,0}, v_{j}\varphi_{k} \in S^{1}_{1,0}$(see Definition \ref{psuopde} for $S^m_{1,0}$).
As for the last term $\mathcal{A}_{k,l,3}$, we have
\beno
\sum_{2^{k} \leq 1/\epsilon,2^{l}\leq 1/\epsilon} \mathcal{A}_{k,l,3}
&\lesssim& \sum_{2^{k} \leq 1/\epsilon,2^{l}\leq 1/\epsilon} 2^{2sl}|\varphi_{l}\widehat{\varphi_{k}f}|^{2}_{L^{2}}  +
\sum_{2^{k} \leq 1/\epsilon,2^{l}\leq 1/\epsilon} 2^{2sk}|\varphi_{l}\widehat{\varphi_{k}f}|^{2}_{L^{2}}
\\&\lesssim&  \sum_{2^{k} \leq 1/\epsilon} |W^{\epsilon}(D)\varphi_{k}f|^{2}_{L^{2}} + \sum_{2^{k} \leq 1/\epsilon} 2^{2sk}|\widehat{\varphi_{k}f}|^{2}_{L^{2}}
\lesssim |W^{\epsilon}(D)f|^{2}_{L^{2}} + |W^{\epsilon}f|^{2}_{L^{2}}.
\eeno	
The lemma follows from the above estimates.
\end{proof}
\begin{rmk}\label{epsilon-to-another} If we change the integral range  $\int_{\epsilon}^{\pi/4}$ in Lemma \ref{a-technical-lemma} to
$\int_{3\epsilon/4}^{\pi/4}$, the estimate still holds true.
\end{rmk}

\begin{lem}\label{gammanonzerotozero}
	Let $
	\mathcal{Z}^{\epsilon,\gamma}(f) := \int b^{\epsilon}(\frac{u}{|u|}\cdot\sigma)\langle u\rangle^{\gamma} |f(|u|\frac{u^{+}}{|u^{+}|})-f(u^{+})|^{2} d\sigma du
	$ with $u^{+} = \frac{u + |u|\sigma}{2}$. Then
	\beno
	\mathcal{Z}^{\epsilon,\gamma}(f) \lesssim |W^{\epsilon}(D)W_{\gamma/2}f|^{2}_{L^{2}} + |W^{\epsilon}W_{\gamma/2}f|^{2}_{L^{2}}.
	\eeno
\end{lem}
\begin{proof} We divide the proof into two steps.
	
{\it Step 1: $\gamma=0$.}
By the change of variable $(u, \sigma) \rightarrow (r, \tau, \varsigma)$ with $u=r\tau$ and $\varsigma=\f{\sigma+\tau}{|\sigma+\tau|}$, we have
	\beno
	 \mathcal{Z}^{\epsilon, 0}(f) = 4 \int_{\R_{+} \times \mathbb{S}^{2} \times \mathbb{S}^{2}} b^{\epsilon}(2(\tau\cdot\varsigma)^{2} - 1)|f(r\varsigma) - f((\tau\cdot\varsigma)r\varsigma)|^{2} (\tau\cdot\varsigma) r^{2} dr d \tau d \varsigma.
	\eeno
	Let $\eta = r\varsigma$, and $\theta$ be the angle between $\tau$ and $\varsigma$. Recalling  the assumption \eqref{angular-lower-upper-bound} and $b^{\epsilon}$ in \eqref{cutoff-kernel-def}, we have
$b^{\epsilon}(2(\tau\cdot\varsigma)^{2} - 1) = b^{\epsilon}(\cos2\theta) \lesssim \theta^{-2-2s} \mathrm{1}_{3\epsilon/4 \leq \leq \pi/4}$. Observing $r^{2} dr d \tau d \varsigma = \sin\theta d\eta d\theta d\varphi$, we have
	\ben
	\mathcal{Z}^{\epsilon,0}(f) &\lesssim&  \int_{\R^{3}}\int_{\epsilon}^{\pi/4} \theta^{-1-2s}|f(\eta) - f(\eta\cos\theta)|^{2} d\eta d\theta
\nonumber	\\&\lesssim& \int_{\R^{3}}\int_{\epsilon}^{\pi/4} \theta^{-1-2s}|f(\eta) - f(\eta/\cos\theta)|^{2} d\eta d\theta
	\lesssim |W^{\epsilon}(D) f|^{2}_{L^{2}} + |W^{\epsilon} f|^{2}_{L^{2}}, \label{gamma-0-case}
	\een
where the last inequality is given by Lemma \ref{a-technical-lemma} and Remark \ref{epsilon-to-another}.

{\it Step 2: $\gamma \neq 0$.}	We  reduce the general case $\gamma \neq 0$ to the special case $\gamma = 0$. For simplicity, denote $w = |u|\frac{u^{+}}{|u^{+}|}$, then $W_{\gamma}(u) = W_{\gamma}(w)$. Then we have
	\beno
	&&\langle u\rangle^{\gamma} (f(w)-f(u^{+}))^{2} \\&=& \{[(W_{\gamma/2}f)(w)-(W_{\gamma/2}f)(u^{+})]
	 + (W_{\gamma/2}f)(u^{+})(1-W_{\gamma/2}(w)W_{-\gamma/2}(u^{+}))\}^{2}
	\\&\leq& 2 [(W_{\gamma/2}f)(w)-(W_{\gamma/2}f)(u^{+})]^{2}
	+ 2 |(W_{\gamma/2}f)(u^{+})|^{2}|1-W_{\gamma/2}(w)W_{-\gamma/2}(u^{+})|^{2}.
	\eeno
	Thus we have
	\beno
	\mathcal{Z}^{\epsilon,\gamma}(f) &\lesssim& \mathcal{Z}^{\epsilon,0}(W_{\gamma/2}f)
	+ \mathcal{B},
	\\ \mathcal{B} &:=& \int b^{\epsilon}(\frac{u}{|u|}\cdot\sigma)|(W_{\gamma/2}f)(u^{+})|^{2}|1-W_{\gamma/2}(w)W_{-\gamma/2}(u^{+})|^{2} du d\sigma
.
	\eeno
	By noticing that
	$
	|W_{\gamma/2}(w)W_{-\gamma/2}(u^{+}) - 1| \lesssim \sin^{2}\frac{\theta}{2},
	$
	we have
	$
	|\mathcal{B}| \lesssim  |W_{\gamma/2}f|^{2}_{L^{2}},
	$
	where the change of variable $u \rightarrow u^{+}$ is used.
	The desired result follows by utilizing \eqref{gamma-0-case} for $\mathcal{Z}^{\epsilon,0}(W_{\gamma/2}f)$.
\end{proof}

Next we want to show
\begin{lem} \label{key-anisotropic-result} Let $\epsilon \geq 0$ be small enough.
For suitable function $f$ defined on $\mathbb{S}^2$, there holds
\beno
\int\frac{|f(\sigma)-f(\tau)|^{2}}{|\sigma-\tau|^{2+2s}}\mathrm{1}_{|\sigma-\tau| \geq \epsilon} d\sigma d\tau + |f|^{2}_{L^{2}(\mathbb{S}^{2})}
 \sim |W^{\epsilon}((-\Delta_{\mathbb{S}^{2}})^{\frac{1}{2}})f|^{2}_{L^{2}(\mathbb{S}^{2})} + |f|^{2}_{L^{2}(\mathbb{S}^{2})}.
\eeno
As a consequence, for suitable function $f$ defined on $\mathbb{R}^3$, there holds
\ben\label{similarlemma5.6}
\int_{\mathbb{S}^2\times\mathbb{S}^2 \times \mathbb{R}_{+}}\frac{|f(r\sigma)-f(r\tau)|^{2}}{|\sigma-\tau|^{2+2s}} \mathrm{1}_{|\sigma-\tau| \geq \epsilon} r^{2}d\sigma d\tau dr + |f|^{2}_{L^{2}}
 \sim |W^{\epsilon}((-\Delta_{\mathbb{S}^{2}})^{\frac{1}{2}})f|^{2}_{L^{2}} + |f|^{2}_{L^{2}}.
\een
\end{lem}
\begin{proof} We only prove the case $\epsilon>0$ since the case $\epsilon=0$ is already proved in \cite{he2}.  By Lemma 5.4 in \cite{he2}, we have
\beno
\int\frac{|f(\sigma)-f(\tau)|^{2}}{|\sigma-\tau|^{2+2s}}\mathrm{1}_{|\sigma-\tau| \geq \epsilon} d\sigma d\tau =\sum_{l=0}^{\infty}\sum_{m=-l}^{l}(f^{m}_{l})^{2} \int\frac{|Y^{m}_{l}(\sigma)-Y^{m}_{l}(\tau)|^{2}}{|\sigma-\tau|^{2+2s}}\mathrm{1}_{|\sigma-\tau| \geq \epsilon} d\sigma d\tau,
\eeno where $f^m_l=\int_{\mathbb{S}^2} fY^m_ld\sigma$.
For simplicity, let $\mathcal{A}^{\epsilon}_{l} = \int\frac{|Y^{m}_{l}(\sigma)-Y^{m}_{l}(\tau)|^{2}}{|\sigma-\tau|^{2+2s}}\mathrm{1}_{|\sigma-\tau| \geq \epsilon} d\sigma d\tau$. We set to analyze $\mathcal{A}^{\epsilon}_{l}$.

{\it Case 1:  $\epsilon^{2} l(l+1) \leq \eta$.}  We have
\beno
\mathcal{A}^{\epsilon}_{l} = \int\frac{|Y^{m}_{l}(\sigma)-Y^{m}_{l}(\tau)|^{2}}{|\sigma-\tau|^{2+2s}}d\sigma d\tau   -\int\frac{|Y^{m}_{l}(\sigma)-Y^{m}_{l}(\tau)|^{2}}{|\sigma-\tau|^{2+2s}}\mathrm{1}_{|\sigma-\tau| \leq \epsilon} d\sigma d\tau. \eeno
Lemma 5.5 in \cite{he2} yields
\beno &&|(-\Delta_{\mathbb{S}^{2}})^{\frac{s}{2}}Y^{m}_{l}|^{2}_{L^{2}(\mathbb{S}^{2})} - |Y^{m}_{l}|^{2}_{L^{2}(\mathbb{S}^{2})} - \epsilon^{2-2s}|\nabla_{\mathbb{S}^{2}}Y^{m}_{l}|^{2}_{L^{2}(\mathbb{S}^{2})}
\\ &\lesssim& \mathcal{A}^{\epsilon}_{l}
\lesssim |(-\Delta_{\mathbb{S}^{2}})^{\frac{s}{2}}Y^{m}_{l}|^{2}_{L^{2}(\mathbb{S}^{2})}+ |Y^{m}_{l}|^{2}_{L^{2}(\mathbb{S}^{2})} + \epsilon^{2-2s}|\nabla_{\mathbb{S}^{2}}Y^{m}_{l}|^{2}_{L^{2}(\mathbb{S}^{2})}. \eeno
Choosing $\eta$ small enough, for $\epsilon^{2}l(l+1)\leq \eta$, we have
\beno
 &&[l(l+1)]^{s}(1 - 2^{-s} - \eta^{1-s})\le [l(l+1)]^{s}(1 - [l(l+1)]^{-s} - [\epsilon^{2} l(l+1)]^{1-s})\\&=&[l(l+1)]^{s}-1-\epsilon^{2-2s}l(l+1)\le\mathcal{A}^{\epsilon}_{l}
 \le  (2+\eta^{1-s})[l(l+1)]^{s}.
\eeno
In other words, in this case, we have $\mathcal{A}^{\epsilon}_{l}
\sim  [l(l+1)]^{s}.$

{\it Case 2:  $\epsilon^{2} l(l+1) \geq R^{2}$.}  Let $\zeta$ be a smooth function with compact support verifying that $0\le \zeta\le1$, $\zeta(x)=1$ if $|x|\ge2$ and $\zeta(x)=0$ if $|x|\le1$.  We have

\beno
\mathcal{A}^{\epsilon}_{l} &\ge& \int\frac{|Y^{m}_{l}(\sigma)|^{2}+|Y^{m}_{l}(\tau)|^{2}-2Y^{m}_{l}(\sigma)Y^{m}_{l}(\tau)}{|\sigma-\tau|^{2+2s}}\zeta(\epsilon^{-1}|\sigma-\tau|) d\sigma d\tau
\\&\gtrsim& \epsilon^{-2s} - \int\frac{Y^{m}_{l}(\sigma)Y^{m}_{l}(\tau)}{|\sigma-\tau|^{2+2s}}\zeta(\epsilon^{-1}|\sigma-\tau|)d\sigma d\tau
:= \epsilon^{-2s} - \mathcal{B}^{\epsilon}_{l}.
\eeno
Since $(-\Delta_{\mathbb{S}^{2}})Y^{m}_{l} = l(l+1)Y^{m}_{l}$, we have
\beno
\mathcal{B}^{\epsilon}_{l} &=& [l(l+1)]^{-1}\int\frac{(-\Delta_{\mathbb{S}^{2}})Y^{m}_{l}(\sigma)Y^{m}_{l}(\tau)}{|\sigma-\tau|^{2+2s}}\zeta(\epsilon^{-1}|\sigma-\tau|) d\sigma d\tau
\\&\leq& C[l(l+1)]^{-1}|\nabla_{\mathbb{S}^{2}}Y^{m}_{l}|_{L^{2}(\mathbb{S}^{2})}|Y^{m}_{l}|_{L^{2}(\mathbb{S}^{2})}\int_{\mathbb{S}^2}|\sigma-\tau|^{-3-2s}\mathrm{1}_{|\sigma-\tau| \geq \epsilon}d\sigma
\\&\leq& C\epsilon^{-2s}[\epsilon^{2}l(l+1)]^{-\frac{1}{2}}.
\eeno
Thus we have
$
\mathcal{A}^{\epsilon}_{l} \gtrsim \epsilon^{-2s}(1-C[\epsilon^{2}l(l+1)]^{-\frac{1}{2}}) \geq \epsilon^{-2s}(1-C/R).
$
Since $\mathcal{A}^{\epsilon}_{l}
\le  4\epsilon^{-2s}$, we obtain that $\mathcal{A}^{\epsilon}_{l}
\sim  \epsilon^{-2s}$ as long as $R$ is large enough.

{\it Case 3:  $\epsilon^{2}l(l+1)\geq \eta$.}
Note that $(N\epsilon)^{2}l(l+1)\geq N^{2}\eta$. Applying the lower bound estimate in {\it Case 2} with $\epsilon:= N\epsilon, R:= N\sqrt{\eta}$, we obtain that
\beno
\mathcal{A}^{N\epsilon}_{l} \gtrsim N^{-2s}\epsilon^{-2s}(1-\frac{C}{N\sqrt{\eta}}).
\eeno
Choosing $N$ large enough, for $\epsilon^{2}l(l+1)\geq \eta$, we have
$
\mathcal{A}^{\epsilon}_{l} \geq \mathcal{A}^{N\epsilon}_{l}  \gtrsim \epsilon^{-2s}.
$
Notice that there still holds $\mathcal{A}^{\epsilon}_{l} \lesssim \epsilon^{-2s}$ in this case. Thus we get
$\mathcal{A}^{\epsilon}_{l}\sim \epsilon^{-2s}$.

Summing up {\it Case 1} and  {\it Case 3}, we finally obtain the desired result.
\end{proof}
\begin{rmk}\label{no-difference} If we change the truncation $|\sigma-\tau| \geq \epsilon$ in Lemma \ref{key-anisotropic-result} to
$|\sigma-\tau| \geq a\epsilon$ for some $\frac{1}{3} \leq a \leq 3$, the results still hold true.
\end{rmk}

Now we are in a position to derive the anisotropic regularity from  $\mathcal{R}_\mu^{\epsilon,\gamma}(f)$ defined in \eqref{defintion-R-g}. Our key strategy is to apply the geometric decomposition in the frequency space.
\begin{prop}\label{lowerboundpart2}
There holds
\ben\label{lowerupper1}
\mathcal{R}_{\mu}^{\epsilon,0}(f) + |W^{\epsilon}f|^{2}_{L^{2}} \sim  |W^{\epsilon}((-\Delta_{\mathbb{S}^{2}})^{\frac{1}{2}})f|^{2}_{L^{2}}+ |W^{\epsilon}(D)f|^{2}_{L^{2}}+|W^{\epsilon}f|^{2}_{L^{2}}.\\
\mathcal{R}_{\mu}^{\epsilon,\gamma}(f) +|W^{\epsilon}f|^{2}_{L^{2}_{\gamma/2}} \gtrsim  |W^{\epsilon}((-\Delta_{\mathbb{S}^{2}})^{\frac{1}{2}})W_{\gamma/2}f|^{2}_{L^{2}}+ |W^{\epsilon}(D)W_{\gamma/2}f|^{2}_{L^{2}}. \label{lowerupper2}
\een
\end{prop}
\begin{proof} The proof is split into two steps.

{\it Step 1: \eqref{lowerupper1} and \eqref{lowerupper2} with $\gamma=0$.}
By Bobylev's formula, we have
\beno
\mathcal{R}_{\mu}^{\epsilon,0}(f) &=& \frac{1}{(2\pi)^{3}}\int b^{\epsilon}(\frac{\xi}{|\xi|} \cdot \sigma)(\hat{\mu}(0)|\hat{f}(\xi) - \hat{f}(\xi^{+})|^{2} + 2\Re((\hat{\mu}(0) - \hat{\mu}(\xi^{-}))\hat{f}(\xi^{+})\bar{\hat{f}}(\xi)) d\xi d\sigma
\\ &:=& \frac{\hat{\mu}(0)}{(2\pi)^{3}}\mathcal{I}_{1} + \frac{2}{(2\pi)^{3}}\mathcal{I}_{2},
\eeno
where $\xi^{+} = \frac{\xi+|\xi|\sigma}{2}$ and $\xi^{-} = \frac{\xi-|\xi|\sigma}{2}$.
Thanks to the fact $\hat{\mu}(0) - \hat{\mu}(\xi^{-}) = \int(1-\cos(v \cdot \xi^{-}))\mu(v) dv$, we have
\beno
|\mathcal{I}_{2}| &=& |\int b^{\epsilon}(\frac{\xi}{|\xi|} \cdot \sigma) (1-\cos(v \cdot \xi^{-}))\mu(v) \Re(\hat{f}(\xi^{+})\bar{\hat{f}}(\xi)) d\sigma d\xi dv |
\\ &\lesssim& (\int b^{\epsilon}(\frac{\xi}{|\xi|} \cdot \sigma) (1-\cos(v \cdot \xi^{-}))\mu(v) |\hat{f}(\xi^{+})|^{2} d\sigma d\xi dv)^{\frac{1}{2}}
\\&& \times (\int b^{\epsilon}(\frac{\xi}{|\xi|} \cdot \sigma) (1-\cos(v \cdot \xi^{-}))\mu(v) |\hat{f}(\xi)|^{2} d\sigma d\xi dv)^{\frac{1}{2}}.
\eeno
Observe that
$
1-\cos(v \cdot \xi^{-}) \lesssim |v|^{2}|\xi^{-}|^{2} = \frac{1}{4}|v|^{2}|\xi|^{2}|\frac{\xi}{|\xi|} - \sigma|^{2} \sim |v|^{2}|\xi^{+}|^{2}|\frac{\xi^{+}}{|\xi^{+}|} - \sigma|^{2},
$
thus
$
1-\cos(v \cdot \xi^{-}) \lesssim \min\{|v|^{2}|\xi|^{2}|\frac{\xi}{|\xi|} - \sigma|^{2},1\} \sim  \min\{|v|^{2}|\xi^{+}|^{2}|\frac{\xi^{+}}{|\xi^{+}|} - \sigma|^{2},1\}.
$
Note that
$
\frac{\xi}{|\xi|} \cdot \sigma  = 2(\frac{\xi^{+}}{|\xi^{+}|} \cdot \sigma)^{2} - 1,
$
by the change of variable $\xi \rightarrow \xi^{+}$, Proposition \ref{symbol} and the fact $W^{\epsilon}(|v||\xi|) \lesssim W^{\epsilon}(|v|)W^{\epsilon}(|\xi|)$, we have
\beno
|\mathcal{I}_{2}| \lesssim \int (W^{\epsilon})^{2}(|v||\xi|)|\hat{f}(\xi)|^{2}\mu(v)dvd\xi
\lesssim |W^{\epsilon}\mu^{\frac{1}{2}}|^{2}_{L^{2}} |W^{\epsilon}(D)f|^{2}_{L^{2}}\lesssim |W^{\epsilon}(D)f|^{2}_{L^{2}}.
\eeno
Now we set to estimate $\mathcal{I}_{1}$. By the geometric decomposition
\beno
\hat{f}(\xi) - \hat{f}(\xi^{+}) = \hat{f}(\xi) - \hat{f}(|\xi|\frac{\xi^{+}}{|\xi^{+}|})+ \hat{f}(|\xi|\frac{\xi^{+}}{|\xi^{+}|}) - \hat{f}(\xi^{+}),
\eeno
we have
\beno
\mathcal{I}_{1} &=& \int b^{\epsilon}(\frac{\xi}{|\xi|} \cdot \sigma)|\hat{f}(\xi) - \hat{f}(\xi^{+})|^{2} d\xi d\sigma
\\&\geq& \frac{1}{2} \int b^{\epsilon}(\frac{\xi}{|\xi|} \cdot \sigma)|\hat{f}(\xi) - \hat{f}(|\xi|\frac{\xi^{+}}{|\xi^{+}|})|^{2} d\xi d\sigma
- \int b^{\epsilon}(\frac{\xi}{|\xi|} \cdot \sigma)|\hat{f}(|\xi|\frac{\xi^{+}}{|\xi^{+}|}) - \hat{f}(\xi^{+})|^{2} d\xi d\sigma
\\&:=& \frac{1}{2}\mathcal{I}_{1,1} - \mathcal{I}_{1,2}.
\eeno
Let $\xi = r \tau$ and $\varsigma = \frac{\tau+\sigma}{|\tau+\sigma|}$, then $\frac{\xi}{|\xi|} \cdot \sigma = 2(\tau\cdot\varsigma)^{2} - 1$ and $|\xi|\frac{\xi^{+}}{|\xi^{+}|} = r \varsigma$. For the change of variable $(\xi, \sigma) \rightarrow (r, \tau, \varsigma)$, one has
$
d\xi d\sigma = 4  (\tau\cdot\varsigma) r^{2} dr d \tau d \varsigma.
$ Let $\theta$ be the angle between $\tau$ and $\sigma$, then $2 \sin\frac{\theta}{2} = |\tau-\sigma|, |\tau - \varsigma| = 2(1-\cos \frac{\theta}{2})$ and thus
$\sin\frac{\theta}{2} = \frac{1}{2} |\tau-\sigma| \leq |\tau - \varsigma| \leq |\tau-\sigma| = 2\sin\frac{\theta}{2}$.
Therefore
\ben \label{equivalent-to-tau-var}
|\tau - \varsigma|^{-2-2s} \mathrm{1}_{|\tau - \varsigma| \geq \frac{8}{3}\epsilon} \lesssim  b^{\epsilon}(\cos\theta) \lesssim |\tau - \varsigma|^{-2-2s} \mathrm{1}_{|\tau - \varsigma| \geq \frac{3}{4}\epsilon}.
\een
By \eqref{similarlemma5.6} in Lemma \ref{key-anisotropic-result} and Remark \ref{no-difference}, we have
\ben \label{lower-I-11}
\mathcal{I}_{1,1}+|f|_{L^2}^2 &=& 4 \int b^{\epsilon}(2(\tau\cdot\varsigma)^{2} - 1)|\hat{f}(r\tau) - \hat{f}(r\varsigma)|^{2} (\tau\cdot\varsigma) r^{2} dr d \tau d \varsigma+|f|_{L^2}^2
\nonumber \\&\gtrsim& \int \frac{|\hat{f}(r\tau) - \hat{f}(r\varsigma)|^{2}}{|\tau - \varsigma|^{2+2s}}\mathrm{1}_{|\tau-\varsigma| \geq 8\epsilon/3}  r^{2} dr d \tau d \varsigma+|f|_{L^2}^2
\nonumber \\&\sim& |W^{\epsilon}((-\Delta_{\mathbb{S}^{2}})^{\frac{1}{2}})\hat{f}|^{2}_{L^{2}}+ |\hat{f}|^{2}_{L^{2}}
\sim |W^{\epsilon}((-\Delta_{\mathbb{S}^{2}})^{\frac{1}{2}})f|^{2}_{L^{2}}+|f|^{2}_{L^{2}}.
\een
Here we use Lemma \ref{comWep} and Plancherel's theorem in the last line. Similarly, by \eqref{equivalent-to-tau-var}, \eqref{similarlemma5.6} in Lemma \ref{key-anisotropic-result} and Remark \ref{no-difference}, we have
\beno
\mathcal{I}_{1,1}+|f|_{L^2}^2 \lesssim |W^{\epsilon}((-\Delta_{\mathbb{S}^{2}})^{\frac{1}{2}})f|^{2}_{L^{2}}+|f|^{2}_{L^{2}}.
\eeno
By Lemma \ref{gammanonzerotozero}, there holds
\ben \label{upper-bound-I-12} \mathcal{I}_{1,2} \lesssim |W^{\epsilon}(D)f|^{2}_{L^{2}} + |W^{\epsilon}f|^{2}_{L^{2}}. \een
Patching together \eqref{lower-I-11}, \eqref{upper-bound-I-12} and Proposition \ref{lowerboundpart1-general-g}, we get
the $\gtrsim$ direction of \eqref{lowerupper1},
\ben \label{lower-bound-case-gamma-0}
\mathcal{R}_{\mu}^{\epsilon,0}(f) + |W^{\epsilon}f|^{2}_{L^{2}} \gtrsim  |W^{\epsilon}((-\Delta_{\mathbb{S}^{2}})^{\frac{1}{2}})f|^{2}_{L^{2}}+ |W^{\epsilon}(D)f|^{2}_{L^{2}}+|W^{\epsilon}f|^{2}_{L^{2}}. \een
The $\lesssim$ direction of \eqref{lowerupper1} follows easily from $\mathcal{R}_{\mu}^{\epsilon,0}(f) \lesssim  \mathcal{I}_{1,1} + \mathcal{I}_{1,2} + |\mathcal{I}_{2}|.$

{\it Step 2: \eqref{lowerupper2} with $\gamma \neq 0$.}
Thanks to Lemma 3.4 in \cite{he2} which reads that
\beno
\mathcal{R}_\mu^{\epsilon,\gamma}(f) + |f|^{2}_{L^{2}_{\gamma/2}}  \gtrsim  \mathcal{R}_{\mu}^{\epsilon,0}(W_{\gamma/2}f),
\eeno
we obtain the desired result by using \eqref{lower-bound-case-gamma-0}. The proof is complete.
\end{proof}

\subsubsection{Lower bound of $\langle \mathcal{L}^{\epsilon}f,f\rangle_v$} The aim of this subsection is to prove that $\langle \mathcal{L}^{\epsilon}f,f\rangle_v$ is bounded below by $\mathcal{R}_\mu^{\epsilon,\gamma}(f) + \mathcal{M}^{\epsilon,\gamma}(f) -C|f|^{2}_{L^{2}_{\gamma/2}}$, and hence by $|f|_{\epsilon,\gamma/2}^2-C|f|^{2}_{L^{2}_{\gamma/2}}$ for some constant $C$.
\begin{thm} \label{equivalencenorm} For $-3<\gamma \leq 1$, there holds
\beno
\langle \mathcal{L}^{\epsilon}f, f\rangle_v + |f|^{2}_{L^{2}_{\gamma/2}} \gtrsim  |f|_{\epsilon,\gamma/2}^2.
\eeno
\end{thm}
\begin{proof} We proceed in the spirit of \cite{amuxy4}.
Recalling \eqref{DefLep} and $(a+b)^{2} \geq a^{2}/2 - b^{2}$, there holds
\ben\label{L1lowerbound1}
&& 2 \langle \mathcal{L}_{1}^{\epsilon}f,f \rangle_{v} =\int B^{\epsilon} (\mu_{*}^{\frac{1}{2}}f -
(\mu^{\frac{1}{2}})_{*}^{\prime}f^{\prime})^{2}dv dv_{*} d\sigma
\nonumber\\&=& \int B^{\epsilon} (\mu_{*}^{\frac{1}{2}}(f - f^{\prime}) + (\mu_{*}^{\frac{1}{2}}- (\mu^{\frac{1}{2}})_{*}^{\prime})f^{\prime})^{2} dv dv_{*} d\sigma
\ge \frac{1}{2} \mathcal{R}_\mu^{\epsilon,\gamma}(f) - \mathcal{M}^{\epsilon,\gamma}(f).
\een
On the other hand, one has $$2 \langle \mathcal{L}_{1}^{\epsilon}f,f \rangle_v = \mathcal{R}_\mu^{\epsilon,\gamma}(f) + \mathcal{M}^{\epsilon,\gamma}(f) + 2\int B^{\epsilon} (\mu_{*}^{\frac{1}{2}}- (\mu^{\frac{1}{2}})_{*}^{\prime})\mu_{*}^{\frac{1}{2}}(f - f^{\prime})f^{\prime} dv dv_{*} d\sigma.$$
From which together with the fact $2(a-b)b = a^{2}-b^{2}-(a-b)^{2}$, one gets
\beno
&&2(\mu_{*}^{\frac{1}{2}}- (\mu^{\frac{1}{2}})_{*}^{\prime})\mu_{*}^{\frac{1}{2}}(f - f^{\prime})f^{\prime} =\frac{1}{2}(f^{2} -f^{\prime 2}  - (f - f^{\prime})^{2}) (\mu_{*} - \mu_{*}^{\prime} + (\mu_{*}^{\frac{1}{2}}- (\mu^{\frac{1}{2}})_{*}^{\prime})^{2})
\\
&=& \frac{1}{2}(f^{2} -f^{\prime 2})(\mu_{*} - \mu_{*}^{\prime}) - \frac{1}{2}(f - f^{\prime})^{2}(\mu_{*}^{\frac{1}{2}}- (\mu^{\frac{1}{2}})_{*}^{\prime})^{2}
+\frac{1}{2}(f^{2} -f^{\prime 2})(\mu_{*}^{\frac{1}{2}}- (\mu^{\frac{1}{2}})_{*}^{\prime})^{2}\\&& -\frac{1}{2}(f - f^{\prime})^{2}(\mu_{*}- \mu_{*}^{\prime})
:= A_{1} + A_{2} + A_{3} + A_{4}.
\eeno
By the change of variable $(v,v_{*}) \rightarrow (v^{\prime},v_{*}^{\prime})$, cancellation lemma in \cite{advw} and \eqref{mu-cancel-sigularity}, there hold
\beno |\int B^{\epsilon} A_{1} dv dv_{*} d\sigma | = |\int B^{\epsilon} \mu_{*} (f^{2} -f^{\prime 2}) dv dv_{*} d\sigma | \leq C |f|^{2}_{L^{2}_{\gamma/2}},\\
\int B^{\epsilon} A_{3} dv dv_{*} d\sigma = \int B^{\epsilon} A_{4} dv dv_{*} d\sigma = 0.
\eeno
Note that
\beno
 \int B^{\epsilon} A_{2} dv dv_{*} d\sigma = - \int B^{\epsilon} \mu_{*} (f -f^{\prime})^{2} dv dv_{*} d\sigma +
\int B^{\epsilon} \mu_{*}^{\frac{1}{2}}(\mu^{\frac{1}{2}})_{*}^{\prime}(f -f^{\prime})^{2} dv dv_{*} d\sigma
\ge - \mathcal{R}_\mu^{\epsilon,\gamma}(f).
\eeno
Patching the above estimates, we infer that
$
2 \langle \mathcal{L}_{1}^{\epsilon}f,f \rangle_v \geq \mathcal{M}^{\epsilon,\gamma}(f) - C |f|^{2}_{L^{2}_{\gamma/2}}$.
From which together with  \eqref{L1lowerbound1}, we have
\beno
5 \langle \mathcal{L}_{1}^{\epsilon}f,f \rangle_v \geq \frac{1}{2}\mathcal{R}_\mu^{\epsilon,\gamma}(f) +  \frac{1}{2}\mathcal{M}^{\epsilon,\gamma}(f) - \frac{3}{2}C |f|^{2}_{L^{2}_{\gamma/2}}
\gtrsim |f|_{\epsilon,\gamma/2}^2-C|f|^{2}_{L^{2}_{\gamma/2}}.
\eeno
 Here we use Proposition \ref{lowerboundpart1} and \eqref{lowerupper2} in Proposition \ref{lowerboundpart2}. By a similar proof as that of
 Lemma 2.15 in \cite{amuxy4}, there holds
 \ben \label{l2-upper-bound}
 |\langle \mathcal{L}_{2}^{\epsilon}g, h\rangle_v| \lesssim |\mu^{1/10^{3}}g|_{L^{2}} |\mu^{1/10^{3}}h|_{L^{2}} \lesssim |g|_{L^{2}_{\gamma/2}}|h|_{L^{2}_{\gamma/2}}.
 \een
Recalling $\mathcal{L}^\epsilon=\mathcal{L}^\epsilon_1+\mathcal{L}^\epsilon_2$, we finish the proof.
\end{proof}

As a result, we can get the
coercivity estimate of $\mathcal{L}^{\epsilon}$ on the perpendicular space $\mathcal{N}^{\perp}$.

\begin{prop}\label{coercvityforLep} There holds
\beno
\langle \mathcal{L}^{\epsilon}f, f\rangle_v  \gtrsim |(\mathbb{I}-\mathbb{P})f|^{2}_{\epsilon,\gamma/2}.
\eeno
Here $\mathbb{I}$ stands for the identity operator.
\end{prop}
\begin{proof}
By \cite{mouhot,mr}, there holds $\langle \mathcal{L}^{\epsilon}f, f\rangle_v  \gtrsim |(\mathbb{I}-\mathbb{P})f|^{2}_{L^{2}_{\gamma/2}}.$
By the definition of  $\mathbb{P}$ and Theorem \ref{equivalencenorm}, we have
\beno
\langle \mathcal{L}^{\epsilon}f, f\rangle_v  = \langle \mathcal{L}^{\epsilon}(\mathbb{I}-\mathbb{P})f, (\mathbb{I}-\mathbb{P})f\rangle_v   \gtrsim |(\mathbb{I}-\mathbb{P})f|^{2}_{\epsilon,\gamma/2} - |(\mathbb{I}-\mathbb{P})f|^{2}_{L^{2}_{\gamma/2}}.
\eeno
Making a suitable combination of the two estimates, we get the desired result.
\end{proof}

\subsection{Upper bound of the nonlinear term}
In this subsection, we will give the upper bound for the nonlinear term $ \Gamma^{\epsilon}(g,h)$.
We prove it by duality.
Observe that
\ben \label{Gamma-to-Q-I}
\langle \Gamma^{\epsilon}(g,h), f\rangle_v &=& \langle Q^{\epsilon}(\mu^{\frac{1}{2}}g,h), f\rangle_v + \mathcal{I}(g,h,f),
 \\ \label{definition-of-Q}
\mathcal{I}(g,h,f)&:=& \int b^{\epsilon}(\cos\theta)|v-v_{*}|^{\gamma}((\mu^{\frac{1}{2}})_{*}^{\prime} - \mu_{*}^{\frac{1}{2}})g_{*} h f^{\prime} d\sigma dv_{*} dv.
\een
We will estimate $Q^\epsilon$ in sub-subsection \ref{estimate-of-Q} and $\mathcal{I}(g,h,f)$ in sub-subsection \ref{estimate-of-I}.
\subsubsection{Upper bounds for the collision operator $Q^\epsilon$} \label{estimate-of-Q}
We  perform the decomposition:
\ben \label{Q-into-two-parts}
\langle Q^{\epsilon}(g,h), f\rangle_v =  \langle Q^{\epsilon}_{-1}(g,h), f\rangle_v + \langle Q^{\epsilon}_{\geq 0}(g,h), f\rangle_v,
\een
where
\beno
\langle Q^{\epsilon}_{-1}(g,h), f\rangle_v := \int b^{\epsilon}(\cos\theta)|v-v_{*}|^{\gamma}\phi(v-v_{*}) g_{*} h (f^{\prime}-f) d\sigma dv_{*} dv,
\\
\langle Q^{\epsilon}_{\geq 0}(g,h), f\rangle_v := \int b^{\epsilon}(\cos\theta)|v-v_{*}|^{\gamma} (1-\phi)(v-v_{*})g_{*} h (f^{\prime}-f) d\sigma dv_{*} dv.
\eeno
Here $\phi$ is defined in \eqref{function-phi-psi}.

To give an estimate for $Q_{-1}^\epsilon$, we begin with two lemmas.
\begin{lem}\label{hardylittlewoodsoblev} Let
$
A :=  \int |v-v_{*}|^{\gamma} \phi(v-v_{*}) g_{*} h f dv d v_{*},
B := \epsilon^{2s} \int b^{\epsilon}(\cos\theta)|v-v_{*}|^{\gamma}\phi(v-v_{*})g_{*} h f^{\prime} d\sigma dv d v_{*}.
$ The following statements are valid.
\begin{itemize}
\item If $\gamma>-\frac{3}{2}$,
$ |A|+|B| \lesssim  |g|_{L^{2}} |h|_{L^{2}}|f|_{L^{2}}$.
\item If $\gamma=-\f32$, for any $\eta>0$, there exists a constant $C_{\eta}$ such that
\beno |A|+|B|\le\left\{\begin{aligned} & C_{\eta} |g|_{H^\eta}|h|_{L^{2}}|f|_{L^{2}};\\&
C_{\eta} (|g|_{L^1}+|g|_{L^2})|h|_{H^{\eta}}|f|_{L^{2}}. \end{aligned}\right.\eeno
\item If $-3<\gamma< -\frac{3}{2}$, for any $\eta>0$, there exists a constant $C_{\eta}$ such that
\beno
|A|+|B| \lesssim \left\{\begin{aligned} & C_{\eta} |g|_{H^{\eta-\frac{3}{2}-\gamma}}|h|_{L^{2}}|f|_{L^{2}}
;\\
& |g|_{H^{s_{1}}}|h|_{H^{s_{2}}}|f|_{H^{s_{3}}}.\end{aligned}\right.
\eeno
Here the constants $s_1, s_2, s_3 \geq 0$ verify $s_{1}+s_{2}+s_{3}=-\gamma-\frac{3}{2}, s_{2}+s_{3}>0$.
\end{itemize}
\end{lem}
\begin{proof} We first handle the term $A$. For the case of
$\gamma>-\frac{3}{2}$, the desired result comes from the inequality
\beno \int |v-v_{*}|^{\gamma} \phi(v-v_{*}) |g_{*}| d v_{*} \lesssim |g|_{L^{2}}.\eeno

For the case of $\gamma=-\f32$, the first result follows the Hardy's  inequality
\beno \int |v-v_*|^{-\f32}\phi(v-v_*) |g_*|dv_*\le C_{\eta}\bigg(\int |v-v_*|^{-2\eta} |g_*|^2dv_*\bigg)^{\f12}\lesssim C_\eta |g|_{H^\eta}.   \eeno
The second result follows the Hardy-Littlewood-Sobolev inequality,  Sobolev embedding theorem and interpolation inequality. Indeed, we have
\beno |A|\lesssim |g|_{L^{p_1}}|h|_{L^{p_2}}|f|_{L^2} \lesssim C_{\eta} (|g|_{L^1}+|g|_{L^2})|h|_{H^{\eta}}|f|_{L^{2}},  \eeno
where $\frac{1}{p_{1}} + \frac{1}{p_2}=1, \frac{\eta}{3} = \frac{1}{2} - \frac{1}{p_{2}}$ with $p_2>2, 1<p_1<2$.

For the case of $-3<\gamma< -\frac{3}{2}$, the first result follows the Hardy's  inequality
\beno \int |v-v_*|^{\gamma}\phi(v-v_*) |g_*|dv_*\le C_{\eta}\bigg(\int |v-v_*|^{2\gamma+3-2\eta}  |g_*|^2dv_*\bigg)^{\f12}\lesssim C_\eta |g|_{H^{\eta-\f32-\gamma}}.   \eeno
The second result follows the Hardy-Littlewood-Sobolev inequality and Sobolev embedding theorem
\beno
|A| \lesssim |g|_{L^{p_{1}}}|h|_{L^{p_{2}}}|f|_{L^{p_{3}}} \lesssim |g|_{H^{s_{1}}}|h|_{H^{s_{2}}}|f|_{H^{s_{3}}},
\eeno
where $\f{-\gamma}3+\f1{p_1}+\f1{p_2}+\f1{p_3}=2, p_{i} \geq 2, \frac{1}{p_{2}} + \frac{1}{p_{3}} < 1,  \frac{s_{i}}{3} = \frac{1}{2} - \frac{1}{p_{i}}$ and thus
 $s_{1}+s_{2} + s_{3} = -\frac{3}{2}-\gamma, s_{2}+s_{3}>0$.

Now we point out how to derive the same estimates for $B$. From the above proof for $A$, thanks to the change of variable $v \rightarrow v^{\prime}$ and the estimate $\epsilon^{2s}\int b^\epsilon(\cos\theta) d\sigma \lesssim 1$,
we only need to prove that the Hardy-Littlewood-Sobolev inequality is still valid for $B$.
To this end, we observe that for $\f{-\gamma}3+\f1{p_1}+\f1{r}=2$ and $\f1{p_2}+\f1{p_3}=\f1{r}$,
\beno |B|&\lesssim& \bigg(\epsilon^{2s}\int b^\epsilon(\cos\theta) |v-v_*|^{\gamma}\phi(v-v_*)|g_*||h|^{\f{p_2}{r}}d\sigma dv_* dv\bigg)^{\f{r}{p_2}}\\
&&\times\bigg(\epsilon^{2s}\int b^\epsilon(\cos\theta) |v-v_*|^{\gamma}\phi(v-v_*)|g_*||f'|^{\f{p_3}{r}}d\sigma dv_* dv\bigg)^{\f{r}{p_3}}
\lesssim |g|_{L^{p_{1}}}|h|_{L^{p_{2}}}|f|_{L^{p_{3}}}.
\eeno Then we conclude the results for $B$ by copying the same argument used for $A$.
\end{proof}

\begin{lem}\label{aftercancellation}
Set
$
A :=  \int |v-v_{*}|^{\gamma}g_{*} h f dv d v_{*}
$ and    $
B := \epsilon^{2s} \int b^{\epsilon}(\cos\theta)|v-v_{*}|^{\gamma} g_{*} h f^{\prime} d\sigma dv d v_{*}$.
The following statements are valid.
\begin{itemize}
	\item If $\gamma \geq 0$, then
$ |A|+|B| \lesssim |g|_{L^{1}_{\gamma}}|h|_{L^{2}_{\gamma/2}}|f|_{L^{2}_{\gamma/2}}.$
\item If $-\frac{3}{2}<\gamma<0$, then
$ |A|+|B| \lesssim  (|g|_{L^{1}_{|\gamma|}}+|g|_{L^{2}_{|\gamma|}})|h|_{L^{2}_{\gamma/2}}|f|_{L^{2}_{\gamma/2}}.$
\item If $\gamma=-\f32$, for any $\eta>0$, there exists a constant $C_{\eta}$ such that
\beno |A|+|B|\le\left\{\begin{aligned} & C_{\eta} (|g|_{L^1_{|\gamma|}}+ |g|_{H^\eta_{|\gamma|}})|h|_{L^{2}_{\gamma/2}}|f|_{L^{2}_{\gamma/2}};\\&
C_{\eta} (|g|_{L^1_{|\gamma|}}+|g|_{L^2_{|\gamma|}})|h|_{H^{\eta}_{\gamma/2}}|f|_{L^{2}_{\gamma/2}}.\end{aligned}\right.\eeno
\item If $-3<\gamma< -\frac{3}{2}$, for any $\eta>0$, there exists a constant $C_{\eta}$ such that\beno
|A|+|B| \lesssim \left\{\begin{aligned} & C_{\eta} (|g|_{L^1_{|\gamma|}}+|g|_{H^{\eta-\frac{3}{2}-\gamma}_{|\gamma|}})|h|_{L^{2}_{\gamma/2}}|f|_{L^{2}_{\gamma/2}};\\
& (|g|_{L^1_{|\gamma|}}+
|g|_{H^{s_{1}}_{|\gamma|}})|h|_{H^{s_{2}}_{\gamma/2}}|f|_{H^{s_{3}}_{\gamma/2}}. \end{aligned}\right.
\eeno
Here $s_1, s_2, s_3 \geq 0$ are constants verifying $s_{1}+s_{2}+s_{3}=-\gamma-\frac{3}{2}, s_{2}+s_{3}>0$.
\end{itemize}
 \end{lem}
\begin{proof} Let $G = gW_{|\gamma|}, H = hW_{\gamma/2}, F = fW_{\gamma/2}$. Then we have \beno
|A| =  |\int |v-v_{*}|^{\gamma}\langle v_{*}\rangle^{-|\gamma|}\langle v\rangle^{-\gamma}G_{*} H F dv d v_{*}|
 \lesssim \int (1+\mathrm{1}_{\gamma \leq 0}|v-v_{*}|^{\gamma}\phi(v-v_*))|G_{*} H F| dv d v_{*}
\eeno
Then the estimates of $A$ follow from Lemma \ref{hardylittlewoodsoblev}. Since $|v-v_{*}| \sim |v^{\prime}-v_{*}|$,
by a similar argument, we can conclude the results for $B$.
\end{proof}

Now we are ready to give the following  upper bounds for $Q^\epsilon_{-1}$.
\begin{prop}\label{ubqepsilonsingular}
For any  $\eta>0$, the following estimates are valid.
\begin{itemize}
\item If $\gamma>-\frac{3}{2}$,
$|\langle Q^{\epsilon}_{-1}(g,h), f\rangle_v| \lesssim |g|_{L^{2}_{|\gamma|}}|W^{\epsilon}(D)h|_{L^{2}_{\gamma/2}}|W^{\epsilon}(D)f|_{L^{2}_{\gamma/2}}.$

\item If $\gamma=-\f32$, $|\langle Q^{\epsilon}_{-1}(g,h), f\rangle_v| \lesssim (|g|_{L^1_{|\gamma|}}+|g|_{H^{s_{1}}_{|\gamma|}})|W^{\epsilon}(D)h|_{H^{s_{2}}_{\gamma/2}}
    |W^{\epsilon}(D)f|_{L^2_{\gamma/2}}.$ Here $(s_1, s_2) = (0, \eta)$ or $(\eta, 0)$.
\item If $-3<\gamma< -\frac{3}{2}$,
$
|\langle Q^{\epsilon}_{-1}(g,h), f\rangle_v| \lesssim |g|_{H^{s_{1}}_{|\gamma|}}|W^{\epsilon}(D)h|_{H^{s_{2}}_{\gamma/2}}|W^{\epsilon}(D)f|_{H^{s_{3}}_{\gamma/2}}.
$ Here $s_1, s_2, s_3 \geq 0$ are constants verifying either $s_{1}+s_{2}+s_{3}=-\gamma-\frac{3}{2}, s_{2}+s_{3}>0$ or  $s_1=-\gamma-\frac{3}{2}+\eta, s_2=s_3=0$.
\end{itemize}
In the case of $\gamma=-\f32$ and $-3<\gamma< -\frac{3}{2}$,
the $\lesssim$ could produce a constant depending on $\eta$ on the right-hand sides.
\end{prop}
\begin{proof} We divide the proof into two steps.

{\it Step 1: Estimates without weight.} Following the proof of Theorem 1.1 in \cite{he2}, we conclude that
	\ben\label{Q-1} | Q^{\epsilon}_{-1}(g,h), f\rangle_v|\lesssim |g|_{L^2}|h|_{H^a}|f|_{H^b}, \een
	where $a+b=2s$ with $a,b\in [0,2s]$. Recalling \eqref{defphilh}, we use the following decomposition
 \beno
\langle Q^{\epsilon}_{-1}(g,h), f\rangle_v =  \langle Q^{\epsilon}_{-1}(g,h_\phi+h^\phi), f_\phi+f^\phi\rangle_{v}.
\eeno
Using \eqref{Q-1},
we have
$ |\langle  Q^{\epsilon}_{-1}(g,h_\phi), f_\phi\rangle_{v}|\lesssim |g|_{L^2}|h_\phi|_{H^s}|f_\phi|_{H^s},
|\langle  Q^{\epsilon}_{-1}(g,h_\phi), f^\phi\rangle_{v}|\lesssim |g|_{L^2}|h_\phi|_{H^{2s}}|f^\phi|_{L^2}$ and $
 |\langle  Q^{\epsilon}_{-1}(g,h^\phi), f_\phi\rangle_{v}|\lesssim |g|_{L^2}|h^\phi|_{L^2}|f_\phi|_{H^{2s}}.$
Thanks to the fact $|h_\phi|_{H^{2s}}\lesssim \epsilon^{-s}|h_\phi|_{H^{s}}$, we have
\beno |\langle  Q^{\epsilon}_{-1}(g,h_\phi), f_\phi\rangle_{v}|+
|\langle  Q^{\epsilon}_{-1}(g,h_\phi), f^\phi\rangle_{v}|+
|\langle  Q^{\epsilon}_{-1}(g,h^\phi), f_\phi\rangle_{v}|\lesssim  |g|_{L^2}|W^\epsilon(D)h|_{L^2}|W^\epsilon(D)f|_{L^2}.\eeno
From which together Lemma \ref{hardylittlewoodsoblev} to deal with $
\langle Q^{\epsilon}_{-1}(g,h^\phi), f^\phi\rangle_{v}$, we conclude that,
\beno
\gamma>-\frac{3}{2}, & \quad
|\langle Q^{\epsilon}_{-1}(g,h), f\rangle_{v}|
\lesssim |g|_{L^2}|W^\epsilon(D)h|_{L^2}|W^\epsilon(D)f|_{L^2};
\\
\gamma=-\f32, & \quad |\langle Q^{\epsilon}_{-1}(g,h), f\rangle_{v}|
\lesssim (|g|_{L^1}+|g|_{H^{s_{1}}})|W^\epsilon(D)h|_{H^{s_{2}}}|W^\epsilon(D)f|_{L^2};
\\
-3<\gamma< -\frac{3}{2}, & \quad
|\langle Q^{\epsilon}_{-1}(g,h), f\rangle_{v}|
\lesssim  |g|_{H^{s_{1}}}|W^\epsilon(D)h|_{H^{s_{2}}}|W^\epsilon(D)f|_{H^{s_{3}}}.
\eeno
That is, the estimates are valid if we take $\gamma=0$ on the right-hand side. In the next step, we recover the weights by using some commutator estimates.

{\it Step 2: Estimates with weight.} We recall that
\beno
\langle Q^{\epsilon}_{-1}(g,h), f\rangle_v &=& \sum_{j\geq N_{0} - 1} \langle Q^{\epsilon}_{-1}(\varphi_{j}g,\tilde{\varphi}_{j}h), \tilde{\varphi}_{j}f\rangle_v + \sum_{j\leq N_{0} - 2} \langle Q^{\epsilon}_{-1}(\varphi_{j}g,\mathcal{U}_{N_{0}-1}h), \mathcal{U}_{N_{0}-1}f\rangle_v
\\&:=&\mathcal{A}_{1} + \mathcal{A}_{2},
\eeno
where $\tilde{\varphi}_{j} := \sum_{k \geq -1,|k-j| \leq N_{0}} \varphi_{k}$ and $\mathcal{U}_{N_{0}-1} := \sum_{-1 \leq k \leq N_{0}-1} \varphi_{k} $ for some fixed integer $N_{0} \geq 4$. We only consider the most difficult case $-3<\gamma< -\frac{3}{2}$. In this case, by {\it Step 1}, we have
\beno |\mathcal{A}_{1}| \lesssim  \sum_{j\geq N_{0} - 1} |\langle D \rangle^{s_{1}}\varphi_{j}g|_{L^{2}}|\langle D \rangle^{s_{2}}W^{\epsilon}(D)\tilde{\varphi}_{j}h|_{L^{2}}
|\langle D \rangle^{s_{3}}W^{\epsilon}(D)\tilde{\varphi}_{j}f|_{L^{2}} := \sum_{j\geq N_{0} - 1} \mathcal{A}_{1,j}.\eeno
For simplicity, we write $\mathcal{A}_{1,j} =  \mathcal{B}_{j}\mathcal{C}_{j}\mathcal{D}_{j} $, where
$ \mathcal{B}_{j} := 2^{-j}|\langle D \rangle^{s_{1}}2^{(- \gamma + 1) j}\langle \cdot \rangle^{\gamma}\varphi_{j}\langle \cdot \rangle^{-\gamma}g|_{L^{2}},   \mathcal{C}_{j} := 2^{-j}|\langle D \rangle^{s_{2}}\\W^{\epsilon}(D)2^{(\gamma/2+1)j}\langle \cdot \rangle^{-\gamma/2}\tilde{\varphi}_{j}\langle \cdot \rangle^{\gamma/2}h|_{L^{2}}$ and $ \mathcal{D}_{j} := 2^{-j}|\langle D \rangle^{s_{3}}W^{\epsilon}(D)2^{(\gamma/2+1)j}\langle \cdot \rangle^{-\gamma/2}\tilde{\varphi}_{j}\langle \cdot \rangle^{\gamma/2}f|_{L^{2}}$.
 Thanks to $2^{(- \gamma + 1) j}\langle \cdot \rangle^{\gamma}\varphi_{j} \in S^{1}_{1,0}$ and $\langle \cdot \rangle^{s_{1}} \in S^{s_{1}}_{1,0}$,  Lemma \ref{operatorcommutator1} yields
\beno \mathcal{B}_{j} \lesssim | \varphi_{j}\langle D \rangle^{s_{1}}\langle \cdot \rangle^{-\gamma}g|_{L^{2}} + 2^{-j}|\langle \cdot \rangle^{-\gamma}g|_{H^{s_{1}-1}}. \eeno
Similarly, Lemma \ref{operatorcommutator1} yields
\beno \mathcal{C}_{j} \lesssim | \tilde{\varphi}_{j}\langle D \rangle^{s_{2}} W^{\epsilon}(D)\langle \cdot \rangle^{\gamma/2}h|_{L^{2}} + 2^{-j}|\langle \cdot \rangle^{\gamma/2}h|_{H^{s_{2}+s-1}},
\\ \mathcal{D}_{j} \lesssim| \tilde{\varphi}_{j}\langle D \rangle^{s_{3}}W^{\epsilon}(D)\langle \cdot \rangle^{\gamma/2}f|_{L^{2}}+ 2^{-j}|\langle \cdot \rangle^{\gamma/2}f|_{H^{s_{3}+s-1}}. \eeno
Thus it is not difficult to conclude that
\beno
|\mathcal{A}_{1}| \lesssim |\langle D \rangle^{s_{1}}\langle \cdot \rangle^{-\gamma}g|_{L^{2}}|\langle D \rangle^{s_{2}}W^{\epsilon}(D)\langle \cdot \rangle^{\gamma/2}h|_{L^{2}}|\langle D \rangle^{s_{3}}W^{\epsilon}(D)\langle \cdot \rangle^{\gamma/2}f|_{L^{2}}.
\eeno
The term $\mathcal{A}_{2}$ is much easier since it has only finite terms. Finally, we have
\beno
|\langle Q^{\epsilon}_{-1}(g,h), f\rangle_v| \lesssim |\langle D \rangle^{s_{1}}\langle \cdot \rangle^{-\gamma}g|_{L^{2}}|\langle D \rangle^{s_{2}}W^{\epsilon}(D)\langle \cdot \rangle^{\gamma/2}h|_{L^{2}}|\langle D \rangle^{s_{3}}W^{\epsilon}(D)\langle \cdot \rangle^{\gamma/2}f|_{L^{2}}.
\eeno

For the case $\gamma\ge-\frac{3}{2}$, we can repeat the above procedure to get the desired results. We finish the proof
with the help of Lemma \ref{func}.
\end{proof}

To give the upper bound for $Q^{\epsilon}_{\geq 0}$, we need the next two lemmas.

\begin{lem}\label{crosstermsimilar}
	Let
	$
	\mathcal{Y}^{\epsilon,\gamma}(h,f) := \int b^{\epsilon}(\frac{u}{|u|}\cdot\sigma)\langle u \rangle^{\gamma} h(u)[f(u^{+}) - f(|u|\frac{u^{+}}{|u^{+}|})] du d\sigma,
	$
	then
	\beno
	|\mathcal{Y}^{\epsilon,\gamma}(h,f)| \lesssim (|W^{\epsilon}W_{\gamma/2}h|_{L^{2}}+|W^{\epsilon}(D)W_{\gamma/2}h|_{L^{2}})
(|W^{\epsilon}W_{\gamma/2}f|_{L^{2}}+|W^{\epsilon}(D)W_{\gamma/2}f|_{L^{2}}).
	\eeno
\end{lem}
\begin{proof}We divide the proof into two steps.
	
{\it Step 1: $\gamma = 0$.} For ease of notation, we denote $\mathcal{Y} = \mathcal{Y}^{\epsilon, 0}(h,f)$. The proof is similar to that of Lemma \ref{a-technical-lemma}.
First applying dyadic decomposition in the phase space, we have
	\beno
	\mathcal{Y}= \sum_{k=-1}^{\infty}\int b^{\epsilon}(\frac{u}{|u|}\cdot\sigma) (\tilde{\varphi}_{k}h)(u) [(\varphi_{k}f)(u^{+})- (\varphi_{k}f)(|u|\frac{u^{+}}{|u^{+}|})] du d\sigma
	:= \sum_{k=-1}^{\infty} \mathcal{Y}_{k}.
	\eeno
	where $\tilde{\varphi}_{k} = \sum_{|l-k|\leq 3} \varphi_{l}$.
	We split the proof into two cases:  $2^{k}\geq 1/\epsilon$ and $2^{k}\leq 1/\epsilon$.
	For the case $2^{k}\geq 1/\epsilon$, we have
	\beno
	|\mathcal{Y}_{k}| \leq \bigg(\int b^{\epsilon}(\frac{u}{|u|}\cdot\sigma) |(\tilde{\varphi}_{k}h)(u)|^{2}  du d\sigma\bigg)^{\frac{1}{2}}
	  \bigg(\int b^{\epsilon}(\frac{u}{|u|}\cdot\sigma)  (|(\varphi_{k}f)(u^{+})|^{2} + |(\varphi_{k}f)(|u|\frac{u^{+}}{|u^{+}|})|^{2}) du d\sigma\bigg)^{\frac{1}{2}}.
	\eeno
	By the change of variable $u \rightarrow u^{+}$ and $u \rightarrow  w= |u|\frac{u^{+}}{|u^{+}|}$ respectively, we have
	$|\mathcal{Y}_{k}| \lesssim\epsilon^{-2s}|\tilde{\varphi}_{k}h|_{L^{2}}|\varphi_{k}f|_{L^{2}}$, which implies
	\beno
	|\sum_{2^{k}\geq 1/\epsilon} \mathcal{Y}_{k}| \lesssim \sum_{2^{k}\geq 1/\epsilon} \epsilon^{-2s}|\tilde{\varphi}_{k}h|_{L^{2}}|\varphi_{k}f|_{L^{2}} \lesssim |W^{\epsilon}h|_{L^{2}}|W^{\epsilon}f|_{L^{2}}.
	\eeno
	For the case $2^{k}\leq 1/\epsilon$, by Proposition \ref{fourier-transform-cross-term} and  the dyadic decomposition in the frequency space, we have
	\beno
	\mathcal{Y}_{k} &=& \int b^{\epsilon}(\frac{\xi}{|\xi|}\cdot\sigma)  [\widehat{\tilde{\varphi}_{k}h}(\xi^{+})- \widehat{\tilde{\varphi}_{k}h}(|\xi|\frac{\xi^{+}}{|\xi^{+}|})] \overline{\widehat{\varphi_{k}f}}(\xi)  d\xi d\sigma
	\\&=& \sum_{l=-1}^{\infty} \int b^{\epsilon}(\frac{\xi}{|\xi|}\cdot\sigma)[({\varphi}_{l}\widehat{\tilde{\varphi}_{k}h})(\xi^{+})- ({\varphi}_{l}\widehat{\tilde{\varphi}_{k}h})(|\xi|\frac{\xi^{+}}{|\xi^{+}|})] (\tilde{\varphi}_{l}\overline{\widehat{\varphi_{k}f}})(\xi)  d\xi d\sigma
	:= \sum_{l=-1}^{\infty} \mathcal{Y}_{k,l}.
	\eeno
	For the case $2^{l}\geq 1/\epsilon$, we have
	$
	|\mathcal{Y}_{k,l}| \lesssim \epsilon^{-2s}|{\varphi}_{l}\widehat{\tilde{\varphi}_{k}h}|_{L^{2}}|\tilde{\varphi}_{l}\widehat{\varphi_{k}f}|_{L^{2}}
	$
and thus
	\beno
	\sum_{2^{l}\geq 1/\epsilon}|\mathcal{Y}_{k,l}| \lesssim \sum_{2^{l}\geq 1/\epsilon} \epsilon^{-2s}|{\varphi}_{l}\widehat{\tilde{\varphi}_{k}h}|_{L^{2}}|\tilde{\varphi}_{l}\widehat{\varphi_{k}f}|_{L^{2}}
	\lesssim |W^{\epsilon}(D)\tilde{\varphi}_{k}h|_{L^{2}}|W^{\epsilon}(D)\varphi_{k}f|_{L^{2}}.
	\eeno
	Then by \eqref{decompostionpacth}, we have
	$
	\sum_{2^{k}\leq 1/\epsilon,2^{l}\geq 1/\epsilon}|\mathcal{Y}_{k,l}| \lesssim |W^{\epsilon}(D)h|_{L^{2}}|W^{\epsilon}(D)f|_{L^{2}}.
	$
For the case $2^{l}\leq 1/\epsilon$,  we have
	\beno
	\mathcal{Y}_{k,l} &=& \int b^{\epsilon}(\frac{\xi}{|\xi|}\cdot\sigma)\mathrm{1}_{\theta \geq 2^{-\frac{k+l}{2}}}[({\varphi}_{l}\widehat{\tilde{\varphi}_{k}h})(\xi^{+})- ({\varphi}_{l}\widehat{\tilde{\varphi}_{k}h})(|\xi|\frac{\xi^{+}}{|\xi^{+}|})] (\tilde{\varphi}_{l}\overline{\widehat{\varphi_{k}f}})(\xi) d\xi d\sigma
	\\&& + \int b^{\epsilon}(\frac{\xi}{|\xi|}\cdot\sigma)\mathrm{1}_{\theta \leq 2^{-\frac{k+l}{2}}}[({\varphi}_{l}\widehat{\tilde{\varphi}_{k}h})(\xi^{+})- ({\varphi}_{l}\widehat{\tilde{\varphi}_{k}h})(|\xi|\frac{\xi^{+}}{|\xi^{+}|})] (\tilde{\varphi}_{l}\overline{\widehat{\varphi_{k}f}})(\xi) d\xi d\sigma
	\\&:=& \mathcal{Y}_{k,l,1} + \mathcal{Y}_{k,l,2}.
	\eeno
Since $\int b^{\epsilon}(\cos\theta) \mathrm{1}_{\theta \geq 2^{-k/2-l/2}} d\sigma \lesssim 2^{s(k+l)}$, we have
	$
	|\mathcal{Y}_{k,l,1}| \lesssim 2^{s(k+l)}|{\varphi}_{l}\widehat{\tilde{\varphi}_{k}h}|_{L^{2}}|\tilde{\varphi}_{l}\widehat{\varphi_{k}f}|_{L^{2}}
	$
and thus
	\beno
	\sum_{2^{k}\leq 1/\epsilon,2^{l}\leq 1/\epsilon} |\mathcal{Y}_{k,l,1}| &\leq& \{\sum_{2^{k}\leq 1/\epsilon,2^{l}\leq 1/\epsilon} 2^{s(k+l)}|{\varphi}_{l}\widehat{\tilde{\varphi}_{k}h}|^{2}_{L^{2}}\}^{\frac{1}{2}}\{\sum_{2^{k}\leq 1/\epsilon,2^{l}\leq 1/\epsilon}2^{s(k+l)}|\tilde{\varphi}_{l}\widehat{\varphi_{k}f}|^{2}_{L^{2}}\}^{\frac{1}{2}}
	\\&\lesssim& (|W^{\epsilon}h|_{L^{2}}+|W^{\epsilon}(D)h|_{L^{2}})(|W^{\epsilon}f|_{L^{2}}+|W^{\epsilon}(D)f|_{L^{2}}).
	\eeno
	By Taylor expansion,
	$
	({\varphi}_{l}\widehat{\tilde{\varphi}_{k}h})(\xi^{+})- ({\varphi}_{l}\widehat{\tilde{\varphi}_{k}h})(|\xi|\frac{\xi^{+}}{|\xi^{+}|}) = (1-\frac{1}{\cos\theta})\int_{0}^{1}
	(\nabla \varphi_{l}\widehat{\tilde{\varphi}_{k}h})(\xi^{+}(\kappa))\cdot \xi^{+} d\kappa,
	$
	where $\xi^{+}(\kappa) = (1-\kappa)|\xi|\frac{\xi^{+}}{|\xi^{+}|} + \kappa \xi^{+}$. From which we get
	\beno
	|\mathcal{Y}_{k,l,2}| &=& | \int_{[0,1]\times \R^{3} \times \mathbb{S}^{2}}  b^{\epsilon}(\frac{\xi}{|\xi|}\cdot\sigma)(1-\frac{1}{\cos\theta})\mathrm{1}_{\theta \leq 2^{-\frac{k+l}{2}}}(\tilde{\varphi}_{l}\overline{\widehat{\varphi_{k}f}})(\xi)
	(\nabla \varphi_{l}\widehat{\tilde{\varphi}_{k}h})(\xi^{+}(\kappa))\cdot \xi^{+} d\kappa d\xi d\sigma |
	\\&\lesssim& \{\int_{0}^{2^{-\frac{k+l}{2}}} \int_{\R^{3}}  \theta^{1-2s} |(\tilde{\varphi}_{l}\widehat{\varphi_{k}f})(\xi)|^{2} d\theta d\xi\}^{\frac{1}{2}} \{\int_{0}^{2^{-\frac{k+l}{2}}} \int_{\R^{3}}  \theta^{1-2s} |\eta|^{2}|(\nabla \varphi_{l}\widehat{\tilde{\varphi}_{k}h})(\eta)|^{2} d\theta d\eta\}^{\frac{1}{2}}
	\\&\lesssim& 2^{s(k+l)/2}|\tilde{\varphi}_{l}\widehat{\varphi_{k}f}|_{L^{2}}\{2^{-(2-s)(k+l)}\int |\eta|^{2}|(\nabla \varphi_{l}\widehat{\tilde{\varphi}_{k}h})(\eta)|^{2}d\eta\}^{\frac{1}{2}},
	\eeno
	where we use the change of variable $\xi \rightarrow \eta = \xi^{+}(\kappa)$.
	It is not difficult to compute that
	\beno
	2^{-(2-s)(k+l)}\int |\eta|^{2}|(\nabla \varphi_{l}\widehat{\tilde{\varphi}_{k}h})(\eta)|^{2}d\eta\ \lesssim 2^{-(2-s)(l+k)} |\tilde{\varphi}_{k}h|^{2}_{L^{2}} + 2^{s(l+k)-2k}|\varphi_{l} \widehat{v\tilde{\varphi}_{k}h}|^{2}_{L^{2}}.
	\eeno
	Thus by \eqref{decompostionpacth},
	$
	\sum_{2^{k}\leq 1/\epsilon,2^{l}\leq 1/\epsilon}|\mathcal{Y}_{k,l,2}| \lesssim (|W^{\epsilon}h|_{L^{2}}+|W^{\epsilon}(D)h|_{L^{2}})(|W^{\epsilon}f|_{L^{2}}+|W^{\epsilon}(D)f|_{L^{2}}).
	$
	Patching together all the above results, we conclude
	\ben \label{case-gamma-0}
	|\mathcal{Y}^{\epsilon,0}(h,f)| \lesssim (|W^{\epsilon}h|_{L^{2}}+|W^{\epsilon}(D)h|_{L^{2}})(|W^{\epsilon}f|_{L^{2}}+|W^{\epsilon}(D)f|_{L^{2}}).
	\een

{\it Step 2: $\gamma \neq 0$.}  For simplicity, denote $w = |u|\frac{u^{+}}{|u^{+}|}$, then $W_{\gamma/2}(u) = W_{\gamma/2}(w)$.  Note that the identity
	\beno
	\langle u \rangle^{\gamma} h(u)[f(u^{+}) - f(w)] &=& (W_{\gamma/2}h)(u)[(W_{\gamma/2}f)(u^{+})-(W_{\gamma/2}f)(w)]
	\\&& +(W_{\gamma/2}h)(u)(W_{\gamma/2}f)(u^{+})(W_{\gamma/2}(w)W_{-\gamma/2}(u^{+}) - 1).
	\eeno
and thus
	\beno
	\mathcal{Y}^{\epsilon,\gamma}(h,f) &=& \mathcal{Y}^{\epsilon,0}(W_{\gamma/2}h, W_{\gamma/2}f) + \mathcal{A}
\\  \mathcal{A} &:=& \int b^{\epsilon}(\frac{u}{|u|}\cdot\sigma)(W_{\gamma/2}h)(u)
	(W_{\gamma/2}f)(u^{+})(W_{\gamma/2}(w)W_{-\gamma/2}(u^{+}) - 1) du d\sigma.
	\eeno
	Observing that
	$
	|W_{\gamma/2}(u)W_{-\gamma/2}(u^{+}) - 1| \lesssim \sin^{2}\frac{\theta}{2},
	$
	we have
	\beno
	|\mathcal{A}| &\leq&  \{\int b^{\epsilon}(\frac{u}{|u|}\cdot\sigma)|(W_{\gamma/2}h)(u)|^{2}|W_{\gamma/2}(w)W_{-\gamma/2}(u^{+}) - 1|du d\sigma\}^{\frac{1}{2}}
	\\&&\times\{\int b^{\epsilon}(\frac{u}{|u|}\cdot\sigma)|(W_{\gamma/2}f)(u^{+})|^{2}|W_{\gamma/2}(w)W_{-\gamma/2}(u^{+}) - 1| du d\sigma\}^{\frac{1}{2}}
	\lesssim  |W_{\gamma/2}h|_{L^{2}}|W_{\gamma/2}f|_{L^{2}},
	\eeno
	where the change of variable $u \rightarrow u^{+}$ is used. We use \eqref{case-gamma-0} to handle $\mathcal{Y}^{\epsilon,0}(W_{\gamma/2}h, W_{\gamma/2}f)$ and finish the proof.
\end{proof}

\begin{rmk} \label{exact-cross-term} Denote $
	\mathcal{X}^{\epsilon,\gamma}(h,f) := \int b^{\epsilon}(\frac{u}{|u|}\cdot\sigma)|u|^{\gamma}(1-\phi)(u) h(u)[f(u^{+}) - f(|u|\frac{u^{+}}{|u^{+}|})] du d\sigma,
	$
then
\beno
	|\mathcal{X}^{\epsilon,\gamma}(h,f)| \lesssim (|W^{\epsilon}W_{\gamma/2}h|_{L^{2}}+|W^{\epsilon}(D)W_{\gamma/2}h|_{L^{2}})
	(|W^{\epsilon}W_{\gamma/2}f|_{L^{2}}+|W^{\epsilon}(D)W_{\gamma/2}f|_{L^{2}}).
	\eeno
Indeed, since $|u|^{\gamma}(1-\phi)(u) = \langle u \rangle^{\gamma}(|u|^{\gamma}\langle u \rangle^{-\gamma}-1)(1-\phi)(u)+\langle u \rangle^{\gamma}(1-\phi)(u)$, we have \beno \mathcal{X}^{\epsilon,\gamma}(h,f) = \mathcal{Y}^{\epsilon,\gamma}((|\cdot|^{\gamma}\langle \cdot \rangle^{-\gamma}-1)(1-\phi)h,f)+\mathcal{Y}^{\epsilon,\gamma}((1-\phi)h,f).\eeno
 Then the result follows from Lemma \ref{crosstermsimilar} and \eqref{func5}.
\end{rmk}

\begin{lem}\label{nosigularityeg}
Recall $\mathcal{R}^{\epsilon,\gamma}_{*,g}(h) = \int b^{\epsilon}(\cos\theta)\langle v-v_{*} \rangle^{\gamma}g_{*}(h^{\prime}-h)^{2}d\sigma dv dv_{*}$ defined in \eqref{defintion-R-star-g}.
If $g\ge0$, then
\beno
\mathcal{R}^{\epsilon,\gamma}_{*,g}(h)\lesssim \mathcal{R}^{\epsilon,0}_{g W_{|\gamma|}}(W_{\gamma/2}h) + |g|_{L^1_{|\gamma+2|}}|h|^{2}_{L^{2}_{\gamma/2}}.
\eeno
\end{lem}
\begin{proof}
Let $H = W_{\gamma/2}h$, then we have
\beno
(h^{\prime}-h)^{2} = (H^{\prime}W_{-\gamma/2}^{\prime} - H W_{-\gamma/2})^{2}
\lesssim W^{\prime}_{-\gamma}(H^{\prime}-H)^{2} + (W^{\prime}_{-\gamma/2}-W_{-\gamma/2})^{2}H^{2}.
\eeno
Observing that
$
\langle v^{\prime} \rangle^{-\gamma} \lesssim \langle v^{\prime} - v_{*} \rangle^{-\gamma} \langle v_{*} \rangle^{|\gamma|} \sim
\langle v - v_{*} \rangle^{-\gamma} \langle v_{*} \rangle^{|\gamma|},
$
we have
\beno
\mathcal{R}^{\epsilon,\gamma}_{*,g}(h) \lesssim \mathcal{R}^{\epsilon,0}_{g W_{|\gamma|}}(W_{\gamma/2}h) + \int b^{\epsilon}(\cos\theta)\langle v-v_{*} \rangle^{\gamma}(\langle v^{\prime} \rangle^{-\gamma/2}-\langle v \rangle^{-\gamma/2})^{2}g_{*}H^{2} d\sigma dv dv_{*}
\eeno
By Taylor expansion, one has
$
(W^{\prime}_{-\gamma/2}-W_{-\gamma/2})^{2} \lesssim \int_{0}^{1}\langle v(\kappa) \rangle^{-\gamma-2} \langle v-v_{*} \rangle^{2} \sin^{2}\frac{\theta}{2} d\kappa,
$
where $v(\kappa) = \kappa v + \kappa(v^{\prime}-v)$.
Note that $\langle v-v_{*} \rangle^{\gamma+2}\sim \langle v(\kappa)-v_{*} \rangle^{\gamma+2}\lesssim  \langle v(\kappa) \rangle^{\gamma+2}\langle  v_{*} \rangle^{|\gamma+2|}$. Then we have
\beno \mathcal{R}^{\epsilon,\gamma}_{*,g}(h) \lesssim \mathcal{R}^{\epsilon,0}_{g W_{|\gamma|}}(W_{\gamma/2}h) + \int b^{\epsilon}(\cos\theta)\theta^2 \langle  v_{*} \rangle^{|\gamma+2|} g_{*}H^{2} d\sigma dv dv_{*},\eeno
which yields the desired result.
\end{proof}

Now we are in a position to prove the following upper bound  of $Q^\epsilon_{\geq 0}$.
\begin{prop}\label{ubqepsilonnonsingular} There holds
\beno
|\langle Q^{\epsilon}_{\geq 0}(g,h), f\rangle_v| \lesssim | g|_{L^{1}_{|\gamma|+2}}|h|_{\epsilon,\gamma/2}|f|_{\epsilon,\gamma/2}.
\eeno
\end{prop}
\begin{proof} Define the translation operator $T_{v_{*}}$ by $(T_{v_{*}}f)(u) = f(v_{*}+u)$ for $v_{*}, u \in \mathbb{R}^{3}$.
By geometric decomposition, we have
$\langle Q^{\epsilon}_{\geq 0}(g,h), f\rangle_v = \mathcal{D}_{1} + \mathcal{D}_{2}$,
where
\beno \mathcal{D}_{1} := \int b^{\epsilon}(\frac{u}{|u|}\cdot\sigma)|u|^{\gamma}(1-\phi)(u)g_{*} (T_{v_{*}}h)(u) ((T_{v_{*}}f)(u^{+})-(T_{v_{*}}f)(|u|\frac{u^{+}}{|u^{+}|})) d\sigma dv_{*} du,
\\
\mathcal{D}_{2} :=\int b^{\epsilon}(\frac{u}{|u|}\cdot\sigma)|u|^{\gamma}(1-\phi)(u)g_{*} (T_{v_{*}}h)(u)((T_{v_{*}}f)(|u|\frac{u^{+}}{|u^{+}|})- (T_{v_{*}}f)(u)) d\sigma dv_{*} du.\eeno

{\it Step 1: Estimate of $\mathcal{D}_{1}$.}
By Lemma \ref{crosstermsimilar} and Remark \ref{exact-cross-term}, we have
\beno
|\mathcal{D}_{1}| &\lesssim& \int |g_{*}| (|W^{\epsilon}W_{\gamma/2}T_{v_{*}}h|_{L^{2}}+|W^{\epsilon}(D)W_{\gamma/2}T_{v_{*}}h|_{L^{2}})
\\&&\times(|W^{\epsilon}W_{\gamma/2}T_{v_{*}}f|_{L^{2}}+|W^{\epsilon}(D)W_{\gamma/2}T_{v_{*}}f|_{L^{2}}) dv_{*}.
\eeno
It is easy to check that
\ben\label{tvstartonovstar1}
|W^{\epsilon}W_{\gamma/2}T_{v_{*}}h|_{L^{2}} \lesssim W^{\epsilon}(v_{*})W_{|\gamma|/2}(v_{*})|W^{\epsilon}W_{\gamma/2}h|_{L^{2}}.
\een
By Lemma \ref{operatorcommutator1}, we have
\ben\label{tvstartonovstar2}
|W^{\epsilon}(D)W_{\gamma/2}T_{v_{*}}h|_{L^{2}} &\lesssim& |W_{\gamma/2}W^{\epsilon}(D)T_{v_{*}}h|_{L^{2}} + |T_{v_{*}}h|_{H^{s-1}_{\gamma/2-1}}
\\&\lesssim& W_{|\gamma|/2}(v_{*})(|W_{\gamma/2}W^{\epsilon}(D)h|_{L^{2}} + |h|_{L^{2}_{\gamma/2-1}}) \nonumber
\\&\lesssim& W_{|\gamma|/2}(v_{*})|W^{\epsilon}(D)W_{\gamma/2}h|_{L^{2}}. \nonumber
\een
Thus we get the estimate of $\mathcal{D}_{1}$  as follows
\beno
|\mathcal{D}_{1}| \lesssim | g|_{L^{1}_{|\gamma|+2}}( |W^{\epsilon}(D)W_{\gamma/2}h|_{L^{2}} + |W^{\epsilon}W_{\gamma/2}h|_{L^{2}})
( |W^{\epsilon}(D)W_{\gamma/2}f|_{L^{2}} + |W^{\epsilon}W_{\gamma/2}f|_{L^{2}}).
\eeno

{\it Step 2: Estimate of $\mathcal{D}_{2}$.}
Let $u = r \tau$ and $\varsigma = \frac{\tau+\sigma}{|\tau+\sigma|}$, then $\frac{u}{|u|} \cdot \sigma = 2(\tau\cdot\varsigma)^{2} - 1$ and $|u|\frac{u^{+}}{|u^{+}|} = r \varsigma$. By the change of variable $(u, \sigma) \rightarrow (r, \tau, \varsigma)$, one has
$
du d\sigma = 4  (\tau\cdot\varsigma) r^{2} dr d \tau d \varsigma.
$
Then
\beno
\mathcal{D}_{2} &=& 4 \int r^\gamma(1-\phi)(r)b^{\epsilon}(2(\tau\cdot\varsigma)^{2} - 1)(T_{v_*}h)(r\tau)\big((T_{v_*}f)(r\varsigma) - (T_{v_*}f) (r\tau)\big) (\tau\cdot\varsigma) r^{2} dr d \tau d \varsigma dv_{*}\\
&=& 2 \int r^\gamma(1-\phi)(r)b^{\epsilon}(2(\tau\cdot\varsigma)^{2} - 1)\big((T_{v_*}h)(r\tau) - (T_{v_*}h) (r\varsigma)\big)\\ &&\times \big((T_{v_*}f)(r\varsigma) - (T_{v_*}f) (r\tau)\big) (\tau\cdot\varsigma) r^{2} dr d \tau d \varsigma dv_{*}\\&=&
-\frac{1}{2}\int b^{\epsilon}(\frac{u}{|u|}\cdot\sigma)|u|^{\gamma}(1-\phi)(u)g_{*} ((T_{v_{*}}h)(|u|\frac{u^{+}}{|u^{+}|})-(T_{v_{*}}h)(u))\\ &&\times
 ((T_{v_{*}}f)(|u|\frac{u^{+}}{|u^{+}|})-(T_{v_{*}}f)(u)) d\sigma dv_{*} du.
\eeno
Then by Cauchy-Schwartz inequality and the fact $|u|^{\gamma}(1-\phi)(u) \lesssim \langle u \rangle^{\gamma}$, we have
\beno
|\mathcal{D}_{2}| &\lesssim& \{\int b^{\epsilon}(\frac{u}{|u|}\cdot\sigma)\langle u \rangle^{\gamma}|g_{*}| ((T_{v_{*}}h)( |u|\frac{u^{+}}{|u^{+}|})-(T_{v_{*}}h)( u))^{2} d\sigma dv_{*} du\}^{\frac{1}{2}}
\\&& \times \{\int b^{\epsilon}(\frac{u}{|u|}\cdot\sigma)\langle u \rangle^{\gamma}|g_{*}|
((T_{v_{*}}f)(|u|\frac{u^{+}}{|u^{+}|})-(T_{v_{*}})f(u))^{2} d\sigma dv_{*} du\}^{\frac{1}{2}}
:= (\mathcal{D}_{2,h})^{\frac{1}{2}}(\mathcal{D}_{2,f})^{\frac{1}{2}}.
\eeno
Note that $\mathcal{D}_{2,h}$ and $\mathcal{D}_{2,f}$ have exactly the same structure. It suffices to focus on  $\mathcal{D}_{2,f}$.
Since
\beno
((T_{v_{*}}f)(|u|\frac{u^{+}}{|u^{+}|})-(T_{v_{*}}f)(u))^{2} \leq 2 ((T_{v_{*}}f)(|u|\frac{u^{+}}{|u^{+}|})-(T_{v_{*}}f)(u^{+}))^{2} + 2 ((T_{v_{*}}f)(u^{+})-(T_{v_{*}}f)(u))^{2},
\eeno
we have
\beno
\mathcal{D}_{2,f} &\lesssim& \int b^{\epsilon}(\frac{u}{|u|}\cdot\sigma)\langle u \rangle^{\gamma}|g_{*}| ((T_{v_{*}}f)(|u|\frac{u^{+}}{|u^{+}|})-(T_{v_{*}}f)(u^{+}))^{2} d\sigma dv_{*} du
\\&&+\int b^{\epsilon}(\frac{u}{|u|}\cdot\sigma)\langle u \rangle^{\gamma}|g_{*}| ((T_{v_{*}}f)(u^{+})-(T_{v_{*}}f)(u))^{2} d\sigma dv_{*} du
:= \mathcal{D}_{2,f,1}+ \mathcal{D}_{2,f,2}.
\eeno
By Lemma \ref{gammanonzerotozero}, and the facts \eqref{tvstartonovstar1} and \eqref{tvstartonovstar2}, we have
\beno
\mathcal{D}_{2,f,1} \lesssim \int |g_{*}| \mathcal{Z}^{\epsilon,\gamma}(T_{v_{*}}f) dv_{*} \lesssim| g|_{L^{1}_{|\gamma|+2}}(|W^{\epsilon}(D)W_{\gamma/2}f|^{2}_{L^{2}}+|W^{\epsilon}W_{\gamma/2}f|^{2}_{L^{2}}).
\eeno
Thanks to Lemma \ref{nosigularityeg}, we have
\beno
\mathcal{D}_{2,f,2} \le  \mathcal{R}^{\epsilon,\gamma}_{*,|g|}(f)
\lesssim \mathcal{R}^{\epsilon,0}_{|g| W_{|\gamma|}}(W_{\gamma/2}f) +  | g|_{L^{1}_{|\gamma|+2}}|f|^{2}_{L^{2}_{\gamma/2}}.
\eeno
Thanks to the estimate (see Lemma 3.3 in \cite{he2})
$
\mathcal{R}^{\epsilon,0}_{g}(f) \lesssim |g|_{L^{1}} \mathcal{R}^{\epsilon,0}_{\mu}(f) + | g|_{L^{1}_2}|W^{\epsilon}(D)f|^{2}_{L^{2}},
$
using \eqref{lowerupper1} in Proposition \ref{lowerboundpart2} to get
$
\mathcal{R}^{\epsilon,0}_{\mu}(f) \lesssim |W^{\epsilon}((-\Delta_{\mathbb{S}^{2}})^{\frac{1}{2}})f|^{2}_{L^{2}} + |W^{\epsilon}(D)f|^{2}_{L^{2}} + |W^{\epsilon}f|^{2}_{L^{2}},
$
we have
\beno
\mathcal{D}_{2,f,2} \lesssim  | g|_{L^{1}_{|\gamma|+2}}(|W^{\epsilon}((-\Delta_{\mathbb{S}^{2}})^{\frac{1}{2}})W_{\gamma/2}f|^{2}_{L^{2}} + |W^{\epsilon}(D)W_{\gamma/2}f|^{2}_{L^{2}} + |W^{\epsilon}W_{\gamma/2}f|^{2}_{L^{2}}).
\eeno
Thus we have
$
\mathcal{D}_{2,f} \lesssim  |g|_{L^{1}_{|\gamma|+2}}|f|^{2}_{\epsilon,\gamma/2}.
$
Similarly $\mathcal{D}_{2,h}$ has the same upper bound, so we have
\beno
|\mathcal{D}_{2}| \lesssim | g|_{L^{1}_{|\gamma|+2}}|h|_{\epsilon,\gamma/2}|f|_{\epsilon,\gamma/2}.
\eeno

We finish the proof of
 by patching together the estimates of $\mathcal{D}_{1}$ and $\mathcal{D}_{2}$.
\end{proof}

Recalling \eqref{Q-into-two-parts},
by Proposition \ref{ubqepsilonsingular} and Proposition \ref{ubqepsilonnonsingular}, we are led to
\begin{thm} \label{upQepsilon}
	For any $\eta>0$, the following estimates are valid.
\begin{itemize}
\item If $\gamma>-\frac{3}{2}$,
 $|\langle Q^{\epsilon}(g,h), f\rangle_v| \lesssim (|g|_{L^{2}_{|\gamma|}}+| g|_{L^{1}_{|\gamma|+2}})|h|_{\epsilon,\gamma/2}|f|_{\epsilon,\gamma/2}.$

\item If $\gamma=-\f32$, $|\langle Q^{\epsilon}(g,h), f\rangle_v| \lesssim
(|g|_{L^1_{|\gamma|}}+|g|_{H^{s_{1}}_{|\gamma|}})|W^{\epsilon}(D)h|_{H^{s_{2}}_{\gamma/2}}
   |f|_{\epsilon,\gamma/2} + |g|_{L^1_{|\gamma|+2}}|h|_{\epsilon,\gamma/2}|f|_{\epsilon,\gamma/2}.$ Here $(s_1, s_2) = (0, \eta)$ or $(\eta, 0)$.
\item If $-3< \gamma < -\frac{3}{2}$,
$
|\langle Q^{\epsilon} (g,h), f\rangle_v| \lesssim |g|_{H^{s_{1}}_{|\gamma|}}|W^{\epsilon}(D)h|_{H^{s_{2}}_{\gamma/2}}|W^{\epsilon}(D)f|_{H^{s_{3}}_{\gamma/2}}+| g|_{L^{1}_{|\gamma|+2}}|h|_{\epsilon,\gamma/2}|f|_{\epsilon,\gamma/2}.
$ Here the constants $s_1, s_2, s_3 \geq 0$ verify either $s_{1}+s_{2}+s_{3}=-\gamma-\frac{3}{2}, s_{2}+s_{3}>0$ or  $s_1=-\gamma-\frac{3}{2}+\eta, s_2=s_3=0$.
\end{itemize}
\end{thm}

\subsubsection{Upper bounds of $\mathcal{I}(g,h,f)$} \label{estimate-of-I}
For ease of notation, we simply write $\mathcal{I}(g,h,f)$ as $\mathcal{I}$. We first make some rearrangement.
Noting that
\beno
(\mu^{\frac{1}{2}})_{*}^{\prime} - \mu_{*}^{\frac{1}{2}} = ((\mu^{\frac{1}{4}})_{*}^{\prime} + \mu_{*}^{\frac{1}{4}})((\mu^{\frac{1}{4}})_{*}^{\prime} - \mu_{*}^{\frac{1}{4}})
= ((\mu^{\frac{1}{8}})_{*}^{\prime} + \mu_{*}^{\frac{1}{8}})^{2}((\mu^{\frac{1}{8}})_{*}^{\prime} - \mu_{*}^{\frac{1}{8}})^{2} + 2\mu_{*}^{\frac{1}{4}}((\mu^{\frac{1}{4}})_{*}^{\prime} - \mu_{*}^{\frac{1}{4}}),
\eeno
and $h=(h-h^{\prime}) + h^{\prime}$, recalling \eqref{definition-of-Q},
we have
\ben \label{I-into-I1-I2-I3}
\mathcal{I} &=& \mathcal{I}_{1} + \mathcal{I}_{2} + \mathcal{I}_{3}.
\\ \label{defintion-of-I1}
\mathcal{I}_{1}  &:=&
\int b^{\epsilon}(\cos\theta)|v-v_{*}|^{\gamma} ((\mu^{\frac{1}{8}})_{*}^{\prime} + \mu_{*}^{\frac{1}{8}})^{2}((\mu^{\frac{1}{8}})_{*}^{\prime} - \mu_{*}^{\frac{1}{8}})^{2}g_{*} h f^{\prime} d\sigma dv_{*} dv,
\\ \label{defintion-of-I2}
\mathcal{I}_{2}  &:=&  2 \int b^{\epsilon}(\cos\theta)|v-v_{*}|^{\gamma} ((\mu^{\frac{1}{4}})_{*}^{\prime} - \mu_{*}^{\frac{1}{4}})(\mu^{\frac{1}{4}}g)_{*}(h-h^{\prime})f^{\prime} d\sigma dv_{*}dv,
\\ \label{defintion-of-I3}
\mathcal{I}_{3}  &:=&  2 \int b^{\epsilon}(\cos\theta)|v-v_{*}|^{\gamma}((\mu^{\frac{1}{4}})_{*}^{\prime} - \mu_{*}^{\frac{1}{4}})(\mu^{\frac{1}{4}}g)_{*} h^{\prime} f^{\prime} d\sigma dv_{*}dv.
\een

\begin{lem}\label{upforI}
For any $\eta>0$, the following estimates are valid.
\begin{itemize}
	\item If $\gamma>-\frac{3}{2}$, then
 $|\mathcal{I}(g,h,f)| \lesssim  |g|_{L^{2}}|h|_{\epsilon,\gamma/2}|W^{\epsilon}f|_{L^{2}_{\gamma/2}}.$
 \item If $\gamma=-\f32$, then
 \beno |\mathcal{I}(g,h,f)|\lesssim |\mu^{\frac{1}{8}}g|_{H^{s_{1}}}(|W^{\epsilon}(D)h|_{H^{s_{2}}_{\gamma/2}}+|h|_{\epsilon,\gamma/2})
  |W^{\epsilon}f|_{L^{2}_{\gamma/2}}+
|g|_{L^{2}}|h|_{\epsilon,\gamma/2}|W^{\epsilon}f|_{L^{2}_{\gamma/2}},\eeno
where $(s_1, s_2) = (0, \eta)$ or $(\eta, 0)$.
 \item If $-3< \gamma < -\frac{3}{2}$, then
 \beno |\mathcal{I}(g,h,f)|\lesssim |\mu^{\frac{1}{8}}g|_{H^{s_{1}}}(|W^{\epsilon}(D)h|_{H^{s_{2}}_{\gamma/2}}+|h|_{\epsilon,\gamma/2})
  |W^{\epsilon}f|_{L^{2}_{\gamma/2}}+
|g|_{L^{2}}|h|_{\epsilon,\gamma/2}|W^{\epsilon}f|_{L^{2}_{\gamma/2}},\eeno
where $s_1, s_2 \geq 0$ are constants verifying either $s_{1}+2s_{2}=-\gamma-\frac{3}{2},  s_{2}>0$ or $s_1=-\gamma-\frac{3}{2}+\eta, s_2=0$.
 \end{itemize}
\end{lem}
\begin{proof}  In what follows, we will
constantly use the fact:
\ben \label{freq-used-inequality}
((\mu^{\frac{1}{8}})_{*}^{\prime} - \mu_{*}^{\frac{1}{8}})^{2} \lesssim \min\{1,|v-v_{*}|^{2}\sin^{2}\frac{\theta}{2}\} \sim \min\{1,|v^{\prime}-v_{*}|^{2}\sin^{2}\frac{\theta}{2}\}\sim \min\{1,|v-v^{\prime}_{*}|^{2}\sin^{2}\frac{\theta}{2}\}.
\een

We estimate $\mathcal{I}_{1}, \mathcal{I}_{2}$ and $\mathcal{I}_{3}$ one by one.

{\it Step 1: Estimate of $\mathcal{I}_{1}$.} Recalling \eqref{defintion-of-I1},
we use $\phi$ defined in \eqref{function-phi-psi} to separate the relative velocity into two parts:
\beno
\mathcal{I}_{1} &=& \int b^{\epsilon}(\cos\theta)|v-v_{*}|^{\gamma}(1-\phi)(v-v_{*})((\mu^{\frac{1}{8}})_{*}^{\prime} + \mu_{*}^{\frac{1}{8}})^{2}((\mu^{\frac{1}{8}})_{*}^{\prime} - \mu_{*}^{\frac{1}{8}})^{2}g_{*} h f^{\prime} d\sigma dv_{*} dv
\\&&+\int b^{\epsilon}(\cos\theta)|v-v_{*}|^{\gamma}\phi(v-v_{*})((\mu^{\frac{1}{8}})_{*}^{\prime} + \mu_{*}^{\frac{1}{8}})^{2}((\mu^{\frac{1}{8}})_{*}^{\prime} - \mu_{*}^{\frac{1}{8}})^{2}g_{*} h f^{\prime} d\sigma dv_{*} dv
\\&:=&  \mathcal{I}_{1,1}+\mathcal{I}_{1,2}.
\eeno

\underline{Estimate of $\mathcal{I}_{1,1}$.} Note that $|v-v_{*}| \sim \langle v-v_{*}\rangle$
in $\mathcal{I}_{1,1}$.
By Cauchy-Schwartz inequality, we have
\beno
|\mathcal{I}_{1,1}| &\lesssim& \{\int b^{\epsilon}(\cos\theta)\langle v-v_{*}\rangle^{\gamma}((\mu^{\frac{1}{8}})_{*}^{\prime} + \mu_{*}^{\frac{1}{8}})^{2}((\mu^{\frac{1}{8}})_{*}^{\prime} - \mu_{*}^{\frac{1}{8}})^{2}g^{2}_{*} h^{2}  d\sigma dv_{*} dv\}^{\frac{1}{2}}
\\&&\times \{\int b^{\epsilon}(\cos\theta)|v-v_{*}|^{\gamma}((\mu^{\frac{1}{8}})_{*}^{\prime} + \mu_{*}^{\frac{1}{8}})^{2}((\mu^{\frac{1}{8}})_{*}^{\prime} - \mu_{*}^{\frac{1}{8}})^{2} (f^{2})^{\prime} d\sigma dv_{*} dv\}^{\frac{1}{2}}
:= (\mathcal{I}_{1,1,1})^{\frac{1}{2}} (\mathcal{I}_{1,1,2})^{\frac{1}{2}}.
\eeno
We claim that
\ben \label{claim-1-L2} \mathcal{A} := \int b^{\epsilon}(\cos\theta)\langle v-v_{*}\rangle^{\gamma}((\mu^{\frac{1}{8}})_{*}^{\prime} + \mu_{*}^{\frac{1}{8}})^{2}((\mu^{\frac{1}{8}})_{*}^{\prime} - \mu_{*}^{\frac{1}{8}})^{2}  d\sigma \lesssim (W^{\epsilon})^{2}(v)\langle v\rangle^{\gamma},\een
which yields
$\mathcal{I}_{1,1,1} \lesssim  |g|^{2}_{L^{2}}|W^{\epsilon}h|^{2}_{L^{2}_{\gamma/2}}$.

To prove \eqref{claim-1-L2}, we notice that
\beno \mathcal{A} &\lesssim& \int b^{\epsilon}(\cos\theta)\langle v-v_{*}\rangle^{\gamma}\mu_{*}^{\frac{1}{4}}((\mu^{\frac{1}{8}})_{*}^{\prime} - \mu_{*}^{\frac{1}{8}})^{2}  d\sigma + \int b^{\epsilon}(\cos\theta)\langle v-v_{*}\rangle^{\gamma}(\mu^{\frac{1}{4}})_{*}^{\prime}((\mu^{\frac{1}{8}})_{*}^{\prime} - \mu_{*}^{\frac{1}{8}})^{2}  d\sigma
\\&:=& \mathcal{A}_{1}+\mathcal{A}_{2}.
\eeno
Due to Proposition \ref{symbol} and the fact \eqref{combinewithgamma},
we get $\mathcal{A}_{1} \lesssim \langle v-v_{*}\rangle^{\gamma}\mu_{*}^{\frac{1}{4}}(W^{\epsilon})^{2}(v-v_{*}) \lesssim (W^{\epsilon})^{2}(v)\langle v\rangle^{\gamma}$.
As for $\mathcal{A}_{2}$, thanks to $|v-v_{*}| \sim |v-v^{\prime}_{*}|$ and thus $\langle v-v_{*}\rangle^{\gamma} \lesssim \langle v-v^{\prime}_{*}\rangle^{\gamma} \lesssim \langle v\rangle^{\gamma}\langle v^{\prime}_{*}\rangle^{|\gamma|}$, we have
 $\mathcal{A}_{2} \lesssim \langle v\rangle^{\gamma} \int b^{\epsilon}(\cos\theta) (\mu^{\frac{1}{8}})_{*}^{\prime} \min\{1,|v-v_{*}|^{2}\sin^{2}\frac{\theta}{2}\} d\sigma$.
If $|v-v_{*}|\geq 10|v|$, then there holds $|v^{\prime}_{*}| = |v^{\prime}_{*}-v+v|\geq |v^{\prime}_{*}-v| -|v| \geq (1/\sqrt{2} - 1/10)|v-v_{*}| \geq \frac{1}{5}|v-v_{*}|$, and thus $(\mu^{\frac{1}{8}})_{*}^{\prime} \lesssim \mu^{\frac{1}{2}00}(v-v_{*})$,
which indicates
\beno \mathcal{A}_{2} \lesssim \langle v\rangle^{\gamma}\mu^{\frac{1}{2}00}(v-v_{*}) (W^{\epsilon})^{2}(v-v_{*}) \lesssim \langle v\rangle^{\gamma}.\eeno
If $|v-v_{*}|\leq 10|v|$, by Proposition \ref{symbol}, we have
\beno \mathcal{A}_{2} \lesssim \langle v\rangle^{\gamma} \int b^{\epsilon}(\cos\theta)  \min\{1,|v|^{2}\sin^{2}\frac{\theta}{2}\} d\sigma  \lesssim (W^{\epsilon})^{2}(v)\langle v\rangle^{\gamma}.\eeno
We finished the proof of \eqref{claim-1-L2}.
By the change of variable $(v, v_{*})\rightarrow (v^{\prime}, v_{*}^{\prime})$, we have
\beno
\mathcal{I}_{1,1,2} &=& \int b^{\epsilon}(\cos\theta)|v-v_{*}|^{\gamma}((\mu^{\frac{1}{8}})_{*}^{\prime} + \mu_{*}^{\frac{1}{8}})^{2}((\mu^{\frac{1}{8}})_{*}^{\prime} - \mu_{*}^{\frac{1}{8}})^{2} f^{2} d\sigma dv_{*} dv
\\&\leq& 2\int b^{\epsilon}(\cos\theta)|v-v_{*}|^{\gamma}\mu_{*}^{\frac{1}{4}}((\mu^{\frac{1}{8}})_{*}^{\prime} - \mu_{*}^{\frac{1}{8}})^{2} f^{2} d\sigma dv_{*} dv
\\&& + 2\int b^{\epsilon}(\cos\theta)|v-v_{*}|^{\gamma}(\mu^{\frac{1}{4}})_{*}^{\prime}((\mu^{\frac{1}{8}})_{*}^{\prime} - \mu_{*}^{\frac{1}{8}})^{2} f^{2} d\sigma dv_{*} dv
:= \mathcal{I}_{1,1,2,1} + \mathcal{I}_{1,1,2,2}.
\eeno
With the help of \eqref{freq-used-inequality}, Proposition \ref{symbol},
\eqref{combinewithgamma} and \eqref{mu-cancel-sigularity}, we have
\beno
\mathcal{I}_{1,1,2,1} &\lesssim& \int b^{\epsilon}(\cos\theta)|v-v_{*}|^{\gamma}\mu_{*}^{\frac{1}{4}}\min\{1,|v-v_{*}|^{2}\sin^{2}\frac{\theta}{2}\} f^{2} d\sigma dv_{*} dv
\\&\lesssim&  \int |v-v_{*}|^{\gamma}\mu_{*}^{\frac{1}{8}} (W^{\epsilon})^{2}(v)f^{2}  dv_{*} dv
\lesssim |W^{\epsilon}f|^{2}_{L^{2}_{\gamma/2}}.
\eeno
By the fact $|v-v_{*}|\sim|v-v_{*}^{\prime}|$ and the   change of variable $v_{*}\rightarrow v_{*}^{\prime}$, we have
\beno
\mathcal{I}_{1,1,2,2} \lesssim \int b^{\epsilon}(\cos\theta)|v-v_{*}^{\prime}|^{\gamma}(\mu^{\frac{1}{4}})_{*}^{\prime}
\min\{1,|v-v_{*}^{\prime}|^{2}\sin^{2}\frac{\theta}{2}\} f^{2} d\sigma dv_{*}^{\prime} dv
\lesssim |W^{\epsilon}f|^{2}_{L^{2}_{\gamma/2}}.
\eeno
Therefore we have
$
\mathcal{I}_{1,1,2} \lesssim |W^{\epsilon}f|^{2}_{L^{2}_{\gamma/2}}.
$
Together with the estimate  for $\mathcal{I}_{1,1,1}$, we have
\ben \label{I-1-1}
\mathcal{I}_{1,1} \lesssim |g|_{L^{2}}|W^{\epsilon}h|_{L^{2}_{\gamma/2}}|W^{\epsilon}f|_{L^{2}_{\gamma/2}}.
\een

\underline{Estimate of $\mathcal{I}_{1,2}$.} By Cauchy-Schwartz inequality, we have
\beno
|\mathcal{I}_{1,2}| &\lesssim& \{\int b^{\epsilon}(\cos\theta)|v-v_{*}|^{\gamma}\phi(v-v_{*})((\mu^{\frac{1}{8}})_{*}^{\prime} + \mu_{*}^{\frac{1}{8}})^{2}((\mu^{\frac{1}{8}})_{*}^{\prime} - \mu_{*}^{\frac{1}{8}})^{2}|g_{*}| h^{2}  d\sigma dv_{*} dv\}^{\frac{1}{2}}
\\&&\times \{\int b^{\epsilon}(\cos\theta)|v-v_{*}|^{\gamma}\phi(v-v_{*})((\mu^{\frac{1}{8}})_{*}^{\prime} + \mu_{*}^{\frac{1}{8}})^{2}((\mu^{\frac{1}{8}})_{*}^{\prime} - \mu_{*}^{\frac{1}{8}})^{2} |g_{*}|(f^{2})^{\prime} d\sigma dv_{*} dv\}^{\frac{1}{2}}
\\&:=& (\mathcal{I}_{1,2,1})^{\frac{1}{2}} (\mathcal{I}_{1,2,2})^{\frac{1}{2}}.
\eeno
Note that the support of function $\phi$ is $B_{\frac{4}{3}}$. When $|v-v_{*}| \leq \frac{4}{3}$, there hold
$
|v_{*}|\geq |v|-\frac{4}{3}$ and $ |v_{*}^{\prime}|\geq |v|-|v-v_{*}^{\prime}| \geq |v|-|v-v_{*}| \geq |v|-\frac{4}{3},
$ which imply that $((\mu^{\frac{1}{8}})_{*}^{\prime} + \mu_{*}^{\frac{1}{8}})^{2}\lesssim \mu^{\f14}$. Recalling  Proposition \ref{symbol}, Cauchy-Schwartz inequality and the assumption
 $\gamma+2 > -1$, one has
\beno
\mathcal{I}_{1,2,1} \lesssim \int |v-v_{*}|^{\gamma+2}\phi(v-v_{*})\mu^{\frac{1}{8}}|g_{*}| h^{2}  d\sigma dv_{*} dv
\lesssim |g|_{L^{2}}|\mu^{\frac{1}{16}}h|^{2}_{L^{2}}.
\eeno
By the change of variable $v \rightarrow v^{\prime}$,
we can similarly derive that
$
\mathcal{I}_{1,2,2}
\lesssim |g|_{L^{2}}|\mu^{\frac{1}{16}}f|^{2}_{L^{2}}.
$
Patching together the estimates for $\mathcal{I}_{1,2,1}$ and $\mathcal{I}_{1,2,2}$, we arrive at
$
|\mathcal{I}_{1,2}|  \lesssim |g|_{L^{2}}|\mu^{\frac{1}{16}}h|_{L^{2}}|\mu^{\frac{1}{16}}f|_{L^{2}}.
$
Together with the estimate \eqref{I-1-1} for $\mathcal{I}_{1,1}$, we obtain
$
|\mathcal{I}_{1}| \lesssim |g|_{L^{2}}|W^{\epsilon}h|_{L^{2}_{\gamma/2}}|W^{\epsilon}f|_{L^{2}_{\gamma/2}}.
$

{\it Step 2: Estimate of $\mathcal{I}_{2}$.} Recalling \eqref{defintion-of-I2}, by Cauchy-Schwartz inequality, we have
\beno
\mathcal{I}_{2}
&\lesssim&  \{\int b^{\epsilon}(\cos\theta)|v-v_{*}|^{\gamma}|(\mu^{\frac{1}{4}}g)_{*}|(h-h^{\prime})^{2}d\sigma dv_{*} dv\}^{\frac{1}{2}}
\\&&\times \{\int b^{\epsilon}(\cos\theta)|v-v_{*}|^{\gamma}((\mu^{\frac{1}{4}})_{*}^{\prime} - \mu_{*}^{\frac{1}{4}})^{2}|(\mu^{\frac{1}{4}}g)_{*}|(f^{2})^{\prime}d\sigma dv_{*} dv\}^{\frac{1}{2}}
:= (\mathcal{I}_{2,1})^{\frac{1}{2}}(\mathcal{I}_{2,2})^{\frac{1}{2}}.
\eeno

\underline{Estimate of $\mathcal{I}_{2,1}$.}
Notice that
$
(h-h^{\prime})^{2} = (h^{2})^{\prime} - h^{2} - 2h(h^{\prime}-h),
$
we have
\beno
\mathcal{I}_{2,1} &=& \mathcal{I}_{2,1,1} - 2 \langle Q^{\epsilon}(|\mu^{\frac{1}{4}}g|, h), h\rangle
\\
\mathcal{I}_{2,1,1} &:=& \int b^{\epsilon}(\cos\theta)|v-v_{*}|^{\gamma}|(\mu^{\frac{1}{4}}g)_{*}|((h^{2})^{\prime} - h^{2})d\sigma dv_{*} dv.
\eeno
By cancellation lemma in \cite{advw}, one has
$ \mathcal{I}_{2,1,1} =C(\epsilon)\int |v-v_{*}|^{\gamma}|(\mu^{\frac{1}{4}}g)_{*}|h^{2} dv_{*} dv$ with $|C(\epsilon)|\lesssim 1$. Thus
 by Lemma \ref{aftercancellation} and Theorem \ref{upQepsilon}, we have,
  if $\gamma>-\frac{3}{2}$,
 $|\mathcal{I}_{2,1}| \lesssim  |\mu^{\frac{1}{8}}g|_{L^{2}}|h|_{\epsilon,\gamma/2}^2$;
 if $\gamma=-\f32$, $|\mathcal{I}_{2,1}| \lesssim |\mu^{\frac{1}{8}}g|_{H^{s_{1}}}|W^{\epsilon}(D)h|_{H^{s_{2}}_{\gamma/2}}
   |h|_{\epsilon,\gamma/2} + |\mu^{\frac{1}{8}}g|_{L^2}|h|_{\epsilon,\gamma/2}^{2}$ where $(s_1, s_2) = (0, \eta)$ or $(\eta, 0)$;
  if $-3<\gamma< -\frac{3}{2}$,
\beno
|\mathcal{I}_{2,1}| \lesssim |\mu^{\frac{1}{8}}g|_{H^{s_{1}}}|W^{\epsilon}(D)h|_{H^{s_{2}}_{\gamma/2}}^2+| \mu^{\frac{1}{8}}g|_{L^{2}}|h|_{\epsilon,\gamma/2}^2,
\eeno
where $s_1, s_2 \geq 0$ verify either $s_{1}+2s_{2}=-\gamma-\frac{3}{2},  s_{2}>0$ or $s_1=-\gamma-\frac{3}{2}+\eta, s_2=0$.

\underline{Estimate of $\mathcal{I}_{2,2}$.}
 We separate the relative velocity $|v-v_*|$ into two regions  by introducing the cutoff function $\phi$. If $|v-v_*|\lesssim 1$, the estimate is the same as  that for $\mathcal{I}_{1,1,1}$. If $|v-v_*|\ge 1$ , the estimate is exactly the same as that for $\mathcal{I}_{1,2,2}$. We conclude that
$
\mathcal{I}_{2,2} \lesssim |\mu^{\frac{1}{8}}g|_{L^{2}}|W^{\epsilon}f|_{L^{2}_{\gamma/2}}^{2}.
$ Patching together the estimates of $\mathcal{I}_{2,1}$ and $\mathcal{I}_{2,2}$, we get
\beno
\gamma>-\frac{3}{2}, & \quad
|\mathcal{I}_{2}|
\lesssim |\mu^{\frac{1}{8}}g|_{L^{2}}|h|_{\epsilon,\gamma/2}|W^{\epsilon}f|_{L^{2}_{\gamma/2}};
\\
-3<\gamma \leq -\frac{3}{2}, & \quad |\mathcal{I}_{2}|
\lesssim |\mu^{\frac{1}{8}}g|_{H^{s_{1}}}(|W^{\epsilon}(D)h|_{H^{s_{2}}_{\gamma/2}}+|h|_{\epsilon,\gamma/2})
  |W^{\epsilon}f|_{L^{2}_{\gamma/2}}. \eeno

{\it Step 3: Estimate of $\mathcal{I}_{3}$.} Recalling \eqref{defintion-of-I3},
by the change of variables $(v,v_{*}) \rightarrow (v^{\prime},v_{*}^{\prime})$ and   $(v,v_{*},\sigma) \rightarrow (v_{*},v,-\sigma)$, we have
\beno
\mathcal{I}_{3} =  2\int b^{\epsilon}(\cos\theta)|v-v_{*}|^{\gamma}(\mu^{ \frac{1}{4}} -
(\mu^{\frac{1}{4}})^{\prime})(\mu^{\frac{1}{4}}g)^{\prime} h_{*}f_{*} d\sigma dv_{*}dv.
\eeno
For ease of notation, let $E_{1} = \{(v,v_{*},\sigma): |v-v_{*}| \geq 1/\epsilon\}, E_{2} = \{(v,v_{*},\sigma): |v-v_{*}| \leq 1/\epsilon, \sin\frac{\theta}{2} \geq |v-v_{*}|^{-1}  \}, E_{3} = \{(v,v_{*},\sigma): |v-v_{*}| \leq 1/\epsilon, \sin\frac{\theta}{2} \leq |v-v_{*}|^{-1} \}$. Then $\mathcal{I}_{3}$ can be decomposed into three parts
 $ \mathcal{I}_{3,1}, \mathcal{I}_{3,2}$ and $\mathcal{I}_{3,3}$ which correspond to $E_{1}, E_{2}$ and $E_{3}$ respectively.

 \underline{Estimate of $\mathcal{I}_{3,1}$}
  By the change of variable $v \rightarrow v^{\prime}$ and the fact $|v^{\prime}-v_{*}|\geq |v-v_{*}| /\sqrt{2}$, we have
\beno
|\mathcal{I}_{3,1}| &\lesssim& \int b^{\epsilon}(\cos\theta)|v^{\prime}-v_{*}|^{\gamma}\mathrm{1}_{|v^{\prime}-v_{*}|\geq (\sqrt{2}\epsilon)^{-1}}|(\mu^{\frac{1}{4}}g)^{\prime} h_{*}f_{*}| d\sigma dv_{*}dv^{\prime}
\\&\lesssim& \epsilon^{-2s} \int |v^{\prime}-v_{*}|^{\gamma}\mathrm{1}_{|v^{\prime}-v_{*}|\geq (\sqrt{2}\epsilon)^{-1}}|(\mu^{\frac{1}{4}}g)^{\prime} h_{*}f_{*}| dv_{*}dv^{\prime}.
\eeno
On one hand, by Cauchy-Schwartz inequality, we have
\ben\label{lessthanep2s}
&&\epsilon^{-2s} \int |v^{\prime}-v_{*}|^{\gamma}\mathrm{1}_{|v^{\prime}-v_{*}|\geq (\sqrt{2}\epsilon)^{-1}}|(\mu^{\frac{1}{4}}g)^{\prime}|dv^{\prime}
\\&\leq& |\mu^{\frac{1}{8}}g|_{L^{2}} \epsilon^{-2s} \{\int |v^{\prime}-v_{*}|^{2\gamma}\mathrm{1}_{|v^{\prime}-v_{*}|\geq (\sqrt{2}\epsilon)^{-1}}(\mu^{\frac{1}{4}})^{\prime}dv^{\prime}\}^{\frac{1}{2}}
\lesssim |\mu^{\frac{1}{8}}g|_{L^{2}} \epsilon^{-2s} \langle v_{*} \rangle^{\gamma}, \nonumber
\een
where we use the fact
$
\langle v^{\prime}-v_{*} \rangle^{2\gamma} \lesssim \langle v^{\prime} \rangle^{|2\gamma|}\langle v_{*} \rangle^{2\gamma}.
$
On the other hand, by Cauchy-Schwartz inequality, we have
\ben\label{lessthanvstar2s}
&&\epsilon^{-2s} \int |v^{\prime}-v_{*}|^{\gamma}\mathrm{1}_{|v^{\prime}-v_{*}|\geq (\sqrt{2}\epsilon)^{-1}}|(\mu^{\frac{1}{4}}g)^{\prime}|dv^{\prime}
\\&\lesssim& \int |v^{\prime}-v_{*}|^{\gamma+2s} |(\mu^{\frac{1}{4}}g)^{\prime}|dv^{\prime}
\leq |\mu^{\frac{1}{8}}g|_{L^{2}}  \{\int |v^{\prime}-v_{*}|^{2\gamma+4s}(\mu^{\frac{1}{4}})^{\prime}dv^{\prime}\}^{\frac{1}{2}}
\lesssim |\mu^{\frac{1}{8}}g|_{L^{2}}  \langle v_{*} \rangle^{\gamma+2s}, \nonumber
\een
where use \eqref{mu-cancel-sigularity}.
With estimates \eqref{lessthanep2s} and  \eqref{lessthanvstar2s} in hand, we have
\beno
|\mathcal{I}_{3,1}| \lesssim |\mu^{\frac{1}{8}}g|_{L^{2}}|W^{\epsilon}h|_{L^{2}_{\gamma/2}}|W^{\epsilon}f|_{L^{2}_{\gamma/2}}.
\eeno

 \underline{Estimate of $\mathcal{I}_{3,2}$.}  Thanks to $ \frac{\sqrt{2}}{2}|v-v_{*}| \leq |v^{\prime}-v_{*}| \leq |v-v_{*}| $ and the change of variable $v \rightarrow v^{\prime}$, we get
\ben\label{i32preliminary}
|\mathcal{I}_{3,2}| &\lesssim& \int b^{\epsilon}(\cos\theta)\mathrm{1}_{\sin(\theta/2) \geq (\sqrt{2}|v^{\prime}-v_{*}|)^{-1} }|v^{\prime}-v_{*}|^{\gamma}\mathrm{1}_{|v^{\prime}-v_{*}|\leq 1/\epsilon}|(\mu^{\frac{1}{4}}g)^{\prime} h_{*}f_{*}| d\sigma dv_{*}dv^{\prime}
\\&\lesssim& \int |v^{\prime}-v_{*}|^{\gamma+2s}\mathrm{1}_{|v^{\prime}-v_{*}|\leq 1/\epsilon}|(\mu^{\frac{1}{4}}g)^{\prime} h_{*}f_{*}| dv_{*}dv^{\prime}. \nonumber
\een
On one hand, similar to the argument in \eqref{lessthanvstar2s}, we have
\ben\label{i32lessthanvstar2s}
 \int |v^{\prime}-v_{*}|^{\gamma+2s}\mathrm{1}_{|v^{\prime}-v_{*}|\leq 1/\epsilon}|(\mu^{\frac{1}{4}}g)^{\prime}| dv^{\prime}
\lesssim |\mu^{\frac{1}{8}}g|_{L^{2}}  \langle v_{*} \rangle^{\gamma+2s}.
\een
On the other hand, if $|v_{*}|\geq 2/\epsilon$, then $|v^{\prime}| \geq |v_{*}| - |v^{\prime}-v_{*}| \geq |v_{*}|/2 \geq 1/\epsilon$, which implies $\mu^{\prime} \leq \mu_{*}^{\frac{1}{4}} \lesssim e^{-\frac{1}{2}\epsilon^{2}}$. Then we deduce that
\ben\label{i32lessthanep2s}
&&
\mathrm{1}_{|v_*|\ge \f2{\epsilon}}\int |v^{\prime}-v_{*}|^{\gamma+2s}\mathrm{1}_{|v^{\prime}-v_{*}|\leq 1/\epsilon}|(\mu^{\frac{1}{4}}g)^{\prime}| dv^{\prime}
\\&\lesssim& \mathrm{1}_{|v_*|\ge \f2{\epsilon}}|\mu^{\frac{1}{8}}g|_{L^{2}}  \{\int |v^{\prime}-v_{*}|^{2\gamma+4s}\mathrm{1}_{|v^{\prime}-v_{*}|\leq 1/\epsilon}(\mu^{\frac{1}{4}})^{\prime}dv^{\prime}\}^{\frac{1}{2}} \nonumber
\\&\lesssim& \mathrm{1}_{|v_*|\ge \f2{\epsilon}}|\mu^{\frac{1}{8}}g|_{L^{2}} \mu_{*}^{\frac{1}{64}} (\epsilon^{-1})^{\gamma+2s+\frac{3}{2}} e^{-1/32\epsilon^{2}} \nonumber
\lesssim \mathrm{1}_{|v_*|\ge \f2{\epsilon}}|\mu^{\frac{1}{8}}g|_{L^{2}} \mu_{*}^{\frac{1}{64}}. \nonumber
\een
With estimates \eqref{i32lessthanvstar2s} and  \eqref{i32lessthanep2s} in hand, we have
\beno
|\mathcal{I}_{3,2}| \lesssim |\mu^{\frac{1}{8}}g|_{L^{2}}|W^{\epsilon}h|_{L^{2}_{\gamma/2}}|W^{\epsilon}f|_{L^{2}_{\gamma/2}}.
\eeno

\underline{Estimate of $\mathcal{I}_{3,3}$.}
By Taylor expansion, one has
\beno
\mu^{ \frac{1}{4}} - (\mu^{\frac{1}{4}})^{\prime} = (\nabla \mu^{ \frac{1}{4}})(v^{\prime})\cdot(v-v^{\prime}) + \int_{0}^{1} (1-\kappa) [(\nabla^{2} \mu^{ \frac{1}{4}})(v(\kappa)):(v-v^{\prime})\otimes(v-v^{\prime})] d\kappa,
\eeno
where $v(\kappa) = v^{\prime} + \kappa(v-v^{\prime})$.
Observe that, for any fixed $v_{*}$, there holds
\beno
\int b^{\epsilon}(\cos\theta)|v-v_{*}|^{\gamma} \mathrm{1}_{|v-v_{*}| \leq 1/\epsilon, \sin(\theta/2) \leq |v-v_{*}|^{-1}}(\nabla \mu^{ \frac{1}{4}})(v^{\prime})\cdot(v-v^{\prime})(\mu^{\frac{1}{4}}g)^{\prime}  d\sigma dv = 0.
\eeno
Thus we have
\beno
|\mathcal{I}_{3,3}| &=&  |\int_{E_{3}\times[0,1]} b^{\epsilon}(\cos\theta)|v-v_{*}|^{\gamma}\mathrm{1}_{|v-v_{*}| \leq 1/\epsilon, \sin(\theta/2) \leq |v-v_{*}|^{-1}}
\\&&\times (1-\kappa)[(\nabla^{2} \mu^{ \frac{1}{4}})(v(\kappa)):(v-v^{\prime})\otimes(v-v^{\prime})] (\mu^{\frac{1}{4}}g)^{\prime} h_{*}f_{*} d\kappa d\sigma dv_{*}dv|
\\&\lesssim& \int b^{\epsilon}(\cos\theta)|v^{\prime}-v_{*}|^{\gamma+2} \sin^{2}\frac{\theta}{2} \mathrm{1}_{|v^{\prime}-v_{*}| \leq 1/\epsilon, \sin(\theta/2) \leq |v^{\prime}-v_{*}|^{-1}}|(\mu^{\frac{1}{4}}g)^{\prime} h_{*}f_{*}|  d\sigma dv_{*}dv^{\prime}
\\&\lesssim& \int |v^{\prime}-v_{*}|^{\gamma+2s}\mathrm{1}_{|v^{\prime}-v_{*}| \leq 1/\epsilon}|(\mu^{\frac{1}{4}}g)^{\prime} h_{*}f_{*}|   dv_{*}dv^{\prime}.
\eeno
Copying the argument applied to \eqref{i32preliminary}, we have
$|\mathcal{I}_{3,3}| \lesssim|\mu^{\frac{1}{8}}g|_{L^{2}}|W^{\epsilon}h|_{L^{2}_{\gamma/2}}|W^{\epsilon}f|_{L^{2}_{\gamma/2}}$.

Patching together the above estimates of $\mathcal{I}_{3,1}, \mathcal{I}_{3,2}$ and $\mathcal{I}_{3,3}$, we have
\beno
|\mathcal{I}_{3}| \lesssim |\mu^{\frac{1}{8}}g|_{L^{2}}|W^{\epsilon}h|_{L^{2}_{\gamma/2}}|W^{\epsilon}f|_{L^{2}_{\gamma/2}}.
\eeno

Recalling \eqref{I-into-I1-I2-I3},
the lemma follows from the above estimates of $\mathcal{I}_{1}, \mathcal{I}_{2}$ and $\mathcal{I}_{3}$.
\end{proof}

\subsubsection{Upper bounds for the nonlinear term $ \Gamma^{\epsilon}(g,h)$}
We are now ready to give the upper bound for the inner product $\langle \Gamma^{\epsilon}(g,h), f\rangle_v$.
\begin{thm}\label{upGammagh}
	 For any  $\eta>0$, the following estimates are valid.
\begin{itemize}
	\item If $\gamma>-\frac{3}{2}$, then
 $|\langle \Gamma^\epsilon(g,h), f\rangle_v| \lesssim  |g|_{L^{2}}|h|_{\epsilon,\gamma/2}| f|_{ \epsilon,\gamma/2}.$
 \item  If $\gamma=-\f32$, then
 \beno |\langle \Gamma^\epsilon(g,h), f\rangle_v| \lesssim |\mu^{\frac{1}{8}}g|_{H^{s_{1}}}(|W^{\epsilon}(D)h|_{H^{s_{2}}_{\gamma/2}}+|h|_{\epsilon,\gamma/2})
  |f|_{ \epsilon,\gamma/2} +
|g|_{L^{2}}|h|_{\epsilon,\gamma/2} |f|_{ \epsilon,\gamma/2},\eeno
where $(s_1, s_2) = (0, \eta)$ or $(\eta, 0)$.
  \item If $-3< \gamma < -\frac{3}{2}$, then
\beno
|\langle  \Gamma^\epsilon(g,h), f\rangle_v| \lesssim |\mu^{\frac{1}{8}}g|_{H^{s_{1}}}(|W^{\epsilon}(D)h|_{H^{s_{2}}_{\gamma/2}}+|h|_{\epsilon,\gamma/2})
|f|_{\epsilon,\gamma/2} +
|g|_{L^{2}}|h|_{\epsilon,\gamma/2} |f|_{ \epsilon,\gamma/2},\eeno
where the constants $s_1, s_2 \geq 0$ verify either $s_{1}+s_{2}=-\gamma-\frac{3}{2}, s_{2} >0$ or  $s_1=-\gamma-\frac{3}{2}+\eta, s_2=0$.
 \end{itemize}
 As a direct application,
 we have
\ben\label{upgammamuff1}
 |\langle \Gamma^{\epsilon}(\mu^{\frac{1}{2}},f), f\rangle_v| \lesssim |f|^{2}_{\epsilon,\gamma/2}.
\een
\end{thm}
\begin{proof} Recalling \eqref{Gamma-to-Q-I}, the estimates of $|\langle  \Gamma^\epsilon(g,h), f\rangle_v|$
 follow directly from  Theorem \ref{upQepsilon} and Lemma \ref{upforI}. By taking $s_2=s_3=0,$ we get \eqref{upgammamuff1}.
\end{proof}

Now we are in a position to prove Theorem \ref{main1}.
\begin{proof}[Proof of Theorem \ref{main1}]  On one hand, by Theorem \ref{equivalencenorm}, we derived that
$ \langle \mathcal{L}^{\epsilon}f, f\rangle_v + |f|^{2}_{L^{2}_{\gamma/2}} \gtrsim |f|^{2}_{\epsilon,\gamma/2}$.
On the other hand, recalling \eqref{DefLep}, by \eqref{l2-upper-bound} and \eqref{upgammamuff1}, we have
$
\langle \mathcal{L}^{\epsilon}f, f\rangle_v \lesssim    |f|^{2}_{\epsilon,\gamma/2},
$
 which ends the proof.
\end{proof}

\subsection{Commutator estimates}  In this subsection, we want to prove

\begin{lem}\label{commutatorgamma}
Let $l \geq 2$. The following commutator estimates are valid.
\begin{itemize} \item If $\gamma \ge -2$, then
		 $|\langle \Gamma^{\epsilon}(g,W_{l}h)-W_{l}\Gamma^{\epsilon}(g,h), f\rangle_v| \lesssim |g|_{L^{2}}|W_{l+\gamma/2}h|_{L^{2}}|f|_{\epsilon,\gamma/2}.$	
\item If $-3 < \gamma < -2$, then
        \beno
		 |\langle \Gamma^{\epsilon}(g,W_{l}h)-W_{l}\Gamma^{\epsilon}(g,h), f\rangle_v|  \lesssim |g|_{L^{2}}|W_{l+\gamma/2}h|_{L^{2}}|f|_{\epsilon,\gamma/2}+ |\mu^{\frac{1}{32}}g|_{H^{s_{1}}}|\mu^{\frac{1}{32}}h|_{H^{s_{2}}}|f|_{\epsilon,\gamma/2},
		\eeno
where the constants  $s_{1}, s_{2} \geq 0$ verify  $s_{1} + s_{2} = -\gamma/2-1$.
\end{itemize}
\end{lem}

This lemma is a consequence of Lemma \ref{commutatorQepsilon} and Lemma \ref{commutatorforI} by recalling \eqref{Gamma-to-Q-I}.
We first prove the commutator estimate for $Q^{\epsilon}$.
\begin{lem}\label{commutatorQepsilon}
	Let $l \geq 2$. The following commutator estimates are valid.
\begin{itemize} \item If $\gamma \ge -2$, then
		 $|\langle Q^{\epsilon}(\mu^{\frac{1}{2}}g,W_{l}h)-W_{l}Q^{\epsilon}(\mu^{\frac{1}{2}}g,h), f\rangle_v| \lesssim |\mu^{\frac{1}{32}}g|_{L^{2}}|W_{l+\gamma/2}h|_{L^{2}}|f|_{\epsilon,\gamma/2}.$	
		\item If $-3 < \gamma < -2$, then
        \beno
		 |\langle Q^{\epsilon}(\mu^{\frac{1}{2}}g,W_{l}h)-W_{l}Q^{\epsilon}(\mu^{\frac{1}{2}}g,h), f\rangle_v|  \lesssim (|\mu^{\frac{1}{32}}g|_{L^{2}}|W_{l+\gamma/2}h|_{L^{2}}+ |\mu^{\frac{1}{32}}g|_{H^{s_{1}}}|\mu^{\frac{1}{32}}h|_{H^{s_{2}}})|f|_{\epsilon,\gamma/2},
		\eeno
where the constants  $s_{1}, s_{2} \geq 0$ verify  $s_{1} + s_{2} = -\gamma/2-1$.
\end{itemize}
\end{lem}
\begin{proof}
Recall $B^{\epsilon,\gamma}= |v-v_{*}|^{\gamma} b^{\epsilon}(\cos\theta)$ and note that
\beno &&\langle Q^{\epsilon}(\mu^{\frac{1}{2}}g,W_{l}h)-W_{l}Q^{\epsilon}(\mu^{\frac{1}{2}}g,h), f\rangle_v  = \int B^{\epsilon,\gamma}(W_{l}-W^{\prime}_{l})\mu_{*}^{\frac{1}{2}}g_{*} h f^{\prime} d\sigma dv_{*} dv
\\&=& \int B^{\epsilon,\gamma}(W_{l}-W^{\prime}_{l})\mu_{*}^{\frac{1}{2}}g_{*} h (f^{\prime}-f) d\sigma dv_{*} dv
 +\int B^{\epsilon,\gamma}(W_{l}-W^{\prime}_{l})\mu_{*}^{\frac{1}{2}}g_{*} h f d\sigma dv_{*} dv
:=\mathcal{A}_{1} + \mathcal{A}_{2}. \eeno

{\it Step 1: Estimate of $\mathcal{A}_{1}$.}
By Cauchy-Schwartz inequality, we have
\beno |\mathcal{A}_{1}| \leq \{\int B^{\epsilon,\gamma} \mu_{*}^{\frac{1}{2}}(f^{\prime}-f)^{2} d\sigma dv_{*} dv\}^{\frac{1}{2}}
\{\int B^{\epsilon,\gamma}(W_{l}-W^{\prime}_{l})^{2}\mu_{*}^{\frac{1}{2}}g^{2}_{*} h^{2}  d\sigma dv_{*} dv\}^{\frac{1}{2}}
:=(\mathcal{A}_{1,1})^{\frac{1}{2}}(\mathcal{A}_{1,2})^{\frac{1}{2}}. \eeno
By the estimate of $\mathcal{I}_{2,1}$ in the proof of Lemma \ref{upforI},  we have $  \mathcal{A}_{1,1} \lesssim |f|^{2}_{\epsilon,\gamma/2}$.
It is easy to derive $ \int b^{\epsilon}(W_{l}-W^{\prime}_{l})^{2}d\sigma \lesssim |v-v_{*}|^{2}\langle v \rangle^{2l-2}\langle v_{*} \rangle^{2l-2}, $
which gives $$\mathcal{A}_{1,2} \lesssim \int |v-v_{*}|^{\gamma+2}\langle v \rangle^{2l-2}\langle v_{*} \rangle^{2l-2}\mu_{*}^{\frac{1}{2}}g^{2}_{*} h^{2}   dv_{*} dv.$$
If $\gamma+2 \geq 0$, there holds $ \mathcal{A}_{1,2} \lesssim |\mu^{\frac{1}{16}}g|^{2}_{L^{2}}|h|^{2}_{L^{2}_{l+\gamma/2}}.$
If $\gamma+2 < 0$, we make the decomposition,
\beno \mathcal{A}_{1,2} &\lesssim& \int |v-v_{*}|^{\gamma+2}\mathrm{1}_{|v-v_{*}|\leq 1}\langle v \rangle^{2l-2}\langle v_{*} \rangle^{2l-2}\mu_{*}^{\frac{1}{2}}g^{2}_{*} h^{2}   dv_{*} dv
\\&&+ \int |v-v_{*}|^{\gamma+2}\mathrm{1}_{|v-v_{*}|\geq 1}\langle v \rangle^{2l-2}\langle v_{*} \rangle^{2l-2}\mu_{*}^{\frac{1}{2}}g^{2}_{*} h^{2}   dv_{*} dv
:= \mathcal{A}_{1,2,1}+\mathcal{A}_{1,2,2}.\eeno
When $|v-v_{*}|\leq 1$, there holds $|v_{*}|\geq |v|-1$, thus $|v_{*}|^{2}\gtrsim |v|^{2}/2$ and $\mu_{*} \lesssim \mu^{\frac{1}{2}}$. Therefore we get
$\langle v \rangle^{2l-2}\langle v_{*} \rangle^{2l-2}\mu_{*}^{\frac{1}{2}} \lesssim  \langle v_{*} \rangle^{4l-4}\mu_{*}^{\frac{1}{8}}\mu^{\frac{1}{16}} \lesssim  \mu_{*}^{1/16}\mu^{\frac{1}{16}},$ which yields
\beno \mathcal{A}_{1,2,1} \lesssim \int |v-v_{*}|^{\gamma+2}\mathrm{1}_{|v-v_{*}|\leq 1}\mu_{*}^{1/16}\mu^{\frac{1}{16}}g^{2}_{*} h^{2}dv_{*} dv = \int |v-v_{*}|^{\gamma+2}\mathrm{1}_{|v-v_{*}|\leq 1}G_{*} H dv_{*} dv, \eeno
where $G = \mu^{\frac{1}{16}}g^{2}$ and $H = \mu^{\frac{1}{16}}h^{2}$. We assert that for $s_1,s_2\ge0$ with $s_1+s_2=-(\gamma+2)/2$,
\ben \label{a-elmentary-result}
|\mathcal{A}_{1,2,1}|\lesssim |\mu^{\frac{1}{32}}g|^{2}_{H^{s_{1}}} |\mu^{\frac{1}{32}}h|^{2}_{H^{s_{2}}}.\een
Indeed, if $s_1\in (0, -(\gamma+2)/2)$,
by Hardy-Littlewood-Sobolev inequality and Sobolev embedding theorem, we get the result. If $s_1=0$ or $s_1=-(\gamma+2)/2$, by Hardy's inequality, one has
\beno |\mathcal{A}_{1,2,1}|\lesssim |\sqrt{G}|^{2}_{H^{-(\gamma+2)/2}}|\sqrt{H}|^{2}_{L^2};\quad |\mathcal{A}_{1,2,1}|\lesssim |\sqrt{G}|^{2}_{L^2}|\sqrt{H}|^{2}_{H^{-(\gamma+2)/2}}. \eeno
Patching together these two cases, we get \eqref{a-elmentary-result}.

When $|v-v_{*}|\geq 1$, there holds $|v-v_{*}|^{\gamma+2} \sim \langle v-v_{*} \rangle^{\gamma+2} \lesssim \langle v \rangle^{\gamma+2}\langle v_{*} \rangle^{|\gamma+2|}$, which yields
\beno \mathcal{A}_{1,2,2} \lesssim  \int \langle v \rangle^{2l+\gamma}\langle v_{*} \rangle^{2l-2+|\gamma+2|}\mu_{*}^{\frac{1}{2}}g^{2}_{*} h^{2}   dv_{*} dv
\lesssim |\mu^{\frac{1}{32}}g|^{2}_{L^{2}} |h|^{2}_{L^{2}_{l+\gamma/2}}.\eeno

Patching together the above estimates, we have if $\gamma+2 \geq 0$,   $|\mathcal{A}_{1}| \lesssim |\mu^{\frac{1}{16}}g|_{L^{2}}|h|_{L^{2}_{l+\gamma/2}}|f|_{\epsilon,\gamma/2}$,
and if $\gamma+2 < 0$,  $|\mathcal{A}_{1}| \lesssim (|\mu^{\frac{1}{32}}g|_{H^{s_{1}}} |\mu^{\frac{1}{32}}h|_{H^{s_{2}}}+|\mu^{\frac{1}{32}}g|_{L^{2}} |h|_{L^{2}_{l+\gamma/2}})|f|_{\epsilon,\gamma/2}.$

{\it Step 2: Estimate of $\mathcal{A}_{2}$.}
By Taylor expansion, one has
\beno W^{\prime}_{l} - W_{l} = (\nabla W_{l})(v)\cdot(v^{\prime}-v) +\int_{0}^{1}(1-\kappa)(\nabla^{2}W_{l})(v(\kappa)):(v^{\prime}-v)\otimes(v^{\prime}-v)d\kappa, \eeno
where $v(\kappa) = v + \kappa (v^{\prime}-v)$. Thus we have
\beno \mathcal{A}_{2} &=& -\int B^{\epsilon,\gamma}(\nabla W_{l})(v)\cdot(v^{\prime}-v)\mu_{*}^{\frac{1}{2}}g_{*} h f d\sigma dv_{*} dv
\\&&-\int B^{\epsilon,\gamma}(1-\kappa)(\nabla^{2}W_{l})(v(\kappa)):(v^{\prime}-v)\otimes(v^{\prime}-v)\mu_{*}^{\frac{1}{2}}g_{*} h f d\kappa d\sigma dv_{*} dv
:= \mathcal{A}_{2,1}+\mathcal{A}_{2,2}.\eeno

\underline{Estimate of $\mathcal{A}_{2,1}$.} Thanks to the fact that there exists a constant $C(\epsilon)$  with $|C(\epsilon)|\lesssim 1$ such that
\ben\label{canvvx} \int b^{\epsilon}(\cos\theta) (v^{\prime}-v) d\sigma = -(v-v_{*})\int b^{\epsilon}(\cos\theta)\sin^{2}(\theta/2)d\sigma = -(v-v_{*}) C(\epsilon),\een
we have \beno |\mathcal{A}_{2,1}|  &\lesssim& \int |v-v_{*}|^{\gamma+1}\langle v \rangle^{l-1}\langle v_{*} \rangle^{l-1}\mu_{*}^{\frac{1}{2}}|g_{*} h f| d\sigma dv_{*} dv
\\&\lesssim&\int |v-v_{*}|^{\gamma+1}\mathrm{1}_{|v-v_{*}|\leq 1}\langle v \rangle^{l-1}\langle v_{*} \rangle^{l-1}\mu_{*}^{\frac{1}{2}}|g_{*} h f|   dv_{*} dv
\\&&+ \int |v-v_{*}|^{\gamma+1}\mathrm{1}_{|v-v_{*}|\geq 1}\langle v \rangle^{l-1}\langle v_{*} \rangle^{l-1}\mu_{*}^{\frac{1}{2}}|g_{*} h f|   dv_{*} dv
:=\mathcal{A}_{2,1,1}+\mathcal{A}_{2,1,2}.\eeno
When $|v-v_{*}|\leq 1$, as before, one has
$ \langle v \rangle^{l-1}\langle v_{*} \rangle^{l-1}\mu_{*}^{\frac{1}{2}} \lesssim  \langle v_{*} \rangle^{2l-2}\mu_{*}^{\frac{1}{8}}\mu^{\frac{1}{16}}
\lesssim  \mu_{*}^{\frac{1}{16}}\mu^{\frac{1}{16}}$.
Thus by Cauchy-Schwartz inequality, we have
\beno \mathcal{A}_{2,1,1} &\lesssim& \int |v-v_{*}|^{\gamma+1}\mathrm{1}_{|v-v_{*}|\leq 1}\mu_{*}^{\frac{1}{16}}\mu^{\frac{1}{16}}|g_{*} h f|   dv_{*} dv
\\&\leq& \{\int |v-v_{*}|^{\gamma+2}\mathrm{1}_{|v-v_{*}|\leq 1}\mu_{*}^{\frac{1}{16}}\mu^{\frac{1}{16}}g^{2}_{*} h^{2}    dv_{*} dv\}^{\f12}
\{\int |v-v_{*}|^{\gamma}\mathrm{1}_{|v-v_{*}|\leq 1}\mu_{*}^{\frac{1}{16}}\mu^{\frac{1}{16}}  f^{2}   dv_{*} dv\}^{\f12}
\\&\lesssim& \{\int |v-v_{*}|^{\gamma+2}\mathrm{1}_{|v-v_{*}|\leq 1}\mu_{*}^{\frac{1}{16}}\mu^{\frac{1}{16}}g^{2}_{*} h^{2}dv_{*} dv\}^{\frac{1}{2}}|\mu^{\frac{1}{32}}f|_{L^{2}}. \eeno
By the argument for $\mathcal{A}_{1,2,1}$, we conclude that if $\gamma \geq -2$,
$\mathcal{A}_{2,1,1} \lesssim  |\mu^{\frac{1}{32}}g|_{L^{2}} |\mu^{\frac{1}{32}}h|_{L^{2}}|\mu^{\frac{1}{32}}f|_{L^{2}} $, and if $-3 < \gamma < - 2$,
$ \mathcal{A}_{2,1,1} \lesssim  |\mu^{\frac{1}{32}}g|_{H^{s_{1}}} |\mu^{\frac{1}{32}}h|_{H^{s_{2}}}|\mu^{\frac{1}{32}}f|_{L^{2}}.$
By nearly the same argument as that for $\mathcal{A}_{1,2,2}$, we have
$ \mathcal{A}_{2,1,2} \lesssim |\mu^{\frac{1}{32}}g|_{L^{2}} |h|_{L^{2}_{l+\gamma/2}} |f|_{L^{2}_{\gamma/2}}.$

\underline{Estimate of $\mathcal{A}_{2,2}$.} Since $|(\nabla^{2}W_{l})(v(\kappa))| \lesssim \langle v(\kappa) \rangle^{l-2} \lesssim \langle v \rangle^{l-2}\langle v_{*} \rangle^{l-2}$ and $|v^{\prime}-v|^{2} = \sin^{2}\frac{\theta}{2} |v-v_{*}|^{2}$, we have
\beno |\mathcal{A}_{2,2}| &\lesssim& \int b^{\epsilon}(\cos\theta)\sin^{2}\frac{\theta}{2}|v-v_{*}|^{\gamma+2}\langle v \rangle^{l-2}\langle v_{*} \rangle^{l-2}\mu_{*}^{\frac{1}{2}}|g_{*} h f |d\sigma dv_{*} dv
\\&\lesssim& \int |v-v_{*}|^{\gamma+2}\langle v \rangle^{l-2}\mu_{*}^{\frac{1}{8}}|g_{*} h f | dv_{*} dv.
\eeno
Thanks to $\gamma+2 > -1$, using Cauchy-Schwartz inequality and \eqref{mu-cancel-sigularity},
we have
$\int |v-v_{*}|^{\gamma+2}\mu_{*}^{\frac{1}{8}}|g_{*}| dv_{*} \lesssim \langle v \rangle^{\gamma+2}|\mu^{\frac{1}{16}}g|_{L^{2}}$
and thus $ |\mathcal{A}_{2,2}| \lesssim |\mu^{\frac{1}{16}}g|_{L^{2}} |h|_{L^{2}_{l+\gamma/2}} |f|_{L^{2}_{\gamma/2}}. $
Patching together the estimates of $\mathcal{A}_{2,1,1},\mathcal{A}_{2,1,2}$ and $\mathcal{A}_{2,2}$, we conclude that if $\gamma \geq - 2$,
$ |\mathcal{A}_{2}| \lesssim |\mu^{\frac{1}{32}}g|_{L^{2}} |h|_{L^{2}_{l+\gamma/2}} |f|_{L^{2}_{\gamma/2}}$,
and if $-3 < \gamma < - 2$,   $|\mathcal{A}_{2}| \lesssim (|\mu^{\frac{1}{32}}g|_{H^{s_{1}}} |\mu^{\frac{1}{32}}h|_{H^{s_{2}}}+|\mu^{\frac{1}{32}}g|_{L^{2}} |h|_{L^{2}_{l+\gamma/2}}) |f|_{L^{2}_{\gamma/2}}$.

The lemma   follows by patching together the estimates of $\mathcal{A}_{1}$ and $\mathcal{A}_{2}$.
\end{proof}
The next lemma gives the commutator estimate for $\mathcal{I}(g,h,f)$.
\begin{lem}\label{commutatorforI}
	Let $l \geq 1$, there holds
		\beno
		|\mathcal{I}(g,W_{l}h,f)-\mathcal{I}(g,h,W_{l}f)| \lesssim |g|_{L^{2}}|W_{l+\gamma/2}h|_{L^{2}}|W^{\epsilon}f|_{L^{2}_{\gamma/2}}.
		\eeno	
\end{lem}
\begin{proof}
By the definition of $\mathcal{I}(g,h,f)$ and the fact $(\mu^{\frac{1}{2}})_{*}^{\prime} - \mu_{*}^{\frac{1}{2}} =((\mu^{\frac{1}{4}})_{*}^{\prime} - \mu_{*}^{\frac{1}{4}})^{2}+2\mu_{*}^{\frac{1}{4}}((\mu^{\frac{1}{4}})_{*}^{\prime} - \mu_{*}^{\frac{1}{4}})$, we have
\beno
\mathcal{I}(g,W_{l}h,f)-\mathcal{I}(g,h,W_{l}f) &=& \int B^{\epsilon,\gamma}((\mu^{\frac{1}{2}})_{*}^{\prime} - \mu_{*}^{\frac{1}{2}}) (W_{l}-W^{\prime}_{l})g_{*} h f^{\prime} d\sigma dv_{*} dv
\\&=&  \int B^{\epsilon,\gamma}((\mu^{\frac{1}{4}})_{*}^{\prime} - \mu_{*}^{\frac{1}{4}})^{2}(W_{l}-W^{\prime}_{l})g_{*} h f^{\prime} d\sigma dv_{*} dv
\\&&
 + 2 \int B^{\epsilon,\gamma}\mu_{*}^{\frac{1}{4}}((\mu^{\frac{1}{4}})_{*}^{\prime} - \mu_{*}^{\frac{1}{4}})(W_{l}-W^{\prime}_{l})g_{*} h f^{\prime} d\sigma dv_{*} dv
:= \mathcal{A}_{1} + 2\mathcal{A}_{2}.
\eeno

{\it Step 1: Estimate of $\mathcal{A}_{1}$.}
By Cauchy-Schwartz inequality, we have
\beno |\mathcal{A}_{1}| &\leq& \{\int B^{\epsilon,\gamma} ((\mu^{\frac{1}{4}})_{*}^{\prime} - \mu_{*}^{\frac{1}{4}})^{2} (f^{2})^{\prime} d\sigma dv_{*} dv\}^{\frac{1}{2}}
\\&&\times\{\int B^{\epsilon,\gamma}((\mu^{\frac{1}{4}})_{*}^{\prime} - \mu_{*}^{\frac{1}{4}})^{2}(W_{l}-W^{\prime}_{l})^{2}g^{2}_{*} h^{2}  d\sigma dv_{*} dv\}^{\frac{1}{2}}
:=(\mathcal{A}_{1,1})^{\frac{1}{2}}(\mathcal{A}_{1,2})^{\frac{1}{2}}. \eeno
By the change of variables $(v,v_{*}) \rightarrow (v_{*}^{\prime},v^{\prime})$ and Proposition \ref{lowerboundpart1}, we have
\beno \mathcal{A}_{1,1} = \int B^{\epsilon,\gamma} ((\mu^{\frac{1}{4}})^{\prime} - \mu^{\frac{1}{4}})^{2} f^{2}_{*} d\sigma dv_{*} dv
\lesssim |W^{\epsilon}f|^{2}_{L^{2}_{\gamma/2}}.\eeno
Thanks to  $((\mu^{\frac{1}{4}})_{*}^{\prime} - \mu_{*}^{\frac{1}{4}})^{2} = ((\mu^{\frac{1}{8}})_{*}^{\prime} + \mu_{*}^{\frac{1}{8}})^{2}((\mu^{\frac{1}{8}})_{*}^{\prime} - \mu_{*}^{\frac{1}{8}})^{2} \leq 2 ((\mu^{\frac{1}{4}})_{*}^{\prime} + \mu_{*}^{\frac{1}{4}})((\mu^{\frac{1}{8}})_{*}^{\prime} - \mu_{*}^{\frac{1}{8}})^{2}$, we have
\beno \mathcal{A}_{1,2} &\lesssim& \int B^{\epsilon,\gamma}\mu_{*}^{\frac{1}{4}}((\mu^{\frac{1}{8}})_{*}^{\prime} - \mu_{*}^{\frac{1}{8}})^{2}(W_{l}-W^{\prime}_{l})^{2}g^{2}_{*} h^{2}  d\sigma dv_{*} dv \\&&+ \int B^{\epsilon,\gamma}(\mu^{\frac{1}{4}})_{*}^{\prime}((\mu^{\frac{1}{8}})_{*}^{\prime} - \mu_{*}^{\frac{1}{8}})^{2}(W_{l}-W^{\prime}_{l})^{2}g^{2}_{*} h^{2}  d\sigma dv_{*} dv
:= \mathcal{A}_{1,2,1} + \mathcal{A}_{1,2,2}.\eeno
  Thanks to the facts $|v-v^{\prime}_{*}| \sim |v-v_{*}|$ and
\ben\label{roughaboutwl}
(W_{l}-W^{\prime}_{l})^{2} \lesssim \min\{\sin^{2}\frac{\theta}{2}|v-v_{*}^{\prime}|^{2}\langle v \rangle^{2l-2} \langle v_{*}^{\prime} \rangle^{2l-2}, \sin^{2}\frac{\theta}{2}\langle v \rangle^{2l} \langle v_{*}^{\prime} \rangle^{2l}\},
\\ \label{roughaboutmu}
((\mu^{\frac{1}{8}})_{*}^{\prime} - \mu_{*}^{\frac{1}{8}})^{2} \lesssim \min\{ \sin^{2}\frac{\theta}{2}|v-v_{*}^{\prime}|^{2}, 1\},
\een
we assert
\ben\label{kernelestimate2}
\mathcal{B} := \int B^{\epsilon,\gamma}(\mu^{\frac{1}{4}})_{*}^{\prime}((\mu^{\frac{1}{8}})_{*}^{\prime} - \mu_{*}^{\frac{1}{8}})^{2}(W_{l}-W^{\prime}_{l})^{2}  d\sigma \lesssim \langle v \rangle^{2l+\gamma},
\een
which yields $ \mathcal{A}_{1,2,2} \lesssim |g|^{2}_{L^{2}}|h|^{2}_{L^{2}_{l+\gamma/2}}.$
In fact, by \eqref{roughaboutwl} and \eqref{roughaboutmu}, on one hand, there holds
\beno \mathcal{B}  \lesssim \int b^{\epsilon}(\cos\theta)\sin^{4}\frac{\theta}{2}|v-v^{\prime}_{*}|^{\gamma+4}(\mu^{\frac{1}{4}})_{*}^{\prime} \langle v \rangle^{2l-2} \langle v_{*}^{\prime} \rangle^{2l-2} d\sigma. \eeno
When $|v-v_{*}|\leq 1$, there holds $|v-v_{*}^{\prime}|\leq 1$, $|v-v^{\prime}_{*}|^{\gamma+4}\leq 1$ and $\langle v \rangle \sim \langle v^{\prime}_{*} \rangle$, thus $\langle v \rangle^{2l-2}  \lesssim \langle v \rangle^{2l+\gamma} \langle v^{\prime}_{*} \rangle^{-2-\gamma}$, which yields
\beno \mathcal{B}  \lesssim \int b^{\epsilon}(\cos\theta)\sin^{4}\frac{\theta}{2}(\mu^{\frac{1}{4}})_{*}^{\prime} \langle v \rangle^{2l+\gamma} \langle v_{*}^{\prime} \rangle^{2l-4-\gamma} d\sigma \lesssim \int b^{\epsilon}(\cos\theta)\sin^{4}\frac{\theta}{2} \langle v \rangle^{2l+\gamma} d\sigma \lesssim \langle v \rangle^{2l+\gamma}.  \eeno
By \eqref{roughaboutwl} and \eqref{roughaboutmu}, on the other hand, there holds
$\mathcal{B}  \lesssim \int b^{\epsilon}(\cos\theta)\sin^{2}\frac{\theta}{2}|v-v^{\prime}_{*}|^{\gamma}(\mu^{\frac{1}{4}})_{*}^{\prime} \langle v \rangle^{2l} \langle v_{*}^{\prime} \rangle^{2l} d\sigma. $
When $|v-v_{*}|\geq 1$, there holds   $|v-v^{\prime}_{*}|^{\gamma}\sim \langle v-v^{\prime}_{*} \rangle^{\gamma} \lesssim \langle v \rangle^{\gamma}\langle v^{\prime}_{*} \rangle^{|\gamma|}$, which yields
 \beno \mathcal{B}  \lesssim \int b^{\epsilon}(\cos\theta)\sin^{2}\frac{\theta}{2}(\mu^{\frac{1}{4}})_{*}^{\prime}  \langle v \rangle^{2l+\gamma} \langle v_{*}^{\prime} \rangle^{2l+|\gamma|} d\sigma \lesssim \int b^{\epsilon}(\cos\theta)\sin^{2}\frac{\theta}{2} \langle v \rangle^{2l+\gamma} d\sigma \lesssim \langle v \rangle^{2l+\gamma}. \eeno
Now the estimate \eqref{kernelestimate2} is proved. Note that \eqref{roughaboutwl} and \eqref{roughaboutmu} are still valid if $v_{*}^{\prime}$ is replaced by $v_{*}$ on the right-hand sides.
Then similar to \eqref{kernelestimate2}, we can prove
\beno
\int B^{\epsilon,\gamma}\mu_{*}^{\frac{1}{4}}((\mu^{\frac{1}{8}})_{*}^{\prime} - \mu_{*}^{\frac{1}{8}})^{2}(W_{l}-W^{\prime}_{l})^{2}  d\sigma \lesssim \langle v \rangle^{2l+\gamma}\mu_{*}^{\frac{1}{8}},
\eeno
which yields $ \mathcal{A}_{1,2,1} \lesssim |\mu^{\frac{1}{16}}g|^{2}_{L^{2}}|h|^{2}_{L^{2}_{l+\gamma/2}}.$
Patching together the estimates of $\mathcal{A}_{1,2,1}$ and $\mathcal{A}_{1,2,2}$, we arrive at
 $\mathcal{A}_{1,2} \lesssim |g|^{2}_{L^{2}}|h|^{2}_{L^{2}_{l+\gamma/2}}. $
From which together with the estimate of $\mathcal{A}_{1,1}$,  we conclude that
 $|\mathcal{A}_{1}| \lesssim |g|_{L^{2}}|h|_{L^{2}_{l+\gamma/2}}|W^{\epsilon}f|_{L^{2}_{\gamma/2}}.$

{\it Step 2: Estimate of $\mathcal{A}_{2}$.} By Cauchy-Schwartz inequality, we have
\beno |\mathcal{A}_{2}| &\leq& \{\int B^{\epsilon,\gamma}\mu_{*}^{\frac{1}{4}}((\mu^{\frac{1}{4}})_{*}^{\prime} - \mu_{*}^{\frac{1}{4}})^{2} |g_{*}| (f^{2})^{\prime} d\sigma dv_{*} dv\}^{\frac{1}{2}}
\\&&\times\{\int B^{\epsilon,\gamma}\mu_{*}^{\frac{1}{4}}(W_{l}-W^{\prime}_{l})^{2} |g_{*}| h^{2}  d\sigma dv_{*} dv\}^{\frac{1}{2}}
:=(\mathcal{A}_{2,1})^{\frac{1}{2}}(\mathcal{A}_{2,2})^{\frac{1}{2}}. \eeno

\underline{Estimate of $\mathcal{A}_{2,1}$.} By the change of variable $v \rightarrow v^{\prime}$, we have
\beno \mathcal{A}_{2,1} \lesssim \int b^{\epsilon}(\cos\theta)|v-v_{*}|^{\gamma}\mu_{*}^{\frac{1}{4}}((\mu^{\frac{1}{4}})_{*}^{\prime} - \mu_{*}^{\frac{1}{4}})^{2}|g_{*}| f^{2} d\sigma dv_{*} dv.\eeno
By \eqref{freq-used-inequality} and Proposition \ref{symbol}, we get
\beno\mathcal{A}_{2,1} &\lesssim& \int (\mathrm{1}_{|v-v_{*}|\leq \sqrt{2}} |v-v_{*}|^{\gamma+2}  + \langle v-v_{*} \rangle^{\gamma} (W^{\epsilon})^{2}(v-v_{*}))\mu_{*}^{\frac{1}{4}} |g_{*}| f^{2} dv_{*} dv\\
&\lesssim& \int ( |v-v_{*}|^{\gamma+2} \mu_{*}^{\frac{1}{8}} \mu^{\frac{1}{16}}  + \langle v \rangle^{\gamma} (W^{\epsilon})^{2}(v) \mu_{*}^{\frac{1}{8}}) |g_{*}| f^{2} dv_{*} dv
\lesssim|\mu^{\frac{1}{16}}g|_{L^2}|W^\epsilon f|_{L^2_{\gamma/2}}^2, \eeno
where we use the fact  $\mu_{*}^{\frac{1}{4}}  \lesssim \mu_{*}^{\frac{1}{8}} \mu^{\frac{1}{16}}$ when $|v-v_{*}|\leq \sqrt{2}$ and the following estimate
\ben \label{C-S-then-cancel}
\int |v-v_{*}|^{\gamma+2} \mu_{*}^{\frac{1}{8}} |g_{*}| dv_{*} \leq \{\int |v-v_{*}|^{2\gamma+4}  \mu_{*}^{\frac{1}{8}}  dv_{*}\}^{\frac{1}{2}}
\{\int \mu_{*}^{\frac{1}{8}} g^{2}_{*} dv_{*} \}^{\frac{1}{2}} \lesssim \langle v \rangle^{\gamma+2} |\mu^{\frac{1}{16}}g|_{L^{2}} \een
given by Cauchy-Schwartz inequality and \eqref{mu-cancel-sigularity}.

 \underline{Estimate of $\mathcal{A}_{2,2}$.}
By Taylor expansion, when $l \geq 1$, it is easy to check that
\beno (W_{l}-W^{\prime}_{l})^{2} \lesssim \sin^{2}\frac{\theta}{2}|v-v_{*}|^{2}(\langle v \rangle^{2l-2} + \langle v_{*} \rangle^{2l-2}) \lesssim \sin^{2}\frac{\theta}{2}|v-v_{*}|^{2}\langle v \rangle^{2l-2} \langle v_{*} \rangle^{2l-2}.\eeno
Thus we have
\beno \mathcal{A}_{2,2} &\lesssim& \int b^{\epsilon}(\cos\theta) \sin^{2}\frac{\theta}{2}|v-v_{*}|^{\gamma+2}\langle v \rangle^{2l-2} \langle v_{*} \rangle^{2l-2}\mu_{*}^{\frac{1}{4}} |g_{*}| h^{ 2}  d\sigma dv_{*} dv
\\&\lesssim& \int |v-v_{*}|^{\gamma+2}\langle v \rangle^{2l-2} \mu_{*}^{\frac{1}{8}} |g_{*}| h^{ 2} dv_{*} dv \lesssim |\mu^{\frac{1}{16}}g|_{L^{2}}|h|^{2}_{L^{2}_{l+\gamma/2}}.
\eeno
where we use \eqref{C-S-then-cancel}.
Putting together the estimates of $\mathcal{A}_{2,1}$ and $\mathcal{A}_{2,2}$, we arrive at
\beno |\mathcal{A}_{2}| \lesssim |\mu^{\frac{1}{16}}g|_{L^{2}}|h|_{L^{2}_{l+\gamma/2}}|W^{\epsilon}f|_{L^{2}_{\gamma/2}}.\eeno

The lemma follows   the estimates of $\mathcal{A}_{1}$ and $\mathcal{A}_{2}$.
\end{proof}

\section{Diversity of longtime behavior of the semi-group}
In this section, we will give the proof to Theorem \ref{main2}.
 We begin with a technical lemma for a commutator estimate.

\begin{lem}\label{CommSemi} Let $-2s \leq \gamma <0$.
Let $\chi_M(v):=\chi(v/M)$ with $(\chi, M)=(\phi, 1/\epsilon)$ or $(\chi, M)=(1-\phi, 1/\epsilon)$ or $(\chi, M)=(\psi, 2^{j})$. Here $\phi$ and $\psi$  are defined in \eqref{function-phi-psi} and $j$ verifies  $2^{j\gamma}\leq \epsilon^{2s}$.  For any $0<\eta<1$, there hold
\ben \label{commutator-Gamma-local} |\langle [\Gamma^\epsilon(g, \cdot), \chi_M]h, f \chi_M\rangle_v |\lesssim \eta^{-1}\epsilon^{2s}(|g|_{L^2}^2|h|_{\epsilon,\gamma/2}^2
+|g|_{\epsilon,\gamma/2}^2|h|_{L^2}^2)+\eta|f \chi_M|_{\epsilon,\gamma/2}^2,
\\
\label{commutator-L-local}
|\langle [\mathcal{L}^\epsilon ,\chi_M]f, f\chi_M\rangle_v |\lesssim \eta^{-1}\epsilon^{2s}|f|_{\epsilon,\gamma/2}^2+\eta|f\chi_M|_{\epsilon,\gamma/2}^2. \een
\end{lem}
\begin{proof} Since $-2s \leq \gamma <0$ and $2^{j\gamma}\leq \epsilon^{2s}$, then $2^j \ge 1/\epsilon$.
For simplicity, we denote $\mathcal{I}(g,h,f):= \langle [\Gamma^\epsilon(g, \cdot), \chi_M]h, f \chi_M\rangle_v = \langle \Gamma^\epsilon(g,h\chi_M)-\chi_M \Gamma^\epsilon(g,h), f\chi_M\rangle_v$.

Direct calculation gives
\beno
\mathcal{I}(g,h,f)=\int B^{\epsilon,\gamma}\big[(g\mu^{\f12})_*+g_*\big((\mu^{\f12})_*'
-(\mu^{\f12})_*\big)\big] h(f\chi_M)^{\prime}\big(-(\chi_M)^{\prime}+\chi_M\big) d\sigma dv_*dv.\eeno
By Cauchy-Schwartz inequality, we get that
\ben \nonumber
|\mathcal{I}(g,h,f)| &\lesssim& \bigg(\int B^{\epsilon,\gamma} g_*^2h^2(\mu_*^{\f12}+(\mu^{\f12})_*') \big((\chi_M)'-\chi_M\big)^2d\sigma dv_*dv\bigg)^{\f12} \bigg(\int B^{\epsilon,\gamma} \big[\mu^\f12_* \big((f\chi_M)'-f\chi_M\big)^2
\\ \nonumber
&&+(f\chi_M)^2\big((\mu^{\f14})_*'-(\mu^{\f14})_*\big)^2\big]d\sigma dv_*dv\bigg)^{\f12}+\big|\int B^{\epsilon,\gamma}(g\mu^{\f12})_*hf\chi_M\big((\chi_M)'-\chi_M\big)d\sigma dv_*dv\big|
\\ \label{into-three-parts}
&\lesssim& \eta|f\chi_M|_{\epsilon,\gamma/2}^2+\eta^{-1} \mathcal{J}(g,h) + |\mathcal{K}(g,h,f)|,\een
where
\beno
\mathcal{J}(g,h):=\int B^{\epsilon,\gamma} g_*^2h^2 (\mu_*^{\f12}+(\mu^{\f12})_*')\big((\chi_M)'-\chi_M\big)^2d\sigma dv_*dv,
\\ \mathcal{K}(g,h,f):=
\int B^{\epsilon,\gamma}(g\mu^{\f12})_*hf\chi_M\big((\chi_M)'-\chi_M\big)d\sigma dv_*dv. \eeno
We now go to estimate $\mathcal{J}(g,h)$ and $\mathcal{K}(g,h,f)$.

{\it \underline{Estimate of $\mathcal{J}(g,h)$}.} We separate the integration domain of $\mathcal{J}(g,h)$ into three regions: $\{|v_*|\le \delta M\}$, $\{|v_*|\ge \delta M, |v|\le \delta|v_*|\}$ and $\{|v_*|\ge \delta M, |v|\ge \delta|v_*|\}$ where $\delta=\frac{1}{100}$ in the whole proof. In the region $\{|v_*|\le \delta M\}$, we notice that
\beno \chi_M(v^{\prime})-\chi_M(v)=\int_0^1 (\na \chi_M)(v(\kappa))\cdot (v'-v)d\kappa,\eeno where $v(\kappa)=v+\kappa(v'-v)$.
By the support of $\na \chi$, one has $|v(\kappa)|\sim M$. Therefore we have $|v|\sim|v(\kappa)|\sim M$.  In the region $\{|v_*|\ge \delta M, |v|\le \delta|v_*|\}$, we deduce that $|v_*|\sim |v-v_*|\sim |v-v_*'|\sim |v_*'|$.  In the region $\{|v_*|\ge \delta M, |v|\ge \delta|v_*|\}$, there holds $|v|\ge \delta^2M$. Putting together all the facts, since $\chi \leq 1$ and $|\na \chi_M| \lesssim M^{-1}$,
we get
\beno  |\chi_M(v^{\prime})-\chi_M(v)|^2&\lesssim& \mathrm{1}_{|v_*|\le \delta M}\mathrm{1}_{|v|\sim M} M^{-2}\theta^2 |v-v_*|^2+(\mathrm{1}_{|v_*|\ge \delta M}\mathrm{1}_{|v_*'|\sim |v_*|}\mathrm{1}_{|v|\le \delta|v_*|}\\&&+\mathrm{1}_{|v_*|\ge \delta M}\mathrm{1}_{|v|\ge \delta^2 M}\mathrm{1}_{|v|\ge \delta|v_*|} )\min\{1,M^{-2}|v-v_*|^2\theta^2\},\eeno
from which together with  Proposition \ref{symbol}, thanks to the factor $\mu_*^{\f12}+(\mu^{\f12})_*'$,
we have for any $a\ge0$,
\ben \label{J-estimate-three-cases}
\mathcal{J}(g,h)&\lesssim& |g\mathrm{1}_{|\cdot|\le \delta M}|_{L^2_{-a}}^2|h \mathrm{1}_{|\cdot|\sim M}|^2_{L^2_{\gamma/2+a}}+e^{-\delta^3M^2}|W^\epsilon g \mathrm{1}_{|\cdot|\ge \delta M}|^2_{L^2_{\gamma/2+a}}|  h|^2_{L^2_{-a}}
\\ \nonumber
&&+M^{-2s}|g \mathrm{1}_{|\cdot|\ge \delta M}|^2_{L^2_{-a}}|W^\epsilon
h \mathrm{1}_{|\cdot|\ge \delta^2 M}|^2_{L^2_{a+\gamma/2}}. \een

{\it \underline{Estimate of $\mathcal{K}(g,h,f)$}.} We decompose  the integration domain of $\mathcal{K}(g,h,f)$ into two regions: $\{|v_*|\le \delta M\}$ and  $\{|v_*|\ge \delta M\}$. Correspondingly, $\mathcal{K}(g,h,f)=\mathcal{K}_{1}(g,h,f)+\mathcal{K}_{2}(g,h,f)$.

We first deal with  $\mathcal{K}_1(g,h,f)$ whose integration domain is  $\{|v_*|\le \delta M\}$. In this case, if $|v|\sim M$ or $|v(\kappa)|\sim M$, then $|v|\sim |v-v_*|\sim M$. By \eqref{canvvx} and Taylor expansion
\ben\label{taylorchi} \chi_M(v^{\prime})-\chi_M(v)=(\na \chi_M)(v)\cdot (v-v')+\int_0^1 (1-\kappa)(\na^2 \chi_M)(v(\kappa)):  (v'-v)\otimes (v'-v)d\kappa, \een
we infer that $|\int B^{\epsilon,\gamma} (\chi_M(v')-\chi_M(v))d\sigma|\lesssim \mathrm{1}_{|v_*|\le \delta M} \mathrm{1}_{|v|\sim |v-v_*|\sim M}  \langle v \rangle^{\gamma}$, which yields that
\ben \label{K1-estimate}
|\mathcal{K}_1(g,h,f)|
\lesssim |g\mu^{\f12}\mathrm{1}_{|\cdot|\le \delta M}|_{L^1}
|h\mathrm{1}_{|\cdot|\sim M}|_{L^2_{\gamma/2}}|f\chi_M|_{L^2_{\gamma/2}} \lesssim \epsilon^{s}|g |_{L^2}
|W^\epsilon h |_{L^2_{\gamma/2}}|f\chi_M|_{L^2_{\gamma/2}}. \een

We turn to estimate $\mathcal{K}_2(g,h,f)$ in which $|v_*|\ge \delta M$.
When $(\chi, M)=(\phi, 1/\epsilon)$, the support of $\chi_M$ is in the ball $B_{\delta^{-1}M}$. In this case, we have
\ben \label{decomposition-theta-large-and-small}
\mathcal{K}_2(g,h,f) &=& \mathcal{K}_{2,1}+\mathcal{K}_{2,2},
\\ \nonumber
\mathcal{K}_{2,1} &:=&
\int B^{\epsilon,\gamma} \mathrm{1}_{|v_*|\ge \delta M}\mathrm{1}_{\sin(\theta/2) \le |v-v_*|^{-1}} (g\mu^{\f12})_*\mathrm{1}_{|v|\le \delta^{-1} M}hf\chi_M\big((\chi_M)'-\chi_M\big)d\sigma dv_*dv,
\\ \nonumber \mathcal{K}_{2,2} &:=&
\int B^{\epsilon,\gamma} \mathrm{1}_{|v_*|\ge \delta M}\mathrm{1}_{\sin(\theta/2) \ge |v-v_*|^{-1}} (g\mu^{\f12})_*\mathrm{1}_{|v|\le \delta^{-1} M}hf\chi_M\big((\chi_M)'-\chi_M\big)d\sigma dv_*dv.  \een
By Taylor expansion \eqref{taylorchi}   and \eqref{canvvx}, one has
\ben \label{estimate-of-k21} |\mathcal{K}_{2,1}|&\lesssim&
\big|\int |v-v_*|^\gamma \mathrm{1}_{|v_*|\ge \delta M}\mathrm{1}_{|v|\le \delta^{-1} M} |(g\mu^{\f12})_*hf\chi_M|(|v-v_*|^{2s}+ |v-v_*|^{2s-1})dv_*dv\big|
\\ \nonumber
&\lesssim& e^{-\delta^3M^2} (|g\mu^{\f14}\mathrm{1}_{|\cdot|\ge \delta M}|_{L^1_{1+\gamma+2s}}+|g\mu^{\f14}\mathrm{1}_{|\cdot|\ge \delta M}|_{L^2_{1+\gamma+2s}})|h\mathrm{1}_{|\cdot|\le \delta^{-1}M}|_{L^2_{\gamma/2+s}}|f\chi_M|_{L^2_{\gamma/2+s}}
\\ \nonumber
&\lesssim& \epsilon^{s}|g |_{L^2}
|W^\epsilon h |_{L^2_{\gamma/2}}|W^\epsilon f\chi_M|_{L^2_{\gamma/2}}.
\een
For $\mathcal{K}_{2,2}$, thanks to the fact that $\gamma+2s\ge0$, it is not difficult to check that
\ben \label{estimate-of-k22}
 |\mathcal{K}_{2,2}| \lesssim e^{-\delta^3M^2} |g\mu^{\f14}\mathrm{1}_{|\cdot|\ge \delta M}|_{L^1 } |h\mathrm{1}_{|\cdot|\le \delta^{-1}M}|_{L^2_{\gamma/2+s}}|f\chi_M|_{L^2_{\gamma/2+s}} \lesssim \epsilon^{s}|g |_{L^2}
|W^\epsilon h |_{L^2_{\gamma/2}}|W^\epsilon f\chi_M|_{L^2_{\gamma/2}}.
 \een

When $(\chi, M)=(1-\phi, 1/\epsilon)$ or $(\chi, M)=(\psi, 2^{j})$, the support of $\chi_M$ is outside of the ball $B_{\delta M}$ and so
\beno
\mathcal{K}_{2}(g,h,f):=
\int B^{\epsilon,\gamma}\mathrm{1}_{|v_*|\ge \delta M} \mathrm{1}_{|v|\ge \delta M} (g\mu^{\f12})_*hf\chi_M\big((\chi_M)'-\chi_M\big)d\sigma dv_*dv.
\eeno
When $|v-v_{*}| \geq 1$, then $|v-v_{*}|^{\gamma} \sim \langle v-v_{*}\rangle^{\gamma} \lesssim \langle v\rangle^{\gamma} \langle v_{*}\rangle^{|\gamma|} $ and so
\ben \label{geq-1-part}
&&|\int B^{\epsilon,\gamma}\mathrm{1}_{|v_*|\ge \delta M} \mathrm{1}_{|v|\ge \delta M} \mathrm{1}_{|v-v_{*}| \geq 1} (g\mu^{\f12})_*hf\chi_M\big((\chi_M)'-\chi_M\big)d\sigma dv_*dv|
\\ \nonumber &\lesssim& \epsilon^{-2s} e^{-\delta^3M^2} \int \langle v\rangle^{\gamma} \mathrm{1}_{|v_*|\ge \delta M} \mathrm{1}_{|v|\ge \delta M} |(g\mu^{\f14})_*hf\chi_M| dv_*dv \lesssim \epsilon^{s}|g |_{L^2}
|W^\epsilon h |_{L^2_{\gamma/2}}|W^\epsilon f\chi_M|_{L^2_{\gamma/2}}.
\een
When $|v-v_{*}| \leq 1$, then $\mu^{\f12}_* \lesssim \mu^{\f18}_* \mu^{\f18}$. We can use the decomposition used in \eqref{decomposition-theta-large-and-small} to get
\ben \label{leq-1-part}
&&|\int B^{\epsilon,\gamma}\mathrm{1}_{|v_*|\ge \delta M} \mathrm{1}_{|v|\ge \delta M} \mathrm{1}_{|v-v_{*}| \leq 1} (g\mu^{\f12})_*hf\chi_M\big((\chi_M)'-\chi_M\big)d\sigma dv_*dv|
\\ \nonumber
&\lesssim& \epsilon^{s}|g |_{L^2}
|W^\epsilon h |_{L^2_{\gamma/2}}|W^\epsilon f\chi_M|_{L^2_{\gamma/2}}.
\een
Patching together \eqref{estimate-of-k21} and \eqref{estimate-of-k22}, patching together
 \eqref{geq-1-part} and \eqref{leq-1-part}, for the three cases,
we conclude that
\ben \label{K2-estimate}
|\mathcal{K}_{2}(g,h,f)| \lesssim \epsilon^{s} |g|_{L^2}
|W^\epsilon h |_{L^2_{\gamma/2}}|W^\epsilon f\chi_M|_{L^2_{\gamma/2}}. \een
Patching together \eqref{K1-estimate} and \eqref{K2-estimate}, we get
\ben \label{K-estimate}
|\mathcal{K}(g,h,f)| \lesssim \epsilon^{s} |g|_{L^2}
|W^\epsilon h |_{L^2_{\gamma/2}}|W^\epsilon f\chi_M|_{L^2_{\gamma/2}}. \een
Plugging \eqref{J-estimate-three-cases} and \eqref{K-estimate} into \eqref{into-three-parts}, we get \eqref{commutator-Gamma-local}. Recalling \eqref{DefLep},
plugging \eqref{J-estimate-three-cases} and \eqref{K-estimate} into \eqref{into-three-parts}, by taking $(g,h)=(\mu^{\f12}, f)$, we get
\ben \label{commu-L1}
|\langle [\mathcal{L}^\epsilon_{1}, \chi_M]f, f\chi_M\rangle_v | =  |\mathcal{I}(\mu^{\f12},f,f)| \lesssim \eta^{-1} \epsilon^{2s}|f|_{\epsilon,\gamma/2}^2+\eta|f\chi_M|_{\epsilon,\gamma/2}^2. \een

Using \eqref{l2-upper-bound}, we get
\ben \nonumber
|\langle [\mathcal{L}^\epsilon_{2}, \chi_M]f, f\chi_M\rangle_v | &=& |\langle \mathcal{L}^\epsilon_{2} \chi_M f, f\chi_M\rangle_v - \langle \mathcal{L}^\epsilon_{2} f, \chi_M  f\chi_M\rangle_v |
\\ \nonumber &=& |\langle \mathcal{L}^\epsilon_{2} \chi_M f, (1-\chi_M)f\chi_M\rangle_v - \langle \mathcal{L}^\epsilon_{2} (1-\chi_M) f, \chi_M  f\chi_M\rangle_v|
\\ \nonumber &\lesssim& |\chi_M f|_{L^{2}_{\gamma/2}} |(1-\chi_M)f\chi_M|_{L^{2}_{\gamma/2}} + |(1-\chi_M) f|_{L^{2}_{\gamma/2}} |\chi_M  f\chi_M|_{L^{2}_{\gamma/2}}
\\ \label{commu-L2}
&\lesssim& \epsilon^{s} |W^{\epsilon} f|_{L^{2}_{\gamma/2}} |W^{\epsilon} f\chi_M|_{L^{2}_{\gamma/2}}
\lesssim \eta^{-1} \epsilon^{2s}|f|_{\epsilon,\gamma/2}^2+ \eta|f\chi_M|_{\epsilon,\gamma/2}^2, \een
Patching together \eqref{commu-L1} and \eqref{commu-L2}, we arrive at \eqref{commutator-L-local}.
\end{proof}

In the rest of this section, we  set $f=e^{-\mathcal{L}^{\epsilon}t}f_0$ with $f_0\in {\mathcal{N}}^{\perp}$. Then $f$ verifies that  $f\in {\mathcal{N}}^{\perp}$ and \begin{equation}\label{linlinBo}
\left\{ \begin{aligned}
&\pa_t f+\mathcal{L}^\epsilon f=0 ;\\
&f|_{t=0} = f_{0}.
\end{aligned} \right.
\end{equation}

Now we are in a position to prove  \eqref{semigroupLe1} and \eqref{semigroupLe4} in Theorem \ref{main2}.
\begin{proof}[Proof of Theorem \ref{main2} {\bf (Part I)}] We first prove \eqref{semigroupLe1}.
Since $f=e^{-\mathcal{L}^{\epsilon}t}f_0 \in {\mathcal{N}}^{\perp}$,
by Proposition \ref{coercvityforLep}, there is a universal constant $\lambda_{1}>0$ such that
 $\f{d}{dt}|f|_{L^2}^2+\lambda_1|f|_{\epsilon,\gamma/2}^2\le 0 $ and  thus
for any $t \geq 0$,
\ben \label{basic-L2-estimate}
|f(t)|_{L^2}^2 + \lambda_{1} \int_{0}^{t} |f(\tau)|^2_{\epsilon,\gamma/2} d\tau \leq |f_{0}|_{L^2}^2.
\een
Recall that $f^l(v)=\phi(\epsilon v)f(v)$ and $f^h=f-f^l$. Recalling \eqref{linlinBo}, we have
\beno
\pa_t f^l+\mathcal{L}^\epsilon f^l=[\mathcal{L}^\epsilon,\phi(\epsilon \cdot)]f, \quad \pa_t f^h+\mathcal{L}^\epsilon f^h=[\mathcal{L}^\epsilon,1-\phi(\epsilon \cdot)]f.
 \eeno
Thanks to Theorem \ref{main1}, Proposition \ref{coercvityforLep}
 and the fact $|f^h|_{\epsilon,\gamma/2} \gtrsim \epsilon^{-s}|f^h|_{L^2_{\gamma/2}}$, we have
\ben \label{low-part-coercivity}
\langle \mathcal{L}^\epsilon f^l, f^l\rangle_v \gtrsim |f^l|^2_{\epsilon,\gamma/2}-C|f^l|^2_{L^2_{\gamma/2}}, \quad
\langle \mathcal{L}^\epsilon f^l, f^l\rangle_v \gtrsim  |(\mathbb{I}-\mathbb{P})f^l|^2_{\epsilon,\gamma/2} \geq |(\mathbb{I}-\mathbb{P})f^l|^2_{L^2_{\gamma/2}},
\\ \label{high-part-coercivity}
\langle \mathcal{L}^\epsilon f^h, f^h\rangle_v \gtrsim |f^h|_{\epsilon,\gamma/2}^2 -C \epsilon^{2s} |f^h|_{\epsilon,\gamma/2}^2. \een
From \eqref{low-part-coercivity} and the identity $\mathbb{P}(f^l+f^h)=\mathbb{P}f=0$, we derive that
\ben \label{low-part-coercivity-final}
\langle \mathcal{L}^\epsilon f^l, f^l\rangle_v \gtrsim |f^l|^2_{\epsilon,\gamma/2}-C|\mathbb{P}f^l|_{L^2_{\gamma/2}}^2 = |f^l|^2_{\epsilon,\gamma/2}-C|\mathbb{P}f^h|_{L^2_{\gamma/2}}^2
\geq |f^l|^2_{\epsilon,\gamma/2}-C \epsilon^{2s} |f^h|_{\epsilon,\gamma/2}^2. \een
  Thanks to \eqref{commutator-L-local} in Lemma \ref{CommSemi}, \eqref{low-part-coercivity-final} and \eqref{high-part-coercivity}, for some universal constant $\lambda_{2}>0$, we get
\beno  &&\f{d}{dt}|f^l|_{L^2}^2+\lambda_{2}|f^l|_{\epsilon,\gamma/2}^2 \lesssim \epsilon^{2s}(|f^h|_{\epsilon,\gamma/2}^2+|f|_{\epsilon,\gamma/2}^2) \lesssim \epsilon^{2s} |f|_{\epsilon,\gamma/2}^2,\\
&&\f{d}{dt}|f^h|_{L^2}^2+\lambda_{2}|f^h|_{\epsilon,\gamma/2}^2 \lesssim \epsilon^{2s}(|f^h|_{\epsilon,\gamma/2}^2+|f|_{\epsilon,\gamma/2}^2) \lesssim \epsilon^{2s} |f|_{\epsilon,\gamma/2}^2.
\eeno
Since $\gamma \geq -2s$, then $\lambda_{2}|f^l|_{\epsilon,\gamma/2}^2 \geq c |f^l|_{L^2}^2$ for some universal constant $c$. Recalling \eqref{basic-L2-estimate}, we get \eqref{semigroupLe1} by using Gr\"{o}nwall's inequality.

Next we want to prove \eqref{semigroupLe4}. Recalling \eqref{linlinBo}, it is easy to check that
\beno \pa_t \mathcal{P}_j f+\mathcal{L}^\epsilon \mathcal{P}_jf=[\mathcal{L}^\epsilon, \psi(2^{-j}\cdot)]f. \eeno   Recall that
 $2^j\ge 1/\epsilon$.
Thanks to Theorem \ref{main1} and Lemma \ref{CommSemi}, for some constant $C_{0}>0$,
we obtain
\beno \f{d}{dt}|\mathcal{P}_j f(t)|_{L^2}^2+C_{0}|\mathcal{P}_jf|_{\epsilon,\gamma/2}^2\gtrsim -\epsilon^{2s}|f|_{\epsilon,\gamma/2}^2. \eeno Observe that
$|W^\epsilon\mathcal{P}_jf|^2_{L^2_{\gamma/2}}\sim \epsilon^{-2s}2^{j\gamma}|\mathcal{P}_jf|_{L^2}^2$ and
$|W^{\epsilon}(D)W_{\gamma/2}\mathcal{P}_jf|_{L^2}^2+|W^\epsilon((-\triangle_{\mathbb{S}^2})^{\frac{1}{2}})W_{\gamma/2}\mathcal{P}_jf|^2\lesssim\\ \epsilon^{-2s}2^{j\gamma}|\mathcal{P}_jf|_{L^2}^2$.
We are led to
\beno \f{d}{dt}|\mathcal{P}_j f(t)|_{L^2}^2 \gtrsim -\epsilon^{2s}|f|^{2}_{\epsilon,\gamma/2}-\epsilon^{-2s}2^{j\gamma}|\mathcal{P}_j f|_{L^2}^2.\eeno
From which together with \eqref{basic-L2-estimate}, we get
 $|\mathcal{P}_jf(t)|_{L^2}^2 \ge  |\mathcal{P}_jf_0|_{L^2}^2-C\epsilon^{-2s}2^{j\gamma}t-C\epsilon^{2s}$, which yields
 \eqref{semigroupLe4} for $t \in [0, C^{-1}\eta 2^{-j\gamma}\epsilon^{2s}]$.
\end{proof}

To complete the proof of Theorem \ref{main2}, we need the following proposition.
\begin{prop}\label{propODE} Let $c_1, \mathcal{C}_2$ and $p$ be three universal and positive constants. Consider the ordinary differential inequality
\begin{equation}\label{ODEdecay}
\left\{ \begin{aligned}
&\f{d}{dt} Y+c_1Y_1+\mathcal{C}_2^{-1}Y_2^{1+\f1{p}}\le 0 , \\
&Y|_{t=0} = 1,
\end{aligned} \right.
\end{equation}
where $Y=Y_1+Y_2$ and $Y,Y_1,Y_2\ge0$.  We have
\begin{enumerate}
\item if $\mathcal{C}_2\ll 1$, then
there exists a critical time $t_*=O(p(-\ln\mathcal{C}_2))$ such that
\ben\label{decayY1} Y(t)\lesssim  e^{-c_1t/8}\mathrm{1}_{t\le t_*}+C(c_1,p) \mathcal{C}_2^p(1+t)^{-p}\mathrm{1}_{t\ge t_*}. \een
\item if $\mathcal{C}_2\sim 1$, then
\ben\label{decayY2} Y(t)\lesssim C(c_1,p) (1+t)^{-p}. \een
\end{enumerate}
\end{prop}
\begin{proof}   It is easy to check that $Y(t)$ is a strictly decreasing function before it vanishes. Let $t_j$ be the time  such that $Y(t_j)=2^{-j}$ for $j\in \N$.

To obtain the desired result, the key point is to  give an estimate for $t_j$. Since $Y=Y_1+Y_2$, one has $Y_1\ge \f12 Y$ or $Y_2\ge\f12 Y$. Then for $t\in[t_j,t_{j+1}]$, $(c_1Y_1+\mathcal{C}_2^{-1}Y_2^{1+\f1{p}})(t)\ge \min\{\f{c_{1}}{2}Y(t_{j+1}), \mathcal{C}_2^{-1}(Y(t_{j+1})/2)^{1+\f1{p}}\}$. By \eqref{ODEdecay}, we obtain that for $t\in[t_j,t_{j+1}]$, $(-Y'(t))^{-1}\le  (c_1\f12Y(t_{j+1}))^{-1}+ (\mathcal{C}_2^{-1}(Y(t_{j+1})/2)^{1+\f1{p}})^{-1}$.
By mean value theorem, there exists a $\tilde{t}\in[t_j,t_{j+1}]$ such that
\beno Y(t_j)-Y(t_{j+1})=Y'(\tilde{t})(t_j-t_{j+1}), \eeno
which yields that
$ t_{j+1}-t_j\le 4c_1^{-1}+\epsilon^{2s}4^{1+\f1{p}}c_2^{-1}2^{\f{j}{p}}.$
From which we obtain
\beno t_N\le 4c_1^{-1}N+  \mathcal{C}_2 4^{1+\f1{p}} 2^{\f{N}{p}} (1-2^{-p})^{-1} = 4c_1^{-1}N + C(p)\mathcal{C}_22^{\f{N}{p}},\eeno
where $C(p):=4^{1+\f1{p}}(1-2^{-p})^{-1}.$

We first consider the case that $\mathcal{C}_2\sim 1$. In this case, we have
$t_N\le C(c_1,p)2^{\f{N}{p}}$. In other words, $Y(t_N)\le C(c_1,p)(1+t_N)^{-p}$. Thanks to the monotonic property of $Y(t)$, we obtain  \eqref{decayY2}.

Next we handle the case $\mathcal{C}_2\ll 1$. Set $H(x)=C(p)\mathcal{C}_22^{\f{x}{p}}-4c_1^{-1}x$. Since $\mathcal{C}_2\ll 1$, we have $H(1)\le0$. Thus for $x\ge1$, there exists a unique $x_*>1$ such that if $x\le x_*$, $H(x)\le 0$ and if $x\ge x_*$, $H(x)\ge0$. Moreover,   there exist two constants $C_1$ and $C_2$ depending only on $c_1$ and $p$ such that $C_1p(-\ln \mathcal{C}_2+C(c_1,p))  \le x_*\le  C_2p(-\ln \mathcal{C}_2+C(c_1,p))$.

  From the above argument, we get that
 if $1\le N\le N_*:= [x_*]$,   $t_N\le 8c_1^{-1}N$ and if $N\ge N_*+1$,
$t_N\le  2\mathcal{C}_2C(p)2^{\f{N}{p}}$.
For $N_*-1\le N\le N_*+1$, we have
\beno t_N \le  4c_1^{-1}(N_*+1)+\mathcal{C}_2C(p)2^{\f{N_*+1}{p}} \le 2\mathcal{C}_2C(c_1,p)2^{\f{N_*+1}{p}}\le 2\mathcal{C}_2C(c_1,p)2^{\f{N+2}{p}},\eeno which yields that for $N\ge N_*-1$, $t_N\le (1+2^{\f{2}{p}})\mathcal{C}_2C(c_1,p)2^{\f{N}{p}}$.
 Thanks to the fact that $Y(t)$ is a strictly decreasing function before it vanishes, we obtain that if $t\le t_{N_*}$, $Y(t)\lesssim 2^{-c_1t/8}$ and if $t\ge t_{N_*-1}$, $Y(t)\lesssim  C(c_1,p)\mathcal{C}_2^p(1+t)^{-p}$.
We conclude that for $t\ge 0$, \beno Y(t)\lesssim  2^{- c_1 t/8}\mathrm{1}_{t\le t_{N_*}}+C(c_1,p)\mathcal{C}_2^p(1+t)^{-p}\mathrm{1}_{t\ge t_{N_*}},\eeno
where $t_{N_*}\le 8c_1^{-1}N_*=O(p(-\ln \mathcal{C}_2+C(c_1,p)))$. On the other hand,  for $t\le \f8{c_1\ln 2}p(-\ln \mathcal{C}_2)$, we have
$2^{c_1t/8}\le  \mathcal{C}_2^{-p}\lesssim C(c_1,p)\mathcal{C}_2^{-p}(1+t)^{p}$. In other words, for $t\le \f8{c_1\ln 2}p(-\ln \mathcal{C}_2)$,
$2^{-c_1  t/8}\ge C(c_1,p)\mathcal{C}_2^{-p}(1+t)^{-p}$. Therefore we deduce that there exists a time $t_*=O(p(-\ln \mathcal{C}_2))$ such that \eqref{decayY1} holds. The proof of the proposition is complete now.
\end{proof}

We have two remarks on Proposition \ref{propODE}.
\begin{rmk}\label{not-exactly-equal} If $Y=Y_1+Y_2$ in Proposition \ref{propODE} is changed to $Y \sim Y_1+Y_2$, the results still hold true at the price of different constants appearing on the right-hand sides.
\end{rmk}
The following remark shows that estimate \eqref{decayY1} is sharp for \eqref{ODEdecay}.
\begin{rmk} 
We consider the following special case:
\beno 
\left\{ \begin{aligned}
&\f{d}{dt} Y+Y_1+\epsilon^{-2s}Y_2^{2} = 0,\\
&Y|_{t=0} = 1,
\end{aligned} \right.
\eeno
where we take $c_1 =p=Y(0)=1$ and $\mathcal{C}_2=\epsilon^{2s}$ in \eqref{ODEdecay}. Here $\epsilon>0$ is sufficiently small.
Let us impose $Y_1=\epsilon^{-2s}Y_2^{2}$. Since $Y_{1}+Y_{2}=Y$, we get
  $Y_{2}=\frac{-1+\sqrt{1+4\epsilon^{-2s}Y}}{2\epsilon^{-2s}}$,
which  yields
\beno Y_1+\epsilon^{-2s}Y_2^{2}=2\epsilon^{-2s}Y_2^{2}=\frac{1+2\epsilon^{-2s}Y-\sqrt{1+4\epsilon^{-2s}Y}}{\epsilon^{-2s}}.\eeno
Now let $X=\epsilon^{-2s}Y$. Then we have the following ODE
\beno
\left\{ \begin{aligned}
&\f{d}{dt} X+1+2X- \sqrt{1+4X}= 0, \\
&X|_{t=0} = \epsilon^{-2s}.
\end{aligned} \right.
\eeno
If we set $f(x)=1+2x- \sqrt{1+4x}$, then one has
  $f^{\prime}(x)=2-2(1+4x)^{-\frac{1}{2}}, ~~f^{\prime \prime}(x)=4(1+4x)^{-\frac{3}{2}},~~f^{(3)}(x)=-24(1+4x)^{-5/2},~~f^{(4)}(x)=240(1+4x)^{-7/2}$.
By Taylor expansion, one has
\beno f(x) &=& f(0)+f^{\prime}(0)x+\frac{f^{\prime\prime}(0)}{2}x^{2}+\frac{f^{(3)}(0)}{6}x^{3}+\frac{1}{6}\int_{0}^{x}(x-t)^{3}f^{(4)}(t)dt\\
&=&2x^{2}-4x^{3}+\frac{1}{6}\int_{0}^{x}(x-t)^{3}f^{(4)}(t)dt.\eeno
Since $0\leq f^{(4)}(t) \leq 240$, we have
$ 2x^{2}-4x^{3} \leq f(x) \leq 2x^{2}-4x^{3}+10x^{4}.$
If $x\leq \frac{1}{4}$, then $4x^{3}\leq x^{2}$ and $10x^{4}\leq x^{2},$
which gives
\ben \label{small-value-square}x^{2} \leq 1+2x- \sqrt{1+4x} \leq  3x^{2}, ~~ x \leq \frac{1}{4}.\een
Let $g(x)=f(x)-x/4$, if $x\geq \frac{1}{4}$, then
$ g^{\prime}(x)=1\frac{1}{4}-2(1+4x)^{-\frac{1}{2}} \geq 1\frac{1}{4}-\sqrt{2}> 0,$
which yields
\beno g(x)\geq g(\frac{1}{4})=\frac{3}{2}-\sqrt{2}-\frac{1}{16}>0.\eeno
On the other hand, if $x\geq \frac{1}{4}$, then $1+2x \leq 6x$.
Therefore we have
\ben \label{large-value-linear}x/4 \leq 1+2x- \sqrt{1+4x} \leq 6x,~~ x \geq \frac{1}{4}.\een

Suppose $t_{*}$ is the critical time such that $X(t_{*})=\frac{1}{4}$, then by \eqref{large-value-linear}, we get
\beno \frac{d}{dt} X+X/4 \leq \frac{d}{dt} X+1+2X- \sqrt{1+4X}= 0 \leq \frac{d}{dt} X+6X, ~~t \leq t_{*},\eeno
which yields
$ -6 \leq \frac{d}{dt} \ln X \leq -\frac{1}{4}, ~~t \leq t_{*}$.
Integrating over $[0,t]$ and recalling $X(0)=\epsilon^{-2s}$, we have
\ben\label{exp1-exponetial-decay}\epsilon^{-2s}\exp(-6t)\leq X(t)\leq \epsilon^{-2s}\exp(-t/4),~~t \leq t_{*},\een
By \eqref{small-value-square},  we get
\beno \frac{d}{dt} X+X^{2} \leq \frac{d}{dt} X+1+2X- \sqrt{1+4X}= 0 \leq \frac{d}{dt} X+3X^{2},~~t \geq t_{*}\eeno
which indicates
\beno-3 \leq \frac{d}{dt} (-\frac{1}{X}) \leq -1, ~~t \geq t_{*}.\eeno
Integrating over $[t_{*},t]$, we have
\beno
\frac{1}{4+3(t-t_{*})}\leq X(t)\leq \frac{1}{4+(t-t_{*})},~~t \geq t_{*}.\eeno
By \eqref{exp1-exponetial-decay}, recalling $X(t_{*})=\frac{1}{4}$, we have
\beno \frac{-2s\ln\epsilon-\ln \frac{1}{4}}{6}\leq t_{*}\leq 4(-2s\ln\epsilon-\ln \frac{1}{4}),\eeno
which yields $t_{*}\sim -2s\ln\epsilon$ since $\epsilon$ is small enough.
Recalling $X=\epsilon^{-2s}Y$, we have
\beno  e^{-6t}\mathrm{1}_{t\le t_*}+\epsilon^{2s}\frac{1}{4+3(t-t_{*})}\mathrm{1}_{t> t_*}\leq Y(t)\leq e^{-t/4}\mathrm{1}_{t\le t_*}+\epsilon^{2s}\frac{1}{4+(t-t_{*})}\mathrm{1}_{t> t_*}. \eeno
Comparing which with \eqref{decayY1}, we conclude that estimate \eqref{decayY1} is sharp for \eqref{ODEdecay}.
\end{rmk}

We are in a position to complete the proof of Theorem \ref{main2}.

\begin{proof}[Proof of Theorem \ref{main2} {\bf (Part II)}] By Theorem \ref{main1}, Lemma \ref{commutatorgamma} and \eqref{l2-upper-bound}, for $l\ge 2$, we get
\beno \f{d}{dt}|f|_{L^2_l}^2+\lambda_{3}|f|_{\epsilon,\gamma/2+l}^2\lesssim |f|_{L^2_{l+\gamma/2}}^2,  \eeno
for some universal constant $\lambda_{3}>0$. Observing that
\beno |f|_{L^2_{l+\gamma/2}}^2\lesssim  |f^h|_{L^2_{l+\gamma/2}}^2+\eta |f^l|^2_{L^2_{\gamma/2+l+s}}+C_\eta |f^l|^2_{L^2_{\gamma/2+s}}.\eeno
By taking $\eta$ small enough, when $\epsilon>0$ is small enough,
we infer that
  $\f{d}{dt}|f|_{L^2_l}^2\lesssim |f^l|^2_{L^2_{\gamma/2+s}} \lesssim |f|_{\epsilon,\gamma/2}^2$. Recalling \eqref{basic-L2-estimate},
we have $|f(t)|_{L^2_l}^2\lesssim |f_0|_{L^2_l}^2$ for any $t\ge0$. Recalling
 $\f{d}{dt}|f|_{L^2}^2+\lambda_1|f|_{\epsilon,\gamma/2}^2\le 0 $
and using the interpolation inequality $|f|_{L^2}\le |f|_{L^2_{\gamma/2}}^{\f{p}{p+1}}|f|_{L^2_{-\gamma p/2 }}^{\f1{p+1}}$,
since $\gamma+2s \geq0$, for some universal constants $C_{1}, C_{2}$,
we get
\beno \f{d}{dt}|f|_{L^2}^2+C_{1}|f^l|_{L^2}^2+ C_{2} |f_0|_{L^2_{-\gamma p/2}}^{-2/p}\epsilon^{-2s}|f^h|^{2+\f{2}{p}} \leq 0.\eeno
Let $Y(t)=|f(t)|_{L^2}^2/|f(t_1)|_{L^2}^2$ and then we obtain
\beno \f{d}{dt}Y(t)+C_{1} Y_1(t)+\mathcal{C}_2^{-1}Y_2(t)^{1+1/p}\le     0.\eeno
where $Y_1(t)=|f^l(t)|_{L^2}^2/|f(t_1)|_{L^2}^2$, $Y_2(t)=|f^h(t)|_{L^2}^2/|f(t_1)|_{L^2}^2$ and $\mathcal{C}_2^{-1}=C_{2}|f(t_1)|_{L^2}^{2/p}|f_0|_{L^2_{-\gamma p/2}}^{-2/p}\epsilon^{-2s}$.
From which together with Proposition \ref{propODE} and Remark \ref{not-exactly-equal}, we get \eqref{semigroupLe2} and \eqref{semigroupLe3}, which completes the proof of Theorem \ref{main2}.
\end{proof}

\section{Nonlinear Boltzmann equation in the perturbation framework}
In this section, we will prove Theorem \ref{main3}. In subsection 4.1, we establish
global well-posedness and propagation of regularity for the Boltzmann equation \eqref{linearizedBE}. In subsection 4.2, we derive global dynamics by using Proposition
\ref{propODE}. Subsection 4.3 is devoted to the global asymptotic formula which describes the limit that $\epsilon$ goes to zero.

\subsection{Global well-posedenss and propagation of regularity}
The main task is to provide the {\it a priori} estimates for the equation \eqref{linearizedBE}. We start with the following linear equation
\ben \label{lBE}\partial_{t}f + v\cdot \nabla_{x} f + \mathcal{L}^{\epsilon}f= g. \een
Here $g$ is given and $f$ is unknown.

\subsubsection{Estimate for the linear equation}
Suppose $f$ is a solution to \eqref{lBE}. Recalling \eqref{DefProj}, we set $f_{1} :=\mathbb{P} f$ and $f_{2} := f - \mathbb{P} f$.
The {\it a priori} estimate for \eqref{lBE} can be stated as follows:

\begin{prop}\label{essential-estimate-of-micro-macro} Let $N \geq 1$ and $f$
be a solution to \eqref{lBE}.
Then for $M$ large enough, there holds
\beno 
&&\frac{d}{dt}(M\|f\|^{2}_{H^{N}_{x}L^{2}}+\mathcal{I}_{N}(f))+ \frac{1}{2}(|\nabla_{x}(a,b,c)|^{2}_{H^{N-1}_{x}}+\|f_{2}\|^{2}_{H^{N}_{x}L^2_{\epsilon,\gamma/2}}) \\&\lesssim& \sum_{|\alpha| \leq N}
|(\pa^{\alpha}g, \pa^{\alpha}f)|+ \sum_{|\alpha| \leq N-1}\sum_{j=1}^{13}
\int_{\mathbb{T}^{3}}|\langle  \pa^{\alpha}g, e_j\rangle_{v}|^{2} dx, \nonumber 
\eeno
 where $M\|f\|^{2}_{H^{N}_{x}L^{2}}+\mathcal{I}_{N}(f)\sim \|f\|^{2}_{H^{N}_{x}L^{2}}$,
$\mathcal{I}_{N}(f)$ is a functional defined in \eqref{interactive-INf} and $\{e_{j}\}_{1\leq j \leq 13}$ is defined  explicitly by
\beno e_{1} = \mu^{\frac{1}{2}}, e_{2} = v_{1}\mu^{\frac{1}{2}}, e_{3} = v_{2}\mu^{\frac{1}{2}},e_{4} = v_{3}\mu^{\frac{1}{2}},   e_{5} = v_{1}^{2}\mu^{\frac{1}{2}}, e_{6} = v_{2}^{2}\mu^{\frac{1}{2}},e_{7} = v_{3}^{2}\mu^{\frac{1}{2}},\\ e_{8} = v_{1}v_{2}\mu^{\frac{1}{2}}, e_{9} = v_{2}v_{3}\mu^{\frac{1}{2}},e_{10} = v_{3}v_{1}\mu^{\frac{1}{2}},  e_{11} = |v|^{2}v_{1}\mu^{\frac{1}{2}}, e_{12} = |v|^{2}v_{2}\mu^{\frac{1}{2}},e_{13} = |v|^{2}v_{3}\mu^{\frac{1}{2}}. \eeno
\end{prop}
The proof  of Proposition \ref{essential-estimate-of-micro-macro} will be postponed a little bit. We first recall some basics of macro-micro decomposition. By \eqref{DefProj}, the macro part is defined as
\ben \label{definition-f-1} f_{1}(t,x,v) = \{a(t,x) + b(t,x) \cdot v + c(t,x)|v|^{2}\}\mu^{\frac{1}{2}},\een which solves
\ben \label{macro-micro-LBE-2} \partial_{t}f_{1} + v\cdot \nabla_{x} f_{1}  = -\partial_{t}f_{2} + l + g,\een
 where $ l = - v\cdot \nabla_{x} f_{2} - \mathcal{L}^{\epsilon}f_{2}$.

Let $A = (a_{ij})_{1\leq i \leq 13, 1\leq j \leq 13}$ be the $13 \times 13$ matrix defined by $a_{ij} = \langle e_{i}, e_{j} \rangle_{v} $ and $y$ be the column vector with $13$ components $\partial_{t} a, \{\partial_{t}b_{i}+ \partial_{i} a \}_{1\leq i \leq 3}, \{\partial_{t}c+ \partial_{i} b_{i} \}_{1\leq i \leq 3},  \{\partial_{i}b_{j}+ \partial_{j} b_{i} \}_{1\leq i < j  \leq 3}, \{\partial_{i}c \}_{1\leq i \leq 3}$. Let $e$ be the column vector with $13$ components $\{e_{j}\}_{j=1}^{13}$. Plugging \eqref{definition-f-1} into \eqref{macro-micro-LBE-2}, we get
\ben
\label{macro-micro-express-out} e \cdot y  = -\partial_{t}f_{2} + l + g.
\een

Define column vector $z = (z_i)_{i=1}^{13}:=(\langle -\partial_{t}f_{2}+l+g, e_i\rangle_{v})_{i=1}^{13}$. Taking
 inner product between \eqref{macro-micro-express-out} and the column vector $e$ in the space $L^{2}(\mathbb{R}^{3}_{v})$,  one has $Ay = z$.
For simplicity, we define the following column vectors,
\beno \tilde{f}= (\tilde{f}^{(0)}, \{\tilde{f}^{(1)}_{i}\}_{1\leq i \leq 3}, \{\tilde{f}^{(2)}_{i}\}_{1\leq i \leq 3}, \{\tilde{f}^{(2)}_{ij}\}_{1\leq i < j \leq 3}, \{\tilde{f}^{(3)}_{i}\}_{1\leq i \leq 3})^{T} := A^{-1} (\langle f_{2}, e_{i}\rangle_{v})_{i=1}^{13},
 \\
\tilde{l}= (l^{(0)}, \{l^{(1)}_{i}\}_{1\leq i \leq 3}, \{l^{(2)}_{i}\}_{1\leq i \leq 3}, \{l^{(2)}_{ij}\}_{1\leq i < j \leq 3}, \{l^{(3)}_{i}\}_{1\leq i \leq 3})^{T} := A^{-1} (\langle l, e_{i}\rangle_{v})_{i=1}^{13},
 \\ \tilde{g}=(g^{(0)}, \{g^{(1)}_{i}\}_{1\leq i \leq 3}, \{g^{(2)}_{i}\}_{1\leq i \leq 3}, \{g^{(2)}_{ij}\}_{1\leq i < j \leq 3}, \{g^{(3)}_{i}\}_{1\leq i \leq 3})^{T} := A^{-1} (\langle g, e_{i}\rangle_{v})_{i=1}^{13}. \eeno
Then the equation $Ay = z$ is equivalent to
\ben \label{linear-equation-abc-3} y  = A^{-1} z = -\partial_{t}\tilde{f} + \tilde{l} + \tilde{g}.\een

Following the notations in \cite{duan}, let us define the
 temporal  energy functional  $\mathcal{I}_{N}(f)$ as
\ben \label{interactive-INf} \mathcal{I}_{N}(f) := \sum_{|\alpha|\leq N-1}\sum_{i=1}^{3}( \mathcal{I}^{a}_{\alpha,i}(f)+\mathcal{I}^{b}_{\alpha,i}(f)+\mathcal{I}^{c}_{\alpha,i}(f)+\mathcal{I}^{ab}_{\alpha,i}(f)), \een
where
 \beno \mathcal{I}^{a}_{\alpha,i}(f) := \langle \partial^{\alpha} \tilde{f}^{(1)}_{i}, \partial_{i}\partial^{\alpha} a\rangle_{x}, \mathcal{I}^{c}_{\alpha,i}(f):= \langle \partial^{\alpha} \tilde{f}^{(3)}_{i}, \partial_{i}\partial^{\alpha} c\rangle_{x}, \mathcal{I}^{ab}_{\alpha,i}(f):= \langle \partial_{i}\partial^{\alpha} a, \partial^{\alpha} b_{i}\rangle_{x}
  \\ \mathcal{I}^{b}_{\alpha,i}(f) := -\sum_{j \neq i}\langle \partial^{\alpha} \tilde{f}^{(2)}_{j}, \partial_{i}\partial^{\alpha} b_{i}\rangle_{x} + \sum_{j \neq i}\langle \partial^{\alpha} \tilde{f}^{(2)}_{ji}, \partial_{j}\partial^{\alpha} b_{i}\rangle_{x}  + 2 \langle \partial^{\alpha} \tilde{f}^{(2)}_{i}, \partial_{i}\partial^{\alpha} b_{i}\rangle_{x}.
 \eeno
We are in a position to state a lemma to capture the dissipation of $(a,b,c)$.
\begin{lem}\label{estimate-for-highorder-abc} Let $N \geq 1$. Recall $e = \{e_{j}\}_{j=1}^{13}$.
There exists a constant $C > 0$ such that
\ben \label{solution-property-part2} \frac{d}{dt}\mathcal{I}_{N}(f) + \frac{1}{2}|\nabla_{x}(a,b,c)|^{2}_{H^{N-1}_{x}} \leq C(\|f_{2}\|^{2}_{H^{N}_{x}L^2_{\epsilon,\gamma/2}} + \sum_{|\alpha|\leq N-1}\int_{\mathbb{T}^{3}}|\langle  \pa^{\alpha}g, e\rangle_{v}|^{2} dx).\een
\end{lem}
The proof of lemma \ref{estimate-for-highorder-abc} will be given in the Appendix. Now we are able to prove Proposition \ref{essential-estimate-of-micro-macro}.

\begin{proof}[Proof of Proposition \ref{essential-estimate-of-micro-macro}.] Applying $\partial^{\alpha}$ to equation \eqref{lBE}, taking inner product with $\partial^{\alpha}f$, we have
\beno \frac{1}{2} \frac{d}{dt} \|\partial^{\alpha}f\|^{2}_{L^{2}} + (\mathcal{L}^{\epsilon}\partial^{\alpha}f, \partial^{\alpha}f) =  (\partial^{\alpha}g, \partial^{\alpha}f).\eeno
Thanks to Proposition \ref{coercvityforLep}, for some constant $c_{0}>0$,
we have
\ben \label{solution-property-part-g}\frac{1}{2}\frac{d}{dt}\|f\|^{2}_{H^{N}_{x}L^{2}} + c_0\|f_{2}\|^{2}_{H^{N}_{x}L^2_{\epsilon,\gamma/2}} \lesssim \sum_{|\alpha| \leq N}
|(\pa^{\alpha}g, \pa^{\alpha}f)|.\een
Then Proposition \ref{essential-estimate-of-micro-macro} follows by making a suitable combination of \eqref{solution-property-part-g} and Lemma \ref{estimate-for-highorder-abc}. More precisely, one can multiply \eqref{solution-property-part-g} by a large constant and adding the resultant to \eqref{solution-property-part2}.
\end{proof}

\subsubsection{ A priori  estimate in $H^{N}_{x}L^{2}$.}
In this subsection, we derive the {\it a priori}  estimate in $H^{N}_{x}L^{2}$ for solutions to the Cauchy problem \eqref{linearizedBE}. We apply Proposition \ref{essential-estimate-of-micro-macro} by taking $g = \Gamma^{\epsilon}(f,f)$. For ease of notation, let us define the energy and the dissipation functionals as
\beno \mathcal{E}_{N}(f) := \|f \|^{2}_{H^{N}_{x}L^{2}}, \quad \mathcal{D}_{N}(f) := |(a,b,c)|^{2}_{H^{N}_{x}} + \|f_{2}\|^{2}_{H^{N}_{x}L^2_{\epsilon,\gamma/2}}.\eeno
The {\it a priori} result can be concluded as follows:
\begin{thm}\label{a-priori-estimate-LBE}
Let $-\frac{3}{2}< \gamma<0, N \geq 2$. There exists $\delta_{0}>0$ independent of $\epsilon$ such that if a solution
 $f^{\epsilon}$  to the Cauchy problem \eqref{linearizedBE}  satisfies $\sup_{0 \leq t \leq T}  \mathcal{E}_{2}(f^{\epsilon}(t))\le \delta_{0}$ for some $0< T \leq  \infty$,
then
\beno \sup_{t\in[0,T]}\mathcal{E}_{N}(f^{\epsilon}(t)) +  \int_{0}^{T}\mathcal{D}_{N}(f^{\epsilon}(s))ds \leq C(\mathcal{E}_{N}(f_{0})),\eeno	where $C(\cdot)$ is a continuous increasing function verifying $C(0)=0$. When $N=2$,
$C(x) \lesssim x.$
\end{thm}

\begin{proof}  Thanks to \eqref{conserveq} and \eqref{Nuspace}, we can apply Poincar\'{e} inequality to $(a,b,c)$ to obtain  $|(a,b,c)|_{H^N_x}\sim |\na_x(a,b,c)|_{H^{N-1}_x}$. By Proposition  \ref{essential-estimate-of-micro-macro}, we need to estimate $
|(\pa^{\alpha}\Gamma^{\epsilon}(f,f), \pa^{\alpha}f)|$ and $\int_{\mathbb{T}^{3}}|\langle  \pa^{\alpha}\Gamma^{\epsilon}(f,f), e \rangle_{v}|^{2} dx$ for $|\alpha| \leq N$.

In this sequel, we denote the Fourier transform of $f$ with respect to $x$ variable by $\hat{f}$. Observe
\beno (\Gamma^\epsilon(g, h), f)=\sum_{k,m\in\Z^3} \langle \Gamma^\epsilon (\hat{g}(k), \hat{h}(m-k)), \hat{f}(m)\rangle_v. \eeno
From which together with Theorem \ref{upGammagh}, we get
\beno |(\Gamma^\epsilon(\pa_x^\alpha g, \pa_x^\beta h), f)|\lesssim \sum_{k,m\in\Z^3} |k|^{|\alpha|}|m-k|^{|\beta|}|\hat{g}(k)|_{L^2}|\hat{h}(m-k)|_{L^2_{\epsilon,\gamma/2}}|\hat{f}(m)|_{L^2_{\epsilon,\gamma/2}}.
\eeno
From which we derive that  for $a, b\ge 0$ with $a+b>\f32$,
\ben\label{HNGamma1}
|(  \Gamma^\epsilon(\pa_x^\alpha g, \pa_x^\beta h), f)|\lesssim \|g\|_{H^{|\alpha|+a}_xL^2}\|h\|_{H^{|\beta|+b}_xL^2_{\epsilon,\gamma/2}}\|f\|_{L^2_{\epsilon,\gamma/2}}.\een
As a result, for $|\alpha|\leq N$,
\ben\label{HNGamma}  |(\pa^\alpha \Gamma^\epsilon(g, h), f)|\lesssim \|g\|_{H^2_xL^2}\|h\|_{H^N_xL^2_{\epsilon,\gamma/2}}\|f\|_{L^2_{\epsilon,\gamma/2}}+
\mathrm{1}_{N\ge3}\|g\|_{H^N_xL^2}\|h\|_{H^{N-1}_xL^2_{\epsilon,\gamma/2}}\|f\|_{L^2_{\epsilon,\gamma/2}}. \een
Observing that $\|f\|_{H^N_xL^2_{\epsilon,\gamma/2}}\lesssim |(a,b,c)|_{H^N_x}+\|f_2\|_{H^N_xL^2_{\epsilon,\gamma/2}}$, we obtain that
\ben\label{solution-property-part2-2} \sum_{|\alpha| \leq N}|(\partial^{\alpha}\Gamma^{\epsilon}(f,f), \partial^{\alpha}f)| \lesssim  \mathcal{E}_{2}^{\frac{1}{2}}(f)\mathcal{D}_{N}(f)+\mathrm{1}_{N\ge3}\mathcal{E}_{N}^{\frac{1}{2}}(f)\mathcal{D}_{N-1}^{\frac{1}{2}}(f)\mathcal{D}_{N}^{\frac{1}{2}}(f).\een
Thanks to Theorem \ref{upGammagh}, estimate \eqref{HNGamma1}, similar to \eqref{HNGamma} and \eqref{solution-property-part2-2}, we have for $|\alpha|\leq N$,
\ben \label{solution-property-part1}\int_{\mathbb{T}^{3}}|\langle  \pa^{\alpha}\Gamma^{\epsilon}(f,f), e \rangle_{v}|^{2} dx \lesssim  \mathcal{E}_{2}(f)\mathcal{D}_{N}(f)+\mathrm{1}_{N\ge3}\mathcal{E}_{N}(f)\mathcal{D}_{N-1}(f).\een

Let $\mathcal{E}^M_{N}(f^\epsilon):= M\mathcal{E}_{N}(f^\epsilon)+\mathcal{I}_{N}(f^\epsilon)\sim \mathcal{E}_{N}(f^\epsilon)$. Then by  Proposition \ref{essential-estimate-of-micro-macro}, \eqref{solution-property-part2-2} and \eqref{solution-property-part1}, for some universal constant $c_{0}>0$,
 we arrive at
\ben\label{ENERG1}  \frac{d}{dt}\mathcal{E}^M_{N}(f^\epsilon) + c_{0}\mathcal{D}_{N}(f^\epsilon) \lesssim C(\mathcal{E}_{2}^{\frac{1}{2}}(f^\epsilon)+\mathcal{E}_{2}(f^\epsilon))\mathcal{D}_{N}(f^\epsilon)
+\mathrm{1}_{N\ge3}\mathcal{E}_{N}(f^\epsilon)\mathcal{D}_{N-1}(f^\epsilon).  \een

For $N=2$, if $\delta_{0}$ is sufficiently small,
under the condition $\sup_{0 \leq t \leq T}  \mathcal{E}_{2}(f^{\epsilon}(t))\le \delta_{0}$, we have
\beno 
\frac{d}{dt}\mathcal{E}^M_{2}(f^\epsilon) + \frac{c_{0}}{2}\mathcal{D}_{2}(f^\epsilon) \le0, \quad \sup_{t\in[0,T]}\mathcal{E}_{2}(f^\epsilon(t))+\int_0^{T} \mathcal{D}_2(f^\epsilon(s))ds\lesssim \mathcal{E}_{2}(f_0). \eeno

For $N\ge3$, using the smallness assumption, \eqref{ENERG1} gives
\beno  \frac{d}{dt}\mathcal{E}^M_{N}(f^\epsilon)+ \frac{c_{0}}{2} \mathcal{D}_{N}(f^\epsilon)\lesssim   \mathcal{E}_{N}(f^\epsilon)\mathcal{D}_{N-1}(f^\epsilon).  \eeno
Then we can get the desired result by using mathematical induction.
 \end{proof}

\subsubsection{Propagation of the weighted Sobolev regularity $H^{N}_{x}L^{2}_{l}$.} We aim to prove:
\begin{prop}\label{a-priori-estimate-LBE-HNL2l}
Let $-\frac{3}{2}< \gamma<0, l \geq 2, N \geq 2$. There exists $\delta_{0}>0$ independent of $\epsilon$ such that if a solution
 $f^{\epsilon}$  to the Cauchy problem \eqref{linearizedBE}  satisfies $\sup_{0 \leq t \leq T}  \mathcal{E}_{2}(f^{\epsilon}(t))\le \delta_{0}$ for some $0< T \leq  \infty$,
then
  \beno \sup_{t\in [0,T]}\|f^{\epsilon}(t)\|^{2}_{H^{N}_{x}L^{2}_{l}} + \int_{0}^{T}\|f^{\epsilon}(s)\|^{2}_{H^{N}_{x}L^{2}_{\epsilon,l+\gamma/2}}ds \leq C(\|f_{0}\|^{2}_{H^{N}_{x}L^{2}_{l}}),
 \eeno
where $C(\cdot)$ is a continuous increasing function verifying $C(0)=0$. When $N=2$,
$C(x) \lesssim x.$
\end{prop}
\begin{proof}
We omit the superscript $\epsilon$ in $f^{\epsilon}$ to write
\ben \label{linearized-Boltzamann} \partial_{t}f + v\cdot \nabla_{x} f + \mathcal{L}^{\epsilon}f= \Gamma^{\epsilon}(f,f).\een
Applying $W_{l}\pa^{\alpha}$ to both sides of \eqref{linearized-Boltzamann}, we have
\beno  \partial_{t}W_{l}\pa^{\alpha}f + v\cdot \nabla_{x} W_{l}\pa^{\alpha}f  + W_{l}\mathcal{L}^{\epsilon}\pa^{\alpha}f =  W_{l}\pa^{\alpha}\Gamma^{\epsilon}(f,f). \eeno
Taking inner product with $W_{l} \pa^{\alpha}f$ and taking sum over $|\alpha| \leq N$, we get
\beno&& \frac{1}{2}\frac{d}{dt}\|f\|^{2}_{H^{N}_{x}L^{2}_{l}}  + \sum_{|\alpha| \leq N} (W_{l}\mathcal{L}^{\epsilon}\pa^{\alpha}f,W_{l} \pa^{\alpha}f) = \sum_{|\alpha| \leq N}( W_{l}\pa^{\alpha}\Gamma^{\epsilon}(f,f),W_{l} \pa^{\alpha}f). \nonumber \eeno
By Theorem \ref{main1}, Lemma \ref{commutatorgamma} and the condition $\gamma/2+l\ge0$, for some constant $c_{0}>0$,
 we have
\beno\sum_{|\alpha| \leq N}  (W_{l}\mathcal{L}^{\epsilon}\pa^{\alpha}f,W_{l} \pa^{\alpha}f) \geq c_{0} \mathcal{D}_N(W_lf)- C \|f\|^{2}_{H^{N}_{x}L^{2}_{l+\gamma/2}}.\eeno
 Observe that
\beno\sum_{|\alpha| \leq N}(W_{l}\pa^{\alpha}\Gamma^{\epsilon}(f,f),W_{l} \pa^{\alpha}f) &=&  \sum_{|\alpha| \leq N}(W_{l}\pa^{\alpha}\Gamma^{\epsilon}(f,f)-\pa^{\alpha}\Gamma^{\epsilon}(f,W_{l}f),W_{l} \pa^{\alpha}f)
 \nonumber \\&&+\sum_{|\alpha| \leq N}(\pa^{\alpha}\Gamma^{\epsilon}(f,W_{l}f),W_{l} \pa^{\alpha}f).\eeno
With the help of the proof of \eqref{HNGamma}, Theorem \ref{upGammagh} and Lemma \ref{commutatorgamma} imply that
\beno  \sum_{|\alpha| \leq N}|(W_{l}\pa^{\alpha}\Gamma^{\epsilon}(f,f),W_{l} \pa^{\alpha}f)|\lesssim \mathcal{E}^{\frac{1}{2}}_{2}(f)\mathcal{D}_N(W_l f)+ \mathrm{1}_{N\ge3}\mathcal{E}^{\frac{1}{2}}_{N}(f)\mathcal{D}_{N-1}^{\frac{1}{2}}(W_l f)\mathcal{D}^{\frac{1}{2}}_N(W_l f)\eeno

Putting together the above results and using the condition $\sup_{0 \leq t \leq T}  \mathcal{E}_{2}(f^{\epsilon}(t))\le \delta_{0}$ with $\delta_{0}$ small enough, we arrive at
\beno  \frac{d}{dt} \mathcal{E}_N(W_l f)  + \frac{c_{0}}{2}\mathcal{D}_N(W_l f) \lesssim  \|f\|^{2}_{H^{N}_{x}L^{2}_{l+\gamma/2}}+ \mathrm{1}_{N\ge3} \mathcal{E}_N(f) \mathcal{D}_{N-1}(W_l f). \eeno
It is not difficult to check that
\beno
\|f\|_{H^{N}_{x}L^{2}_{l+\gamma/2}} &\le& \|f^l\|_{H^{N}_{x}L^{2}_{l+\gamma/2}}+\|f^h\|_{H^{N}_{x}L^{2}_{l+\gamma/2}}
\\&\le&  \eta \|f^l\|_{H^{N}_{x}L^{2}_{l+\gamma/2+s}}+C_\eta \|f^l\|_{H^{N}_{x}L^{2}_{\gamma/2+s}} +\epsilon^{s}\|\epsilon ^{-s}f^h\|_{H^{N}_{x}L^{2}_{l+\gamma/2}}.
\eeno
Taking $\eta$ small enough, when $\epsilon$ is small, we derive
\beno  \frac{d}{dt} \mathcal{E}_N(W_l f)  + \frac{c_{0}}{4}\mathcal{D}_N(W_l f) \lesssim \mathcal{D}_N(f)+ \mathrm{1}_{N\ge3} \mathcal{E}_N(f) \mathcal{D}_{N-1}(W_l f). \eeno
From which, for $N=2$ the desired result is easily obtained thanks to Theorem \ref{a-priori-estimate-LBE}. For $N\ge 3$, we use mathematical induction to get the desired result.
\end{proof}

\subsubsection{Propagation of full regularity} We first give a useful lemma.
\begin{lem}\label{interWsd}  Let $l_2 \geq l_1 \geq 0, m \geq 0, l \in \mathbb{R}$. For any $\eta>0$, there is a constant $C_{\eta}$ such that
\beno |f|_{H^m_l}^2\lesssim (\eta+\epsilon^{2s})|W^\epsilon(D)f|_{H^m_l}^2+C_{\eta}|f|_{L^2_l}^2, \quad |f|_{L^2_{\epsilon,l_1}}\lesssim |f|_{L^2_{\epsilon,l_2}}. \eeno
\end{lem}
\begin{proof} By interpolation inequality, it is easy to check that
\beno |f|_{H^m_l}^2\lesssim |f^\phi|_{H^m_l}^2+|f_\phi|_{H^m_l}^2\lesssim |f^\phi|_{H^m_l}^2+\eta |f_\phi|_{H^{m+s}_l}^{2}+C_\eta|f_\phi|_{L^2_l}^2.\eeno
Then the first result follows Lemma \ref{func}. The second result follows from the definition of $|\cdot|_{L^2_{\epsilon, l}}$.
\end{proof}

We aim to prove:
\begin{prop}\label{a-priori-estimate-LBE-HNlL2q}
Suppose $-\frac{3}{2}< \gamma<0, N \geq 2$. Recalling the weight functions \eqref{AsuWf}, the functionals \eqref{pure-order-m-j}, \eqref{total-order-k}, \eqref{full-order-k}.
There exists $\delta_{0}>0$ independent of $\epsilon$ such that if a solution
 $f^{\epsilon}$  to the Cauchy problem \eqref{linearizedBE}  satisfies $\sup_{0 \leq t \leq T}  \mathcal{E}_{2}(f^{\epsilon}(t))\le \delta_{0}$ for some $0< T \leq  \infty$,
then
\beno \sup_{t\in[0,T]}\mathcal{E}^{N,J}(f^{\epsilon}(t))+\int_{0}^{T}\mathcal{D}^{N,J}(f^{\epsilon}(\tau)) d\tau \leq C(\mathcal{E}^{N,J}(f_0)), \eeno
where $C(\cdot)$ is a continuous increasing function verifying $C(0)=0$.
\end{prop}	
\begin{proof} Since we have the control of $\dot{\mathcal{E}}^{m,0}(f)$ for $0 \leq m \leq N$ by Proposition \ref{a-priori-estimate-LBE-HNL2l}, we will focus on the estimate of $\dot{\mathcal{E}}^{k-j,j}(f)$ with   $1\leq k \leq N, 1\le j\le k$.
We denote
\beno \Gamma^{\epsilon}(g,h;\beta)(v):=
\int B^{\epsilon}(v-v_*,\sigma)(\pa_{\beta}\mu^{\frac{1}{2}})_{*}(g'_*h'-g_*h)d\sigma dv_*.
\eeno
With this notation, one has
\beno 
\pa^{\alpha}_{\beta}\Gamma^{\epsilon}(g,h) = \sum _{\beta_{0}+\beta_{1}+\beta_{2}= \beta,\alpha_{1}+\alpha_{2}=\alpha} C^{\beta_{0},\beta_{1},\beta_{2}}_{\beta} C^{\alpha_{1},\alpha_{2}}_{\alpha} \Gamma^{\epsilon}(\pa^{\alpha_{1}}_{\beta_{1}}g,\pa^{\alpha_{2}}_{\beta_{2}}h;\beta_{0}).\eeno
It is easy to check that for any fixed $\beta$, $\Gamma^{\epsilon}(g,h;\beta)$ shares the same upper bound and commutator estimates as those for $\Gamma^{\epsilon}(g,h)$.
Recalling
$ \mathcal{L}^{\epsilon}g = -\Gamma^{\epsilon}(\mu^{\frac{1}{2}},g) - \Gamma^{\epsilon}(g, \mu^{\frac{1}{2}}). $
Thus
\ben \label{alpha-beta-Lep} &&\pa^{\alpha}_{\beta}\mathcal{L}^{\epsilon}g \\&=& \mathcal{L}^{\epsilon}\pa^{\alpha}_{\beta}g
-\sum_{\beta_{0}+\beta_{1}+\beta_{2}= \beta, \beta_{2} < \beta} C^{\beta_{0},\beta_{1},\beta_{2}}_{\beta}
 [\Gamma^{\epsilon}(\pa_{\beta_{1}}\mu^{\frac{1}{2}}, \pa^{\alpha}_{\beta_{2}}g;\beta_{0}) + \Gamma^{\epsilon}(\pa^{\alpha}_{\beta_{2}}g, \pa_{\beta_{1}}\mu^{\frac{1}{2}};\beta_{0})]. \nonumber\een
 Let $1 \leq k \leq N, 1 \leq  j \leq k$.
 Taking two indexes $\alpha$ and $\beta$ such that $|\alpha|= k-j, |\beta|= j, \beta=(\beta^{1},\beta^{2},\beta^{3})$,  applying $W_{q}\pa^{\alpha}_{\beta}$ to both sides of \eqref{linearized-Boltzamann}, we obtain
\beno \partial_{t}W_{q}\pa^{\alpha}_{\beta}f + v\cdot \nabla_{x} W_{q}\pa^{\alpha}_{\beta}f +
\sum_{j=1}^{3} W_{q}\beta^{j}\pa^{\alpha+e_{j}}_{\beta-e_{j}}f + W_{q}\pa^{\alpha}_{\beta}\mathcal{L}^{\epsilon}f = W_{q}\pa^{\alpha}_{\beta}\Gamma^{\epsilon}(f,f). \eeno
Here $e_{1}=(1,0,0), e_{2}=(0,1,0), e_{3}=(0,0,1).$
Let $W_q=W_{k-j,j}$. Taking inner product with $W_{q}\pa^{\alpha}_{\beta} f$, one has
\beno \frac{1}{2}\frac{d}{dt}\|\pa^{\alpha}_{\beta}f \|^{2}_{L^{2}_{q}}  +
\sum_{j=1}^{3} \beta^{j} (W_{q}\pa^{\alpha+e_{j}}_{\beta-e_{j}}f,W_{q}\pa^{\alpha}_{\beta}f) + (W_{q}\pa^{\alpha}_{\beta}\mathcal{L}^{\epsilon}f,W_{q}\pa^{\alpha}_{\beta}f) = (W_{q}\pa^{\alpha}_{\beta}\Gamma^{\epsilon}(f,f),W_{q}\pa^{\alpha}_{\beta}f).  \eeno
Let us give the estimates term by term.

\underline{(i). The estimate of $(W_{q}\pa^{\alpha+e_{j}}_{\beta-e_{j}}f,W_{q}\pa^{\alpha}_{\beta}f) $.} It is not difficult to check that
\beno |(W_{q}\pa^{\alpha+e_{j}}_{\beta-e_{j}}f,W_{q}\pa^{\alpha}_{\beta}f) |\lesssim \|W_{q}W_{-\gamma/2}\pa^{\alpha+e_{j}}_{\beta-e_{j}}f\|_{L^2}\|W_{q}W_{\gamma/2}\pa^{\alpha}_{\beta}f\|_{L^2}
\lesssim \eta \dot{\mathcal{D}}^{k-j,j}(f)+C_\eta\dot{\mathcal{D}}^{k-j+1,j-1}(f), \eeno
where we use \eqref{AsuWf}.

\underline{(ii). The estimate of $(W_{q}\pa^{\alpha}_{\beta}\mathcal{L}^{\epsilon}f,W_{q}\pa^{\alpha}_{\beta}f) $.} Thanks to \eqref{alpha-beta-Lep}, Theorem \ref{main1}, Theorem \ref{upGammagh} and Lemma \ref{commutatorgamma}, for some universal constant $c_{0}>0$,
 we have
\beno (W_{q}\pa^{\alpha}_{\beta}\mathcal{L}^{\epsilon}f,W_{q}\pa^{\alpha}_{\beta}f) \geq c_{0}\|W_{q}\pa^{\alpha}_{\beta}f\|^{2}_{\epsilon,\gamma/2} - C\|W_{q}\pa^{\alpha}_{\beta}f\|^{2}_{L^2_{\gamma/2}} - C \|f\|^{2}_{H^{k-j}_{x}H^{j-1}_{\epsilon,q+\gamma/2}}.
\eeno
Due to Lemma \ref{interWsd} and our assumption for $W_{m,j}$ in \eqref{AsuWf}, the above inequality can be rewritten as follows
\beno (W_{q}\pa^{\alpha}_{\beta}\mathcal{L}^{\epsilon}f,W_{q}\pa^{\alpha}_{\beta}f) \geq c_{0}\|W_{q}\pa^{\alpha}_{\beta}f\|^{2}_{\epsilon,\gamma/2} - (\eta + \epsilon^{2s}) \dot{\mathcal{D}}^{k-j,j}(f)-C_\eta \dot{\mathcal{D}}^{k-j,0}(f)- C\mathcal{D}^{k-1}(f).\eeno

\underline{(iii). The estimate of $(W_{q}\pa^{\alpha}_{\beta}\Gamma^{\epsilon}(f,f),W_{q}\pa^{\alpha}_{\beta}f) $.}
It is easy to check that
\beno
(W_{q}\pa^{\alpha}_{\beta}\Gamma^{\epsilon}(f,f),W_{q}\pa^{\alpha}_{\beta}f) =(W_{q}\Gamma^{\epsilon}(f,\pa^{\alpha}_{\beta}f),W_{q}\pa^{\alpha}_{\beta}f)\\+\sum_{\beta_{0}+\beta_{1}+\beta_{2}= \beta, \alpha_1+\alpha_2=\alpha,  |\alpha_2|+|\beta_{2}|\le k-1} C^{\beta_{0},\beta_{1},\beta_{2}}_{\beta}
C^{\alpha_{1},\alpha_{2}}_{\alpha}(W_{q}
\Gamma^{\epsilon}(\pa^{\alpha_1}_{\beta_1}f,\pa^{\alpha_2}_{\beta_2}f;\beta_0),W_{q}\pa^{\alpha}_{\beta}f)
\eeno
By Theorem \ref{upGammagh} and Lemma \ref{commutatorgamma}, for $a,b \geq 0, a+b =2$,
 we have
\beno
|(W_{q}\Gamma^{\epsilon}(g,h),W_{q}f)| \lesssim \|g\|_{H^{a}_{x}L^{2}} \|h\|_{H^{b}_{x}L^{2}_{\epsilon, q+\gamma/2}} \|f\|_{L^{2}_{\epsilon,q+\gamma/2}},
\eeno
which gives for any $0<\eta<1$,
\beno
|(W_{q}\Gamma^{\epsilon}(f,\pa^{\alpha}_{\beta}f),W_{q}\pa^{\alpha}_{\beta}f)| \lesssim \mathcal{E}^{\frac{1}{2}}_2(f)\dot{\mathcal{D}}^{k-j,j}(f) \lesssim (\eta + \eta^{-1}\mathcal{E}_2(f)) \dot{\mathcal{D}}^{k-j,j}(f).
\eeno
It remains to consider $A:= (W_{q}\Gamma^{\epsilon}(\pa^{\alpha_1}_{\beta_1}f,\pa^{\alpha_2}_{\beta_2}f;\beta_0),W_{q}\pa^{\alpha}_{\beta}f)$ where $|\alpha_2|+|\beta_{2}|\le k-1$. We will give the estimate for $A$ case by case.

{\it Case  1: $k=1$.} There are only two situations $(|\alpha_1|,|\beta_1|)=(0,0)$ or $(0,1)$. Then we have
\beno |A| &\lesssim& (\|\pa_{\beta}f\|_{L^2}+\|f\|_{L^2})  \|f\|_{H^{2}_{x}L^{2}_{\epsilon, q+\gamma/2}} \|\pa_{\beta}f\|_{L^2_{\epsilon,q+\gamma/2}}
\\&\lesssim& \eta^{-1}(\dot{\mathcal{E}}^{0,1}(f)+1)\mathcal{D}_2(W_{1,0}f)+ (\eta+\eta^{-1}\mathcal{E}_2(f))\dot{\mathcal{D}}^{0,1}(f).\eeno

{\it Case 2: $k=2$.} We divide the estimate into two cases:
  $|\alpha_2|+|\beta_2|=1$ and $|\alpha_2|+|\beta_2|=0$.

In the case of $|\alpha_2|+|\beta_2|=1$, we have $(|\alpha_2|, |\beta_2|)=(1,0)$ or $(|\alpha_2|, |\beta_2|)=(0,1)$. If $(|\alpha_2|, |\beta_2|)=(1,0)$, we get that $j=1$ and $(|\alpha_1|, |\beta_1|)=(0,1)$ or $(0,0)$. Then we have
\beno |A|\lesssim \eta^{-1} \mathcal{E}_2(f)\|W_q  f\|_{H^1_xL^2_{\epsilon,\gamma/2}}^2+ \eta^{-1} \|f\|_{H^1_x\dot{H}^1_v}^2\|W_q  f\|_{H^2_xL^2_{\epsilon,\gamma/2}}^2+\eta\|\pa^\alpha_\beta f\|_{L^2_{\epsilon,q+\gamma/2}}^2. \eeno
If $(|\alpha_2|, |\beta_2|)=(0,1)$, then we have $(|\alpha_1|, |\beta_1|)=(2-j,j-1)$ or $(2-j,j-2)$ if $j\ge2$. These imply that
\beno |A|\lesssim  \eta^{-1}(\mathcal{E}_2(f) + \dot{\mathcal{E}}^{2-j+1,j-1}(f))
\|W_q  f\|_{H^1_x\dot{H}^1_{\epsilon,\gamma/2}}^2 +\eta\|\pa^\alpha_\beta f\|_{L^2_{\epsilon,q+\gamma/2}}^2. \eeno

 In the case of $|\alpha_2|+|\beta_2|=0$, we deduce that
 $(|\alpha_1|, |\beta_1|)=(2-j,j)$ or $(2-j,j-1)$ or $(2-j,j-2)$ if $j\ge2$. Then we arrive at
 \beno  |A|\lesssim \eta^{-1}(\|f\|_{\dot{H}^{2-j}_x\dot{H}^j}^2+\mathcal{E}^1(f))\mathcal{D}_2(W_{q}f) +
\eta\|\pa^\alpha_\beta f\|_{L^2_{\epsilon,q+\gamma/2}}^2.\eeno

{\it Case 3: $k\ge3$.} We consider four subcases.

{\quad \it Case 3.1: $|\alpha_2|+|\beta_2|=k-1$.} Either $(|\alpha_2|,|\beta_2|)=(k-j-1,j)$ or $(k-j,j-1)$, and we have
\beno |A|&\lesssim& \eta^{-1}\mathcal{E}_2(f)(\|W_qf\|_{\dot{H}^{k-j}_x\dot{H}^j_{\epsilon,\gamma/2}}^2
+\|W_qf\|_{\dot{H}^{k-j-1}\dot{H}^j_{\epsilon,\gamma/2}}^2)
\\&&+\eta^{-1}\|f\|^2_{H^2_xH^1}\|W_qf\|^2_{H^{k-j}_xH^{j-1}_{\epsilon,\gamma/2}}+ \eta\|\pa^\alpha_\beta f\|_{L^2_{\epsilon,q+\gamma/2}}^2\\&\lesssim& (\eta^{-1}\mathcal{E}_2(f)+\eta) \dot{\mathcal{D}}^{k-j,j}(f)+ \eta^{-1}\mathcal{D}^{k-1}(f)(\dot{\mathcal{E}}^{2,1}(f)+\mathcal{E}^2(f))
.\eeno

{\quad\it Case 3.2: $|\alpha_2|+|\beta_2|= k-2$ and $|\beta_2|=j$.} We first have $j\le k-2$. It is easy to check that
$(|\alpha_2|,|\beta_2|)=(k-j-2,j)$ or $|\alpha_2|\le k-j-3$ if $k\ge4$. We get that
\beno |A|\lesssim \eta^{-1}(\dot{\mathcal{E}}^{3,0}(f)+\mathcal{E}^2(f))\mathcal{D}^{k-1}(f)+\mathrm{1}_{k\ge 4}\eta^{-1}\mathcal{E}^{k-j}(f)\mathcal{D}^{k-1}(f)+\eta \dot{\mathcal{D}}^{k-j,j}(f).  \eeno

{\quad\it Case 3.3: $|\alpha_2|+|\beta_2|= k-2$ and $|\beta_2|\le j-1$.} Observing $|\alpha_1|+|\beta_1|\le 2$ and $|\beta_0|+|\beta_1|\ge 1$, we have
\beno A|\lesssim \eta^{-1}(\dot{\mathcal{E}}^{3,0}(f)+\dot{\mathcal{E}}^{2,1}(f)+\dot{\mathcal{E}}^{1,2}(f)\mathrm{1}_{j\ge2}+
\mathcal{E}^2(f))\mathcal{D}^{k-1}(f) +\eta \dot{\mathcal{D}}^{k-j,j}(f).  \eeno

{\quad\it Case 3.4: $|\alpha_2|+|\beta_2|\le k-3$.}
It is not difficult to see that
\beno |A|\lesssim \eta^{-1}(\dot{\mathcal{E}}^{k-j,j}(f)+\mathcal{E}^{k-1}(f))\mathcal{D}^{k-1}(f)+\eta \dot{\mathcal{D}}^{k-j,j}(f).\eeno
Now we patch together the above estimates to derive that
\begin{enumerate}
\item if $k=1$,    $|(W_{q}\pa^{\alpha}_{\beta}\Gamma^{\epsilon}(f,f),W_{q}\pa^{\alpha}_{\beta}f)| \lesssim
\eta^{-1}(\dot{\mathcal{E}}^{0,1}(f)+1)\mathcal{D}_2(W_{1,0}f)+ (\eta+\eta^{-1}\mathcal{E}_2(f))\dot{\mathcal{D}}^{0,1}(f)$;
\item if $k=2$, $|(W_{q}\pa^{\alpha}_{\beta}\Gamma^{\epsilon}(f,f),W_{q}\pa^{\alpha}_{\beta}f)|\lesssim
    (\eta+\eta^{-1}\mathcal{E}_2(f)) \dot{\mathcal{D}}^{2-j,j}(f) + \eta^{-1}(\dot{\mathcal{E}}^{2-j,j}(f)+\dot{\mathcal{E}}^{1,1}(f)
 +\mathcal{E}^{1}(f))\mathcal{D}_2(W_{2,0}f) +
 \eta^{-1}(\mathcal{E}_2(f)+\mathcal{E}^{2-j+1,j-1}(f))( \dot{\mathcal{D}}^{1,1}(f)+\mathcal{D}^1(f))
 $;
\item if $k\geq3$, $|(W_{q}\pa^{\alpha}_{\beta}\Gamma^{\epsilon}(f,f),W_{q}\pa^{\alpha}_{\beta}f)|\lesssim
(\eta^{-1}\mathcal{E}_2(f)+\eta) \dot{\mathcal{D}}^{k-j,j}(f)+ \eta^{-1} \mathcal{D}^{k-1}(f)(\dot{\mathcal{E}}^{2,1}(f)+\dot{\mathcal{E}}^{3,0}(f)+\dot{\mathcal{E}}^{1,2}(f)\mathrm{1}_{j\ge2}
+\mathcal{E}^{k-1}(f)+\dot{\mathcal{E}}^{k-j,j}(f)).$
\end{enumerate}

To get the estimate of $\mathcal{E}^1(f)$, we only need to bound $\dot{\mathcal{E}}^{0,1}$. From the above estimates, we have
\beno \f{d}{dt} \dot{\mathcal{E}}^{0,1}(f)+\f12c_{0} \dot{\mathcal{D}}^{0,1}(f)\lesssim C_{\eta} (\dot{\mathcal{E}}^{0,1}(f)+1)\mathcal{D}_2(W_{1,0}f)+(\eta+\eta^{-1}\mathcal{E}_2(f))\dot{\mathcal{D}}^{0,1}(f).\eeno
Taking $\eta$ small enough and since $\sup_{0 \leq t \leq T}  \mathcal{E}_{2}(f(t))\le \delta_{0}$ with $\delta_{0}$ small enough, by Proposition \ref{a-priori-estimate-LBE-HNL2l} and
Gr\"{o}nwall inequality, we conclude that
 \beno \sup_{t\in[0,T]}\mathcal{E}^1(f(t))+\int_0^{T} \mathcal{D}^1(f(\tau))d\tau
 \leq C(\|f_0\|_{H^{2}_{x}L^{2}_{l_{1,0}}}, \mathcal{E}^{1}(f_0)). \eeno

To prove the propagation of $\mathcal{E}^2(f)$, we need to consider the energy $\dot{\mathcal{E}}^{2-j,j}$ with $j=1,2$. It is not difficult to conclude from the above estimates that
\beno
 \f{d}{dt} \dot{\mathcal{E}}^{1,1}(f)+\f12c_{0} \dot{\mathcal{D}}^{1,1}(f)\lesssim   \mathcal{D}^1(f)+\mathcal{D}_2(W_{2,0}f)+\dot{\mathcal{E}}^{1,1}\mathcal{D}_2(W_{2,0}f)+\dot{\mathcal{D}}^{2,0}(f),\eeno
which gives
\beno \sup_{t\in[0,T]}\dot{\mathcal{E}}^{1,1}(f(t))+\int_0^{T} \dot{\mathcal{D}}^{1,1}(f(\tau))d\tau\leq C(\mathcal{E}^{2,1}(f_0)). \eeno

Next we have \beno
 \f{d}{dt} \dot{\mathcal{E}}^{0,2}(f)+\f12c_{0} \dot{\mathcal{D}}^{0,2}(f)\lesssim  ( 1+\dot{\mathcal{E}}^{0,2})\mathcal{D}^1(f)+\dot{\mathcal{D}}^{1,1}(f)+\mathcal{D}^1(f),
 \eeno which yields
\beno \sup_{t\in[0,T]}\dot{\mathcal{E}}^{2,0}(f(t))+\int_0^{T} \dot{\mathcal{D}}^{2,0}(f(\tau))d\tau\leq C(  \mathcal{E}^{2,2}(f_0)). \eeno
In other words,  for $0 \leq J\le2$, we have
$ \sup_{t\in[0,T]}\mathcal{E}^{2,J}(f(t))+\int_0^{T} \mathcal{D}^{2,J}(f(\tau))d\tau\leq C(\mathcal{E}^{2,J}(f_0))$.

 Now we shall use mathematical induction to complete the proof. We assume that the result in the proposition holds for $0\leq J\le N\le n$ with $n\ge2$. For $0\leq J\le N=n+1$, since $J=0$ is handled in Proposition \ref{a-priori-estimate-LBE-HNL2l},
 we begin with the propagation of $\dot{\mathcal{E}}^{n,1}(f)$. From the above inequalities, we have
 \beno \f{d}{dt} \dot{\mathcal{E}}^{n,1}(f)+\f12c_{0} \dot{\mathcal{D}}^{n,1}(f)\lesssim  ( 1+\dot{\mathcal{E}}^{n,1}(f)+\mathcal{E}^{n}(f)+\dot{\mathcal{E}}^{3,0}(f))\mathcal{D}^{n}(f)+\dot{\mathcal{D}}^{n+1,0}(f), \eeno
 which yields that $\sup_{t\in[0,T]} \mathcal{E}^{n+1,1}(f(t))+\int_0^{T}  \mathcal{D}^{n+1,1}(f(\tau))d\tau\leq C(\mathcal{E}^{n+1,1}(f_0))$ thanks to Gr\"{o}nwall inequality. For $j\ge2$, we derive that
 \beno \f{d}{dt} \dot{\mathcal{E}}^{n+1-j,j}(f)+\f12c_{0} \dot{\mathcal{D}}^{n+1-j,j}(f)&\lesssim&  ( 1+\dot{\mathcal{E}}^{n+1-j,j}(f)+\mathcal{E}^{n}(f)+\dot{\mathcal{E}}^{3,0}(f)\\&&+\dot{\mathcal{E}}^{2,1}(f))\mathcal{D}^{n}(f)+\dot{\mathcal{D}}^{n+2-j,j-1}(f). \eeno
Using mathematical induction to index $j$, we get for $2\le j\le J$, \beno \sup_{t\in[0,T]}\dot{\mathcal{E}}^{n+1-j,j}(f(t))+\int_0^{T} \dot{\mathcal{D}}^{n+1-j,j}(f(\tau)))d\tau\leq C(\mathcal{E}^{n+1,J}(f_0)),\eeno
 which completes the inductive argument for $n$. We end the proof of the proposition.
\end{proof}

\begin{proof}[Proof of Theorem \ref{main3} (Part I: Global Well-posedness and propagation of regularity)]
By a standard continuity argument, the global well-posedness in $H^{2}_{x}L^{2}$ follows form   the {\it a priori} estimate in Theorem \ref{a-priori-estimate-LBE} and the local well-posedness result (see \cite{Guo1} for instance). The propagation results \eqref{propagation-h-n-l-2-l} and \eqref{propagation-h-n-h-m-l} follows directly from
Proposition \ref{a-priori-estimate-LBE-HNL2l} and Proposition \ref{a-priori-estimate-LBE-HNlL2q}.
\end{proof}

\subsection{Global dynamics} We now give the proof to the second part of Theorem \ref{main3}.

\begin{proof}[Proof of Theorem \ref{main3}(Part II: Global dynamics)] We first give the proof to  \eqref{localizedenergy}. It is easy to check that $\mathcal{P}_j f^{\epsilon}$ verifies
\beno \pa_t\mathcal{P}_j f^{\epsilon}+v\cdot\na_x \mathcal{P}_j f^{\epsilon}+\mathcal{L^\epsilon}\mathcal{P}_j f^{\epsilon}=[\mathcal{L^\epsilon}, \mathcal{P}_j]f^{\epsilon}+\mathcal{P}_j\Gamma^\epsilon(f^{\epsilon}, f^{\epsilon}).  \eeno
Thanks to Theorem \ref{main1}, Lemma \ref{CommSemi} and \eqref{HNGamma1}, for some $C_{0}>0$,
one has
\beno \f{d}{dt}\|\mathcal{P}_j f^{\epsilon}\|_{L^2}^2 \geq -C_{0}(\|\mathcal{P}_j f^{\epsilon}\|_{L^2_{\epsilon,
\gamma/2}}^2
+ \epsilon^{2s}\|f^{\epsilon}\|_{L^2_{\epsilon,\gamma/2}}^2 +
\epsilon^{2s}\|f^{\epsilon}\|_{H^2_xL^2}^2\|f^{\epsilon}\|_{L^2_{\epsilon,\gamma/2}}^2 +\|f^{\epsilon}\|_{H^2_xL^2}\|\mathcal{P}_j f^{\epsilon}\|_{L^2_{\epsilon,\gamma/2}}^2).\eeno
By Theorem \ref{a-priori-estimate-LBE} for the case $N=2$, we have
\beno \sup_{t\geq 0}\mathcal{E}_{2}(f^{\epsilon}(t)) +  \int_{0}^{\infty}\mathcal{D}_{2}(f^{\epsilon}(s))ds \lesssim \mathcal{E}_{2}(f_{0}) \leq \delta_{0}. \eeno
Recalling that $\|\mathcal{P}_j f^{\epsilon}\|_{L^2_{\epsilon,
\gamma/2}}^2 \lesssim  \epsilon^{-2s}2^{j\gamma}\|\mathcal{P}_j f^{\epsilon}\|_{L^2}^{2} \lesssim  \epsilon^{-2s}2^{j\gamma} \delta_{0}$, we have
\beno  \|\mathcal{P}_j f^{\epsilon}(t)\|_{L^2}^2 \ge \|\mathcal{P}_j f_0\|_{L^2}^2- C\epsilon^{-2s}2^{j\gamma}\delta_0t-C\delta_0\epsilon^{2s}, \eeno
which yields \eqref{localizedenergy}.

We turn to the proof of  \eqref{decay-uniform-formula1} and \eqref{decay-uniform-formula2}. By the interpolation inequality $|f|_{L^2}\lesssim |f|_{L^2_{\gamma/2}}^{\f{p}{p+1}}|f|_{L^2_{-p\gamma/2}}^{\f1{p+1}}$ and the facts $\sup_{t \geq 0} \|f^\epsilon(t)\|_{H^2_xL^2_{l}} \lesssim \|f_0\|_{H^2_xL^2_{l}}$ from
Proposition \ref{a-priori-estimate-LBE-HNL2l} and $\f{d}{dt}\mathcal{E}^M_2(f^\epsilon)+\f{c_{0}}{2}\mathcal{D}_2(f^\epsilon)\le 0$ from Theorem \ref{a-priori-estimate-LBE},
for some universal constants $c_{1},c_{2}>0$,
we obtain that
\beno  \f{d}{dt}\mathcal{E}^M_2(f^\epsilon)+c_1\|f^l\|_{H^2_xL^2}^2+ c_{2}\|f_0\|^{-2/p}_{H^2_xL^2_{-p\gamma/2}}\epsilon^{-2s}\|f^h\|_{H^2_xL^2}^{2(1+\f1{p})}\le 0.\eeno
Let $Y(t)=\mathcal{E}^M_2(f^\epsilon)/\|f_{0}\|_{H^2_xL^2}^2, Y_1(t)=\|f^l(t)\|_{H^2_xL^2}^2/\|f_{0}\|_{H^2_xL^2}^2, Y_2(t)=\|f^h(t)\|_{H^2_xL^2}^2/\|f_{0}\|_{H^2_xL^2}^2$, then
\beno \f{d}{dt}Y(t)+c_1Y_1(t)+\mathcal{C}_2^{-1}Y_2(t)^{1+1/p}\le 0.\eeno
where $\mathcal{C}_2^{-1}=c_2|f_{0}|_{H^2_xL^2}^{2/p}\|f_0\|^{-2/p}_{H^2_xL^2_{-p\gamma/2}}\epsilon^{-2s}$.
By Proposition \ref{propODE} and Remark \ref{not-exactly-equal}, we obtain \eqref{decay-uniform-formula1} and \eqref{decay-uniform-formula2}.
\end{proof}

\subsection{Asymptotic formula for the limit} We want to prove \eqref{error-function-uniform-estimate}.    Let $f^{\epsilon}$ and $f^{0}$ be the solutions to  \eqref{linearizedBE} and \eqref{linearizedNBE} respectively with the same initial data $f_0$. Let $F^{\epsilon}_{R} :=  \epsilon^{2-2s}(f^{\epsilon}-f^{0})$, which solves
\ben \label{error-equation}
\partial_{t}F^{\epsilon}_{R} + v \cdot \nabla_{x} F^{\epsilon}_{R} + \mathcal{L}^{0}F^{\epsilon}_{R}=\epsilon^{2s-2}[(\mathcal{L}^{0}-\mathcal{L}^{\epsilon})
f^{\epsilon}+(\Gamma^{\epsilon}-\Gamma^{0})(f^{\epsilon},f^{0})]
+\Gamma^{\epsilon}(f^{\epsilon},F^{\epsilon}_{R})+\Gamma^{0}(F^{\epsilon}_{R},f^{0}). \een
We first derive the estimate on the operator $\Gamma^{0}-\Gamma^{\epsilon}$.

\begin{lem}\label{estimate-operator-difference} If $\gamma>-3$, there holds
\beno|\langle (\Gamma^{0}-\Gamma^{\epsilon})(g,h), f \rangle_v| \lesssim \epsilon^{2-2s}|g|_{L^{2}}|h|_{H^{2}_{\gamma/2+2}}|f|_{L^{2}_{\gamma/2}}.\eeno
\end{lem}
\begin{proof} By direct calculation, we have
\beno\langle (\Gamma^{0}-\Gamma^{\epsilon})(g,h), f \rangle_v  = \mathcal{A}_{1}+\mathcal{A}_{2}+\mathcal{A}_{3}+\mathcal{A}_{4},\eeno
 where
 \beno
 \mathcal{A}_{1} :=\int (b-b^{\epsilon})(\cos\theta)|v-v_{*}|^{\gamma}  ((\mu^{\frac{1}{2}})_{*}^{\prime}-\mu_{*}^{\frac{1}{2}}) g_{*}h^{\prime}f^{\prime} d\sigma dv_{*} dv,
 \\
 \mathcal{A}_{2} := \int (b-b^{\epsilon})(\cos\theta)|v-v_{*}|^{\gamma}  ((\mu^{\frac{1}{2}})_{*}^{\prime}-\mu_{*}^{\frac{1}{2}}) g_{*}( h -h^{\prime})f^{\prime} d\sigma dv_{*} dv,
 \\
 \mathcal{A}_{3} :=\int (b-b^{\epsilon})(\cos\theta)|v-v_{*}|^{\gamma}  \mu_{*}^{\frac{1}{2}} g_{*}( h -h^{\prime})f^{\prime} d\sigma dv_{*} dv,
 \\
 \mathcal{A}_{4} := \int (b-b^{\epsilon})(\cos\theta)|v-v_{*}|^{\gamma}  \mu_{*}^{\frac{1}{2}} g_{*}(h^{\prime}f^{\prime}- h f) d\sigma dv_{*} dv.
 \eeno

\underline{Estimate of $\mathcal{A}_{1}$.} By change of variables, we have
\beno
\mathcal{A}_{1} =  \int (b-b^{\epsilon})(\cos\theta)|v-v_{*}|^{\gamma}(\mu^{\frac{1}{2}} - (\mu^{\frac{1}{2}})^{\prime}) g^{\prime} h_{*}f_{*} d\sigma dv_{*}dv.
\eeno
By Taylor expansion, one has
\beno
\mu^{\frac{1}{2}} - (\mu^{\frac{1}{2}})^{\prime} = (\nabla \mu^{\frac{1}{2}})(v^{\prime})\cdot(v-v^{\prime}) + \int_{0}^{1} (1-\kappa) [(\nabla^{2} \mu^{\frac{1}{2}})(v(\kappa)):(v-v^{\prime})\otimes(v-v^{\prime})] d\kappa,
\eeno
where $v(\kappa) = v^{\prime} + \kappa(v-v^{\prime})$.
Observe that, for any fixed $v_{*}$, there holds
\beno
\int (b-b^{\epsilon})(\cos\theta)|v-v_{*}|^{\gamma} (\nabla \mu^{\frac{1}{2}})(v^{\prime})\cdot(v-v^{\prime}) g^{\prime} d\sigma dv = 0,
\eeno
which gives
\beno
|\mathcal{A}_{1}| &=& |\int (b-b^{\epsilon})(\cos\theta)|v-v_{*}|^{\gamma}  (1-\kappa) [(\nabla^{2} \mu^{\frac{1}{2}})(v(\kappa)):(v-v^{\prime})\otimes(v-v^{\prime})] g^{\prime} h_{*}f_{*} d \kappa d\sigma dv_{*} dv|
\\&\lesssim&\epsilon^{2-2s}\{\int \langle v_{*}\rangle^{\gamma+4} |g^{\prime}|^{2} |h_{*}|^{2}    dv_{*} dv^{\prime} \}^{\frac{1}{2}}
\{\int |v(\kappa)-v_{*}|^{\gamma} \mu^{\frac{1}{8}}(v(\kappa))|f_{*}|^{2} d \kappa dv_{*} dv(\kappa)  \}^{\frac{1}{2}}
\\&\lesssim& \epsilon^{2-2s}|g|_{L^{2}}|h|_{L^{2}_{\gamma/2+2}}|f|_{L^{2}_{\gamma/2}},
\eeno
where we use the change of variable $v \rightarrow v^{\prime}$ and $v \rightarrow v(\kappa)$, and the estimate
$\int_{0}^{\epsilon} \theta^{1-2s} d \theta \lesssim \epsilon^{2-2s}$.

\underline{Estimate of $\mathcal{A}_{2}$.}
By Cauchy-Schwartz inequality, we have
\beno
|\mathcal{A}_{2}| &\leq& \{\int (b-b^{\epsilon})(\cos\theta)|v-v_{*}|^{\gamma+2}  g^{2}_{*}(h -h^{\prime})^{2} ((\mu^{\frac{1}{4}})_{*}^{\prime}+\mu_{*}^{\frac{1}{4}})^{2}d\sigma dv_{*} dv \}^{\frac{1}{2}}
\\&&\times\{\int (b-b^{\epsilon})(\cos\theta)|v-v_{*}|^{\gamma-2}((\mu^{\frac{1}{4}})_{*}^{\prime}-\mu_{*}^{\frac{1}{4}})^{2}|f^{\prime}|^{2} d\sigma dv_{*} dv  \}^{\frac{1}{2}}
:=  \{\mathcal{A}_{2,1}\}^{\frac{1}{2}} \times \{\mathcal{A}_{2,2}\}^{\frac{1}{2}}.
\eeno
By Taylor expansion,
$
h -h^{\prime} = \int_{0}^{1} (\nabla h)(v(\kappa))\cdot(v-v^{\prime}) d\kappa,
$
where $v(\kappa) = v^{\prime} + \kappa(v-v^{\prime})$. By the change of variable $v\rightarrow v(\kappa)$, we get
\beno
\mathcal{A}_{2,1} \lesssim \epsilon^{2-2s} \int \langle v(\kappa)\rangle^{\gamma+4} g^{2}_{*} |(\nabla h)(v(\kappa))|^{2}  dv_{*} dv(\kappa) d\kappa
\lesssim \epsilon^{2-2s}|g|^{2}_{L^{2}}|h|^{2}_{H^{1}_{\gamma/2+2}}.
\eeno
Note that $((\mu^{\frac{1}{4}})_{*}^{\prime}-\mu_{*}^{\frac{1}{4}})^{2}   \lesssim ((\mu^{\frac{1}{4}})_{*}^{\prime}+\mu_{*}^{\frac{1}{4}}) \sin^{2}\frac{\theta}{2}|v-v_{*}|^{2}$, thus we have
\beno
\mathcal{A}_{2,2} \lesssim \epsilon^{2-2s}\int |v-v_{*}|^{\gamma} \mu_{*}^{\frac{1}{4}} |f|^{2} dv_{*} dv \lesssim \epsilon^{2-2s}|f|^{2}_{L^{2}_{\gamma/2}}.
\eeno
Patching together the estimate for $\mathcal{A}_{2,1}$ and $\mathcal{A}_{2,2}$, we have $|\mathcal{A}_{2}| \lesssim \epsilon^{2-2s}|g|_{L^{2}}|h|_{H^{1}_{\gamma/2+2}}|f|_{L^{2}_{\gamma/2}}$.

\underline{Estimate of $\mathcal{A}_{3}$.}
By Taylor expansion, one has
\beno
h -h^{\prime} = (\nabla h)(v^{\prime})\cdot(v-v^{\prime}) + \int_{0}^{1} (1-\kappa) [(\nabla^{2} h)(v(\kappa)):(v-v^{\prime})\otimes(v-v^{\prime})] d\kappa,
\eeno
where $v(\kappa) = v^{\prime} + \kappa(v-v^{\prime})$.
Observe that, for any fixed $v_{*}$, there holds
\beno
\int (b-b^{\epsilon})(\cos\theta)|v-v_{*}|^{\gamma} (\nabla h)(v^{\prime})\cdot(v-v^{\prime}) f^{\prime} d\sigma dv = 0.
\eeno
Thus we have
\beno
|\mathcal{A}_{3}| &=&  |\int (b-b^{\epsilon})(\cos\theta)|v-v_{*}|^{\gamma} \mu_{*}^{\frac{1}{2}} g_{*} (1-\kappa) [(\nabla^{2} h)(v(\kappa)):(v-v^{\prime})\otimes(v-v^{\prime})] f^{\prime} d \kappa d\sigma dv_{*} dv|
\\&\lesssim&\epsilon^{2-2s}\{\int |v(\kappa)-v_{*}|^{\gamma+4} \mu_{*}^{\frac{1}{2}} g^{2}_{*}|(\nabla^{2} h)(v(\kappa))|^{2} d \kappa dv_{*} dv(\kappa)\}^{\frac{1}{2}}
\\&&\times\{\int |v^{\prime}-v_{*}|^{\gamma} \mu_{*}^{\frac{1}{2}} |f^{\prime}|^{2} dv_{*} dv^{\prime}  \}^{\frac{1}{2}}
\lesssim \epsilon^{2-2s}|g|_{L^{2}}|h|_{H^{2}_{\gamma/2+2}}|f|_{L^{2}_{\gamma/2}}.
\eeno

\underline{Estimate of $\mathcal{A}_{4}$.} By cancellation lemma and Lemma \ref{aftercancellation}, we have
\beno |\mathcal{A}_{4}| \lesssim \epsilon^{2-2s} \int |v-v_{*}|^{\gamma} \mu_{*}^{\frac{1}{2}} |g_{*} h f| dv_{*} dv \lesssim \epsilon^{2-2s} |g|_{L^{2}}|h|_{H^2_{\gamma/2}}|f|_{L^{2}_{\gamma/2}}. \eeno
The lemma then follows by patching together the above estimates.
\end{proof}

We are ready to prove  \eqref{error-function-uniform-estimate}.

\begin{proof}[Proof of Theorem \ref{main3}(Part III: Asymptotic formula)] Recalling \eqref{error-equation}, we
set \beno g=\epsilon^{2s-2}[(\mathcal{L}^{0}-\mathcal{L}^{\epsilon})f^{\epsilon}+(\Gamma^{\epsilon}-\Gamma^{0})(f^{\epsilon},f^{0})]
+\Gamma^{\epsilon}(f^{\epsilon},F^{\epsilon}_{R})+\Gamma^{0}(F^{\epsilon}_{R},f^{0}).\eeno
By applying Proposition  \ref{essential-estimate-of-micro-macro} with the previous nonlinear term $g$, using Poincar\'{e} inequality, for some universal constant $c_{0}>0$,
we have
\beno &&\frac{d}{dt}(M\|F^{\epsilon}_{R}\|^{2}_{H^{N}_{x}L^{2}}+\mathcal{I}_{N}(F^{\epsilon}_{R}))+ c_{0}\|F^{\epsilon}_{R}\|^{2}_{H^{N}_{x}L^{2}_{0,\gamma/2}} \\&\lesssim& \sum_{|\alpha| \leq N }
|(\pa^{\alpha}g, \pa^{\alpha}F^{\epsilon}_{R})|+ \sum_{|\alpha| \leq N-1}\sum_{j=1}^{13}
\int_{\mathbb{T}^{3}}|\langle  \pa^{\alpha}g, e_j\rangle_{v}|^{2} dx. \eeno
Thanks to $|\langle  \Gamma^{\epsilon}(g,h), e_j\rangle_{v}| \lesssim |g|_{L^{2}_{\gamma/2}}|h|_{L^{2}_{\gamma/2}}$, for any $|\alpha|\leq N$, we have
\beno &&\int_{\mathbb{T}^{3}} \big(|\langle  \pa^{\alpha}\Gamma^{\epsilon}(f^{\epsilon},F^{\epsilon}_{R}), e_{j}\rangle_{v}|^{2}+ |\langle  \pa^{\alpha}\Gamma^{0}(F^{\epsilon}_{R},f^{0}), e_{j}\rangle_{v}|^{2}\big) dx \\  &\lesssim& \|f^{\epsilon}\|^{2}_{H^{2}_{x}L^{2} }\|F^{\epsilon}_{R}\|^{2}_{H^{N}_{x}L^{2}_{0,\gamma/2}} +
\mathrm{1}_{N \ge 3}\|f^{\epsilon}\|^{2}_{H^{N}_{x}L^{2}}\|F^{\epsilon}_{R}\|^{2}_{H^{N-1}_{x}L^{2}_{0,\gamma/2}}+
\|f^{0} \|^{2}_{H^{N}_{x}L^{2}_{0,\gamma/2}}\|F^{\epsilon}_{R}\|^{2}_{H^{N}_{x}L^{2}}.\eeno

By Lemma \ref{estimate-operator-difference}, we get that
$ \epsilon^{2s-2}\int_{\mathbb{T}^{3}}|\langle  \pa^{\alpha} (\Gamma^{\epsilon}-\Gamma^{0})(f^{\epsilon},f^{0}), e_{j}\rangle_{v}|^{2} dx \lesssim \|f^{\epsilon}\|^{2}_{H^{N}_{x}L^{2}}\|f^{0}\|^{2}_{H^{N}_{x}H^{2}_{\gamma/2+2}},$
and\\
$ \epsilon^{2s-2}\int_{\mathbb{T}^{3}}|\langle  \pa^{\alpha}(\mathcal{L}^{0}-\mathcal{L}^{\epsilon})f^{\epsilon}, e_{j}\rangle_{v}|^{2} dx \lesssim \|f^{\epsilon}\|^{2}_{H^{N}_{x}H^{2}_{\gamma/2+2}}.$
By Theorem \ref{upGammagh} with $\epsilon=0$ and \eqref{HNGamma}, we have
\beno &&|(\pa^{\alpha}\Gamma^{0}(F^{\epsilon}_{R},f^{0}), \pa^{\alpha}F^{\epsilon}_{R})|+|(\pa^{\alpha}\Gamma^{\epsilon}(f^{\epsilon},F^{\epsilon}_{R}), \pa^{\alpha}F^{\epsilon}_{R})|
\\& \lesssim & (\|F^{\epsilon}_{R}\|_{H^{N}_{x}L^{2}}\|f^{0}\|_{H^{N}_{x}L^{2}_{0,\gamma/2}}
+\|f^{\epsilon}\|_{H^{2}_{x}L^{2}}\|F^{\epsilon}_{R}\|_{H^{N}_{x}L^{2}_{0,\gamma/2}}
\\&&+\mathrm{1}_{N\ge3}\|f^{\epsilon}\|_{H^{N}_{x}L^{2}}\|F^{\epsilon}_{R}\|_{H^{N-1}_{x}L^{2}_{0,\gamma/2}})
\|F^{\epsilon}_{R}\|_{H^{N}_{x}L^{2}_{0,\gamma/2}}.\eeno
By Lemma \ref{estimate-operator-difference}, we have
\beno&&
\epsilon^{2s-2}|(\pa^{\alpha}(\Gamma^{\epsilon}-\Gamma^{0})(f^{\epsilon},f^{0}), \pa^{\alpha}F^{\epsilon}_{R})| +\epsilon^{2s-2}|(\pa^{\alpha}(\mathcal{L}^{0}-\mathcal{L}^{\epsilon})f^{\epsilon}, \pa^{\alpha}F^{\epsilon}_{R})|\\&\lesssim& (\|f^{\epsilon}\|_{H^{N}_{x}L^{2}_{}}  \|f^{0}\|_{H^{N}_{x}H^{2}_{\gamma/2+2}}+\|f^{\epsilon}\|_{H^{N}_{x}H^{2}_{\gamma/2+2}})
\|F^{\epsilon}_{R}\|_{H^{N}_{x}L^{2}_{0,\gamma/2}}.\eeno

Patching together the above results, we arrive at
\beno &&\frac{d}{dt}(M\|F^{\epsilon}_{R}\|^{2}_{H^{N}_{x}L^{2}}+\mathcal{I}_{N}(F^{\epsilon}_{R}))+ \frac{c_{0}}{2}\|F^{\epsilon}_{R}\|^{2}_{H^{N}_{x}L^{2}_{0,\gamma/2}} \\&\lesssim&
\mathrm{1}_{N\ge3}\|f^{\epsilon}\|^{2}_{H^{N}_{x}L^{2}}\|F^{\epsilon}_{R}\|^{2}_{H^{N-1}_{x}L^{2}_{0,\gamma/2}}+
\|f^{0} \|^{2}_{H^{N}_{x}L^{2}_{0,\gamma/2}}\|F^{\epsilon}_{R}\|^{2}_{H^{N}_{x}L^{2}}+\|f^{\epsilon}\|_{H^{N}_{x}L^{2}_{}}^2  \|f^{0}\|_{H^{N}_{x}H^{2}_{\gamma/2+2}}^2\\&&+\|f^{\epsilon}\|_{H^{N}_{x}H^{2}_{\gamma/2+2}}^2.
\eeno
Thanks to Proposition \ref{a-priori-estimate-LBE-HNlL2q}, we derive that
\beno \int_{0}^\infty (\|f^{\epsilon}(\tau)\|_{H^{N}_{x}H^{2}_{\gamma/2+2}}^2+\|f^{0}(\tau)\|_{H^{N}_{x}H^{2}_{\gamma/2+2}}^2+\|f^{0} (\tau)\|^{2}_{H^{N}_{x}L^{2}_{0,\gamma/2}}) d \tau \leq C(\mathcal{E}^{N+2,2}(f_0)),  \eeno
which yields when $N=2$,
\beno \sup_{t\ge0} \|F^{\epsilon}_{R}(t)\|^{2}_{H^{2}_{x}L^{2}}+\int_0^\infty \|F^{\epsilon}_{R}(\tau)\|^{2}_{H^2_{x} L^{2}_{0,\gamma/2}}d \tau \leq C(\mathcal{E}^{4,2}(f_0)).\eeno
From which together with mathematical induction, for $N\geq3$, we arrive at
\beno \sup_{t\ge0} \|F^{\epsilon}_{R}(t)\|^{2}_{H^{N}_{x}L^{2}}+\int_0^\infty \|F^{\epsilon}_{R}(\tau)\|^{2}_{H^{N}_{x}L^{2}_{0,\gamma/2}}d \tau \leq C(\mathcal{E}^{N+2,2}(f_0)),\eeno
which ends the proof to \eqref{error-function-uniform-estimate} and thus completes the proof to Theorem \ref{main3}.
\end{proof}

\section{Appendix}
We first give the definition of the symbol class $S^{m}_{1,0}$.
\begin{defi}\label{psuopde} A smooth function $a(v,\xi)$ is said to a symbol of type $S^{m}_{1,0}$ if   $a(v,\xi)$  verifies for any multi-indices $\alpha$ and $\beta$,
\beno |(\pa^\alpha_\xi\pa^\beta_v a)(v,\xi)|\le C_{\alpha,\beta} \langle \xi\rangle^{m-|\alpha|}, \eeno
where $C_{\alpha,\beta}$ is a constant depending only on   $\alpha$ and $\beta$.
\end{defi}

\begin{lem}(\cite{he2})\label{operatorcommutator1}
Let $l, s, r \in \R, M \in S^{r}_{1,0}$ and $\Phi \in S^{l}_{1,0}$. Then there exists a constant $C$ such that
\beno
|[M(D), \Phi]f|_{H^{s}} \leq C|f|_{H^{r+s-1}_{l-1}}.
\eeno
\end{lem}
As an application of Lemma \ref{operatorcommutator1}, since  $W^{\epsilon}\in S^{s}_{1,0}, 2^{k}\varphi_{k} \in S^{1}_{1,0}$ with $0<s<1$,  we  have
\ben\label{decompostionpacth}
\sum_{k \geq -1}^{\infty}|W^{\epsilon}(D)\varphi_{k}f|^{2}_{L^{2}} &=& \sum_{k \geq -1}^{\infty}2^{-2k}|W^{\epsilon}(D)2^{k}\varphi_{k}f|^{2}_{L^{2}}
\nonumber \\&\lesssim&\sum_{k \geq -1}^{\infty}2^{-2k}(|2^{k}\varphi_{k}W^{\epsilon}(D)f|^{2}_{L^{2}}+|f|^{2}_{H^{s-1}})
\lesssim |W^{\epsilon}(D)f|^{2}_{L^{2}}.
\een

\begin{lem}(\cite{He-Jiang}) \label{func} Let $W_q^\epsilon(v):= \phi(\epsilon v)\langle v\rangle^q+\epsilon^{-q}(1-\phi(\epsilon v))$. Let $l \in \R, m,q\ge0$.  There hold
	\beno |f|_{H^m_l} \sim |f^\phi|_{H^m_l}+|f_\phi|_{H^m_l},\quad
	|W^\epsilon_q(D) W_l f|_{H^m} \sim|W^\epsilon_q(D)f  |_{H^m_l}.
	  \eeno
	  Let $\Phi(v)\in S^l_{1,0}$. Assume that $B^\epsilon(\xi)$ verifies $|B^\epsilon(\xi)|\le W^\epsilon_{q}(\xi)$   and
 	$ |\pa^\alpha B^\epsilon(\xi)|\le W^\epsilon_{(q-|\alpha|)^+}(\xi),$ then
 	\ben\label{func5} |\Phi B^\epsilon(D)f|_{H^m}+| B^\epsilon(D)\Phi f|_{H^m}\lesssim |W^\epsilon_q(D)W_l f|_{H^m}.
 	\een
\end{lem}

\begin{prop}\label{symbol} Let $ A^\epsilon(\xi):= \int b^\epsilon(\f{\xi}{|\xi|}\cdot \sigma)\min\{ |\xi|^2\sin^2(\theta/2),1\} d\sigma$, then
\beno A^\epsilon(\xi)\sim |\xi|^2\mathrm{1}_{|\xi|\le \sqrt{2}}+\mathrm{1}_{|\xi|\ge \sqrt{2}}(W^\epsilon)^{2}(\xi) \lesssim (W^\epsilon)^{2}(\xi).\eeno
\end{prop}
\begin{proof} Recalling \eqref{cutoff-kernel-def}, we first get
$A^\epsilon(\xi)=2\pi\int_0^{\pi/2} \sin\theta b(\cos\theta) (1-\phi)(\sin\f{\theta}2/\epsilon)\min\{|\xi|^2\sin^2(\theta/2),1\} d\theta. $
By the change of variable: $t=\sin(\theta/2)$, we have
\beno A^\epsilon(\xi) &\sim& \int_0^{\f{\sqrt{2}}{2}} t^{-1-2s} (1-\phi)(t/\epsilon)\min\{ |\xi|^2t^2,1\}dt
\\&=& |\xi|^{2s} \int_0^{\sqrt{2}|\xi|/2} t^{-1-2s}(1-\phi)(\epsilon^{-1}t|\xi|^{-1})\min\{ t^2,1\}dt.
\eeno
By the definition of $\phi$, we have
\beno
|\xi|^{2s}\int_{\frac{4}{3}\epsilon|\xi|}^{\sqrt{2}|\xi|/2} t^{-1-2s} \min\{ t^2,1\}dt\lesssim A^\epsilon(\xi)\lesssim |\xi|^{2s}\int_{\frac{3}{4}\epsilon|\xi|}^{\sqrt{2}|\xi|/2} t^{-1-2s} \min\{ t^2,1\} dt.
\eeno
Now we focus on the quantity
$ I(\xi):= |\xi|^{2s}\int_{c\epsilon|\xi|}^{\sqrt{2}|\xi|/2} t^{-1-2s} \min\{ t^2,1\} dt$ for a constant $\frac{3}{4} \leq c \leq \frac{4}{3}$.
\begin{enumerate}
	\item For the case of $|\xi|\le \sqrt{2}$, we have
	$I(\xi)=|\xi|^{2s} \int_{c\epsilon|\xi|}^{\sqrt{2}|\xi|/2} t^{1-2s}  dt\sim (1-s)^{-1}|\xi|^2.$
		\item For the case of $\sqrt{2}<|\xi|\le (c\epsilon)^{-1}$, we have \beno
		I(\xi)&=&|\xi|^{2s} \big(\int_{c\epsilon|\xi|}^{1} t^{1-2s}  dt+ \int_{1}^{\sqrt{2}|\xi|/2} t^{-1-2s}  dt\big)\\
		&\sim& (1-s)^{-1}|\xi|^{2s}(1-(c\epsilon |\xi|)^{2-2s})+|\xi|^{2s}(1-(\sqrt{2}|\xi|^{-1})^{2s}). \eeno
		\item For the case of $|\xi|\ge (c\epsilon)^{-1}$,
		we have
	$ I(\xi)=|\xi|^{2s}   \int_{c
			\epsilon |\xi|}^{\sqrt{2}|\xi|/2} t^{-1-2s}  dt \sim \epsilon^{-2s}. $
\end{enumerate}
The desired result follows from the above estimates.
\end{proof}

\begin{prop} \label{fourier-transform-cross-term} Let $h,f$ be real-valued functions. There holds
	\beno
	\int_{\mathbb{S}^2 \times \R^3 } b(\f{u}{|u|}\cdot \sigma) h(u)(f(u^+)-f(\f{|u|}{|u^+|}u^+)) d\sigma du
	=\int_{\mathbb{S}^2 \times \R^3 } b(\f{\xi}{|\xi|}\cdot \sigma)  (\hat{h}(\xi^+)-\hat{h}(\f{|\xi|}{|\xi^+|}\xi^+))\bar{\hat{f}}(\xi) d\sigma d\xi.
	\eeno
\end{prop}
\begin{proof} Let $F(u):= \int_{\mathbb{S}^2}  b(\f{u}{|u|}\cdot \sigma)  f(\f{|u|}{|u^+|}u^+)d\sigma$.
By Plancherel equality, we have
	\beno \int_{\mathbb{S}^2 \times \mathbb{R}^3 } b(\f{u}{|u|}\cdot \sigma) h(u) f(\f{|u|}{|u^+|}u^+) d\sigma du
	=\int_{\mathbb{R}^3} h(u) F(u) du
	=\int_{\mathbb{R}^3} \hat{h}(\xi) \bar{\hat{F}}(\xi)d\xi.\eeno
Next, we compute the Fourier transform $\hat{F}$ of $F$. By definition, we have
\beno
\hat{F}(\xi)=\int_{\mathbb{R}^3} e^{-iu\cdot \xi} F(u)du
=\f1{(2\pi)^{\frac{3}{2}}}\int_{\mathbb{S}^2 \times \mathbb{R}^3 \times \mathbb{R}^3} e^{-iu\cdot \xi}e^{i\f{|u|}{|u^+|}u^+\cdot\eta} b(\f{u}{|u|}\cdot \sigma)  \hat{f}(\eta)  d\sigma d\eta du.
\eeno
Notice that $\f{|u|}{|u^+|}u^+\cdot\eta= \f12 \big((\f{u}{|u|}\cdot \sigma +1)/2\big)^{-\f12}(u
\cdot \eta+|u\|\eta| \f{\eta}{|\eta|}\cdot\sigma)$
and the fact $\int_{\mathbb{S}^2} b_{1}(\kappa\cdot \sigma) b_{2}(\tau\cdot \sigma)d\sigma=\int_{\mathbb{S}^2} b_{1}(\tau\cdot \sigma) b_{2}(\kappa\cdot \sigma)d\sigma$, one has
\beno
\hat{F}(\xi)
&=&\f1{(2\pi)^{\frac{3}{2}}}\int_{\mathbb{S}^2 \times \mathbb{R}^3 \times \mathbb{R}^3} e^{-iu\cdot \xi}e^{i\f{|\eta|}{|\eta^+|}\eta^+\cdot u} b(\f{u}{|u|}\cdot \sigma)  \hat{f}(\eta)  d\sigma d\eta du\\
&=&\f1{(2\pi)^{\frac{3}{2}}} \int_{\mathbb{S}^2 \times \mathbb{R}^3}    b(\f{\eta}{|\eta|}\cdot \sigma)  \hat{f}(\eta)\delta [\xi=\f{|\eta|}{|\eta^+|}\eta^+]  d\sigma d\eta,
\eeno
which yields
\beno \int_{\mathbb{S}^2 \times \mathbb{R}^3 } b(\f{u}{|u|}\cdot \sigma) h(u) f(\f{|u|}{|u^+|}u^+) d\sigma du
=\int_{\mathbb{S}^2 \times \mathbb{R}^3 } b(\f{\xi}{|\xi|}\cdot \sigma)   \hat{h}(\f{|\xi|}{|\xi^+|}\xi^+)\bar{\hat{f}}(\xi) d\sigma d\xi.\eeno
Similar argument can be applied to the remainder term and then we get the desired result.
\end{proof}

\begin{lem}\label{comWep} Let $\mathcal{F}$ be Fourier transform, then $\mathcal{F}W^\epsilon((-\triangle_{\mathbb{S}^2})^{\frac{1}{2}})=W^\epsilon((-\triangle_{\mathbb{S}^2})^{\frac{1}{2}})\mathcal{F}$.
\end{lem}
\begin{proof} By definition in \eqref{DeltaWe}, if $\xi=\rho \tau$, we  have \beno  \mathcal{F}\big(W^\epsilon((-\triangle_{\mathbb{S}^2})^{\frac{1}{2}})f\big)(\xi)&=&\sum_{l=0}^\infty\sum_{m=-l}^l W^\epsilon((l(l+1))^{\frac{1}{2}})  \mathcal{F}(Y^m_l  f^m_l)(\xi)\\
&=&\sum_{l=0}^\infty\sum_{m=-l}^l  W^\epsilon((l(l+1))^{\frac{1}{2}})  Y_l^m(\tau)W_l^m(\rho),
\eeno where we use the fact $\mathcal{F}(Y^m_l f^m_l)(\xi)=Y_l^m(\tau)W_l^m(\rho)$.
On the other hand, using the same notation, we have
$(\mathcal{F}f)(\xi)=\sum_{l=0}^\infty\sum_{m=-l}^l Y_l^m(\tau)W_l^m(\rho), $ which yields
\beno  W^\epsilon((-\triangle_{\mathbb{S}^2})^{\frac{1}{2}})(\mathcal{F}f)(\xi)=\sum_{l=0}^\infty\sum_{m=-l}^l  W^\epsilon((l(l+1))^{\frac{1}{2}}) Y_l^m(\tau)W_l^m(\rho)= \mathcal{F}\big(W^\epsilon((-\triangle_{\mathbb{S}^2})^{\frac{1}{2}})f\big)(\xi),\eeno and ends the proof of the lemma. \end{proof}
In the rest of this appendix, we aim to prove Lemma \ref{estimate-for-highorder-abc}. For some of the details, \cite{duan} is a good reference.
Note that  \eqref{linear-equation-abc-3} is equivalent to
\ben \nonumber 
\partial_{t} a  = -\partial_{t}\tilde{f}^{(0)} + l^{(0)} + g^{(0)},
\\ \nonumber 
\partial_{t}b_{i}+ \partial_{i} a  = -\partial_{t}\tilde{f}^{(1)}_{i} + l^{(1)}_{i} + g^{(1)}_{i}, ~~~~~~~~ 1\leq i \leq 3,
\\ \label{equation-cb}\partial_{t}c+ \partial_{i} b_{i}  = -\partial_{t}\tilde{f}^{(2)}_{i} + l^{(2)}_{i} + g^{(2)}_{i}, ~~~~~~~~ 1\leq i \leq 3,
\\ \label{equation-b}\partial_{i}b_{j}+ \partial_{j} b_{i}  = -\partial_{t}\tilde{f}^{(2)}_{ij} + l^{(2)}_{ij} + g^{(2)}_{ij},~~~~~~~~1\leq i < j \leq 3,
\\ \nonumber 
\partial_{i}c  = -\partial_{t}\tilde{f}^{(3)}_{i} + l^{(3)}_{i} + g^{(3)}_{i}, ~~~~~~~~ 1\leq i \leq 3.\een
Based on equations \eqref{equation-cb} and \eqref{equation-b}, it is easy to derive:
\begin{prop} 
For $j = 1,2,3$, the macroscopic component $b_{j}$ satisfies
\ben \label{equation-b-itself-2} -\triangle_{x}b_{j}-\partial^{2}_{j}b_{j} &=& \sum_{i\neq j} \partial_{j}[-\partial_{t}\tilde{f}^{(2)}_{i} + l^{(2)}_{i} + g^{(2)}_{i}] - \sum_{i \neq j} \partial_{i}[-\partial_{t}\tilde{f}^{(2)}_{ij} + l^{(2)}_{ij} + g^{(2)}_{ij}] \\&&- 2 \partial_{j}[-\partial_{t}\tilde{f}^{(2)}_{j} + l^{(2)}_{j} + g^{(2)}_{j}]. \nonumber \een
\end{prop}

The functions $\tilde{f}, \tilde{l}, \tilde{g}$  can be controlled as:
\begin{prop} \label{estimate-on-fln-tilde} There holds
\beno  \sum_{|\alpha|\leq N}|\partial^{\alpha}\tilde{f}|^{2}_{L^{2}_{x}} \lesssim \|f_{2}\|^{2}_{H^{N}_{x}L^2_{\epsilon,\gamma/2}}, \sum_{|\alpha|\leq N-1}|\partial^{\alpha}\tilde{l}|^{2}_{L^{2}_{x}} \lesssim \|f_{2}\|^{2}_{H^{N}_{x}L^2_{\epsilon,\gamma/2}},  \\
  \sum_{|\alpha|\leq N-1}|\partial^{\alpha}\tilde{g}|^{2}_{L^{2}_{x}} \lesssim \sum_{|\alpha| \leq N-1}
\int_{\mathbb{T}^{3}}|\langle  \pa^{\alpha}g, e\rangle_{v}|^{2} dx. \eeno
\end{prop}
\begin{proof}
The first one easily follows. The second one is proved by transforming some weight to $\mathcal{L}e_{j}$, and using $|\cdot|_{\epsilon,\gamma/2} \geq |W^{\epsilon}W_{\gamma/2}|_{L^{2}}$. The third is obvious by the fact  $\tilde{g}= A^{-1} \langle g, e\rangle_{v}$.
\end{proof}
The next lemma on the dynamics of (a,b,c) is from macroscopic conservation laws.
\begin{lem}\label{equation-for-ptabc}
The macroscopic components $(a,b,c)$ satisfy the following system of equations:
\beno 
\partial_{t}a - \frac{1}{2} \nabla_{x} \cdot \langle \mu^{\frac{1}{2}}|v|^{2}v, f_{2}\rangle_{v}=\frac{1}{2}\langle  (5-|v|^{2})\mu^{\frac{1}{2}}, g\rangle_{v}.
 \\ 
\partial_{t}b + \nabla_{x}(a+5c) + \nabla_{x} \cdot \langle \mu^{\frac{1}{2}}v \otimes v, f_{2}\rangle_{v}=\langle  v \mu^{\frac{1}{2}}, g\rangle_{v}.
 \\ 
 \partial_{t}c + \frac{1}{3} \nabla_{x} \cdot b + \frac{1}{6}\nabla_{x} \cdot \langle \mu^{\frac{1}{2}}|v|^{2}v, f_{2}\rangle_{v}=\frac{1}{6}\langle  (|v|^{2}-3)\mu^{\frac{1}{2}}, g\rangle_{v}.\eeno
\end{lem}
\begin{proof}
Multiply both sides of the  equation \eqref{lBE} by the collision invariants $\mu^{\frac{1}{2}}\{1, v_{i}, |v|^{2}\}$, and take integration over $\mathbb{R}^{3}$
to get equations for inner products $\langle \mu^{\frac{1}{2}}, f\rangle_v,\langle \mu^{\frac{1}{2}}v_{i}, f\rangle_v, \langle \mu^{\frac{1}{2}}|v|^{2}, f\rangle_v$. Then write out each item in the equations in terms of $(a,b,c)$ as many as possible. Finally, take suitable combinations to get the desired equations.
\end{proof}
The previous lemma yields:
\begin{lem}\label{estimate-for-ptabc}
There holds:
  $$\sum_{|\alpha|\leq N-1}|\partial^{\alpha}\partial_{t}(a,b,c)|^{2}_{L^{2}_{x}} \lesssim \sum_{0<|\alpha|\leq N}\|\mu^{\frac{1}{4}}\partial^{\alpha}f_{2}\|^{2}_{L^{2}}+|\nabla_{x}(a,b,c)|^{2}_{H^{N-1}_{x}}
  +\sum_{|\alpha| \leq N-1} \int_{\mathbb{T}^{3}}|\langle  \pa^{\alpha}g, e\rangle_{v}|^{2} dx.$$
\end{lem}
\begin{proof}
The lemma follows easily from Lemma \ref{equation-for-ptabc}.
\end{proof}
\begin{proof}[Proof of Lemma \ref{estimate-for-highorder-abc}]  For $|\alpha| \leq N-1$,
applying $\pa^{\alpha}$ to equation \eqref{equation-b-itself-2} for $b_{j}$, then by taking inner product with $\pa^{\alpha} b_{j}$,   one has   \beno |\nabla_{x} \pa^{\alpha}b_{j}|^{2}_{L^{2}_{x}} + |\pa_{j} \pa^{\alpha}b_{j}|^{2}_{L^{2}_{x}} &=& \langle \sum_{i\neq j} \partial_{j}\pa^{\alpha}[-\partial_{t}\tilde{f}^{(2)}_{i} + l^{(2)}_{i} + g^{(2)}_{i}], \pa^{\alpha} b_{j}\rangle_{x}   \\&&- \langle \sum_{i \neq j} \partial_{i}\pa^{\alpha}[-\partial_{t}\tilde{f}^{(2)}_{ij} + l^{(2)}_{ij} + g^{(2)}_{ij}] , \pa^{\alpha} b_{j}\rangle_{x}   \\&&- 2 \langle  \partial_{j}\pa^{\alpha}[-\partial_{t}\tilde{f}^{(2)}_{j} + l^{(2)}_{j} + g^{(2)}_{j}] , \pa^{\alpha} b_{j}\rangle_{x}. \eeno
By integration by parts, the time derivative can transferred to $\pa^{\alpha} b_{j}$, one has
\beno |\nabla_{x} \pa^{\alpha}b_{j}|^{2}_{L^{2}_{x}} + |\pa_{j} \pa^{\alpha}b_{j}|^{2}_{L^{2}_{x}} &=& -\frac{d}{dt} \mathcal{I}^{b}_{\alpha,j}(f) +
\langle \sum_{i\neq j} \partial_{j}\pa^{\alpha}\tilde{f}^{(2)}_{i}, \partial_{t}\pa^{\alpha} b_{j}\rangle_{x}
  - \langle \sum_{i\neq j}\partial_{i}\pa^{\alpha}\tilde{f}^{(2)}_{ij}, \partial_{t}\pa^{\alpha} b_{j}\rangle_{x}\\&&- 2 \langle  \partial_{j}\pa^{\alpha}\tilde{f}^{(2)}_{j}, \partial_{t}\pa^{\alpha} b_{j}\rangle_{x}
+\langle \sum_{i\neq j} \partial_{j}\pa^{\alpha}[ l^{(2)}_{i} + g^{(2)}_{i}] , \pa^{\alpha} b_{j}\rangle_{x}
\\&&- \langle \sum_{i\neq j} \partial_{i}\pa^{\alpha}[ l^{(2)}_{ij} +g^{(2)}_{ij}], \pa^{\alpha} b_{j}\rangle_{x}  - 2 \langle  \partial_{j}\pa^{\alpha}[l^{(2)}_{j} + g^{(2)}_{j}] , \pa^{\alpha} b_{j}\rangle_{x}.\eeno
By Cauchy-Schwartz inequality, one has
\beno &&\langle \sum_{i\neq j} \partial_{j}\pa^{\alpha}\tilde{f}^{(2)}_{i}, \partial_{t}\pa^{\alpha} b_{j}\rangle_{x}
 + \langle \sum_{i\neq j}\partial_{i}\pa^{\alpha}\tilde{f}^{(2)}_{ij}, \partial_{t}\pa^{\alpha} b_{j}\rangle_{x}- 2 \langle  \partial_{j}\pa^{\alpha}\tilde{f}^{(2)}_{j}, \partial_{t}\pa^{\alpha} b_{j}\rangle_{x}
 \\&\leq& \eta \sum_{|\alpha|\leq N-1}|\partial^{\alpha}\partial_{t}(a,b,c)|^{2}_{L^{2}_{x}} + \frac{1}{4\eta}\sum_{|\alpha|\leq N}|\partial^{\alpha}\tilde{f}|^{2}_{L^{2}_{x}}. \eeno
Via integrating by parts, by Cauchy-Schwartz inequality, one has
\beno &&\langle \sum_{i\neq j} \partial_{j}\pa^{\alpha}[ l^{(2)}_{i} + g^{(2)}_{i}] , \pa^{\alpha} b_{j}\rangle_{x} -
\langle \sum_{i\neq j} \partial_{i}\pa^{\alpha}[ l^{(2)}_{ij} +g^{(2)}_{ij}], \pa^{\alpha} b_{j}\rangle_{x}  - 2 \langle  \partial_{j}\pa^{\alpha}[l^{(2)}_{j} + g^{(2)}_{j}] , \pa^{\alpha} b_{j}\rangle_{x}
\\&=&-\langle \sum_{i\neq j} \pa^{\alpha}[ l^{(2)}_{i} + g^{(2)}_{i}] , \partial_{j}\pa^{\alpha} b_{j}\rangle_{x} +
\langle \sum_{i\neq j} \pa^{\alpha}[ l^{(2)}_{ij} +g^{(2)}_{ij}], \partial_{i}\pa^{\alpha} b_{j}\rangle_{x}  + 2 \langle  \pa^{\alpha}[l^{(2)}_{j} + g^{(2)}_{j}] , \partial_{j}\pa^{\alpha} b_{j}\rangle_{x}
 \\&\leq& \eta  |\nabla_{x}(a,b,c)|^{2}_{H^{N-1}_{x}} + \frac{1}{\eta}\sum_{|\alpha|\leq N-1}|\partial^{\alpha}\tilde{l}|^{2}_{L^{2}_{x}}+\frac{1}{\eta}\sum_{|\alpha|\leq N-1}|\partial^{\alpha}\tilde{g}|^{2}_{L^{2}_{x}}. \eeno
Taking sum over $1 \leq j \leq 3$, by Proposition \ref{estimate-on-fln-tilde} and Lemma \ref{estimate-for-ptabc}, we get
\beno  |\nabla_{x} \pa^{\alpha}b|^{2}_{L^{2}_{x}}+\frac{d}{dt}\sum_{j=1}^{3}\mathcal{I}^{b}_{\alpha,j}(f)
\le \eta |\nabla_{x}(a,b,c)|^{2}_{H^{N-1}_{x}} + \frac{C}{\eta}(\|f_{2}\|^{2}_{H^{N}_{x}L^2_{\epsilon,\gamma/2}} + \sum_{|\alpha|\leq N-1}\int_{\mathbb{T}^{3}}|\langle  \pa^{\alpha}g, e\rangle_{v}|^{2} dx).\eeno
Similar techniques can be used to deal with $|\nabla_{x}\pa^{\alpha}c|^{2}_{L^{2}_{x}}$ and $|\nabla_{x}\pa^{\alpha}a|^{2}_{L^{2}_{x}}$, and we have
\beno  && |\nabla_{x} \pa^{\alpha}c|^{2}_{L^{2}_{x}}+\frac{d}{dt}\sum_{j=1}^{3}\mathcal{I}^{c}_{\alpha,j}(f)
+|\nabla_{x}\pa^{\alpha}a|^{2}_{L^{2}_{x}}+\frac{d}{dt}
\sum_{j=1}^{3}(\mathcal{I}^{a}_{\alpha,j}(f)+\mathcal{I}^{ab}_{\alpha,j}(f))
\\&\leq& \eta |\nabla_{x}(a,b,c)|^{2}_{H^{N-1}_{x}} + \frac{C}{\eta}(\|f_{2}\|^{2}_{H^{N}_{x}L^2_{\epsilon,\gamma/2}} + \sum_{|\alpha|\leq N-1}\int_{\mathbb{T}^{3}}|\langle  \pa^{\alpha}g, e\rangle_{v}|^{2} dx).\eeno
Patching together the above estimates and taking sum over $|\alpha|\leq N-1$, we have
\beno  \frac{d}{dt}\mathcal{I}_{N}(f) + |\nabla_{x}(a,b,c)|^{2}_{H^{N-1}_{x}}
\le \eta |\nabla_{x}(a,b,c)|^{2}_{H^{N-1}_{x}} + \frac{C}{\eta}(\|f_{2}\|^{2}_{H^{N}_{x}L^2_{\epsilon,\gamma/2}} + \sum_{|\alpha|\leq N-1}\int_{\mathbb{T}^{3}}|\langle  \pa^{\alpha}g, e\rangle_{v}|^{2} dx).\eeno
Taking $\eta=\frac{1}{2}$, the lemma then follows.
\end{proof}

 {\bf Acknowledgments.} Ling-Bing He is supported by NSF of China under the grant 11771236. Yu-Long Zhou is supported by NSF of China under the grant 12001552.

\end{document}